\documentclass[11pt, reqno]{amsart}
\usepackage{a4wide}
\usepackage[english, activeacute]{babel}
\usepackage{pifont}
\usepackage{amsmath,amsthm,amsxtra}
\usepackage{epsfig}
\usepackage{amssymb}

\usepackage{latexsym}
\usepackage{amsfonts}

\usepackage{hyperref}
\usepackage{xparse}
\usepackage{tabularx} 
\usepackage{multicol}
\usepackage{color}
\usepackage{hyperref}
\usepackage{float}
\usepackage{graphicx}
\usepackage{subcaption}
\usepackage{listings}

\newcommand{\R}{\mathbb{R}}
\newcommand{\N}{\mathbb{N}}

\newcommand{\Com}{\mathbb{C}}
\newcommand{\la}{\lambda}

\newcommand{\al}{\alpha}
\newcommand{\bt}{\beta}
\newcommand{\ga}{\gamma}

\newcommand{\spawn}{\operatorname{span}}
\newcommand{\sech}{\operatorname{sech}}

\newcommand{\re}{\operatorname{Re}}
\newcommand{\ima}{\operatorname{Im}}

\newtheorem{thm}{Theorem}[section]

\newtheorem{lem}[thm]{Lemma}
\newtheorem{prop}[thm]{Proposition}

\newtheorem{defn}[thm]{Definition}

\theoremstyle{remark}
\newtheorem{rem}{Remark}[section]

\newcommand{\be}{\begin{equation}}
\newcommand{\ee}{\end{equation}}
\newcommand{\bp}{\begin{proof}}
\newcommand{\ep}{\end{proof}}
\newcommand{\bel}{\begin{equation}\label}
\newcommand{\eeq}{\end{equation}}
\newcommand{\bea}{\begin{eqnarray}}
\newcommand{\eea}{\end{eqnarray}}
\newcommand{\bee}{\begin{eqnarray*}}
\newcommand{\eee}{\end{eqnarray*}}
\newcommand{\ben}{\begin{enumerate}}
\newcommand{\een}{\end{enumerate}}
\newcommand{\nonu}{\nonumber}

\date{\today}

\usepackage[english]{babel}

\title[NLS-like breathers]{Stability and instability of breathers in the $U(1)$ Sasa-Satsuma and Nonlinear Schr\"odinger models}

\author{Miguel A. Alejo}
\thanks{(M.A.A.) Departamento de Matem\'aticas, Universidad de C\'ordoba, Spain,  e-mail: {\tt malejo@uco.es}. 
Partially funded by Product. CNPq grant no. 305205/2016-1, IMUS and VI PPIT-US program ref. I3C}

\author{Luca Fanelli}
\thanks{(L.F.) Departamento de Matem\'aticas, Ikerbasque \& Universidad del Pa\'is Vasco / Euskal Herriko Unibertsitatea, 
e-mail: {\tt luca.fanelli@ehu.es}} 

\author{Claudio Mu\~noz}
\thanks{(C.M.) CNRS and Departamento de Ingenier\'ia Matem\'atica DIM,  Universidad de Chile,
e-mail: {\tt cmunoz@dim.uchile.cl}. Partially funded by CMM Conicyt PIA AFB170001 and Fondecyt 1150202} 

\begin{document}

\begin{abstract}
We consider the Sasa-Satsuma (SS) and Nonlinear Schr\"odinger (NLS) equations posed along the line, in 1+1 dimensions. Both equations are canonical integrable $U(1)$ models, with solitons, multi-solitons and breather solutions \cite{Yang}. For these two equations, we recognize four distinct localized breather modes: the Sasa-Satsuma for SS, and for NLS the Satsuma-Yajima, Kuznetsov-Ma and Peregrine breathers. Very little is known about the stability of these solutions, mainly because of their complex structure, which does not fit into the classical soliton behavior \cite{GSS}. In this paper we find the natural $H^2$ variational characterization for each of them.  This seems to be the first known variational characterization for these solutions; in particular, the first one obtained for the famous Peregrine breather. We also prove that Sasa-Satsuma breathers are $H^2$ nonlinearly stable, improving the linear stability property previously proved by Pelinovsky and Yang \cite{Peli_Yang}. Moreover, in the SS case, we provide an alternative understanding of the SS solution \emph{as a breather}, and not only as an embedded soliton. The method of proof is based in the use of a $H^2$ based Lyapunov functional, in the spirit of  \cite{AM}, extended this time to the vector-valued case. We also provide another rigorous justification of the instability of the remaining three nonlinear modes 
(Satsuma-Yajima, Peregrine and Kuznetsov-Ma), based in the study of their corresponding linear variational structure (as critical points of a suitable Lyapunov functional), and complementing the instability results recently proved e.g. in \cite{Munoz2}. 
\end{abstract}

\maketitle
\numberwithin{equation}{section}

\bigskip

\section{Introduction}

\subsection{Setting} In this paper our main purpose is to deal with the variational stability of complex soliton-like solutions for Schr\"odinger-type,  $U(1)$ invariant models appearing in nonlinear Physics and integrability theory. By $U(1)$ symmetry, we refer to the classical invariance of the equation under the transformation $u\mapsto u e^{i\ga}$, with $\ga\in\R$ and $u$ complex-valued solution. 

\medskip

The first model that we shall consider is the cubic focusing {\bf Nonlinear Schr\"odinger} (NLS) equation posed on the real line
\be\label{NLS0}
iu_t +u_{xx} +|u|^2u=0, \quad u(t,x)\in\Com, \quad (t,x)\in \R^2. 
\ee
For this model, we will assume two boundary value conditions (BC) at infinity: 
\medskip
\ben
\item[(1.1a)]\label{1a} {\bf Zero BC:} $|u(t,x)|\to 0$ as $x\to \pm\infty$, and
\smallskip
\item[(1.1b)] {\bf Nonzero BC}, in the form of an \emph{Stoke wave}: for all $t\in\R$, 
\be\label{Stokes}
|u(t,x)-e^{it}| \to 0 \quad \hbox{as} \quad x\to \pm\infty.
\ee
\een
\medskip
Additionally, we will consider the {\bf Sasa-Satsuma} (SS) equation for a function $q=q(T,X)$ posed on the line \cite{SS}
\be\label{SS0}
\begin{cases}
iq_T +\frac12 q_{XX} +|q|^2q +i \epsilon \Big( q_{XXX} +6|q|^2 q_X +3q (|q|^2)_X \Big) =0,\\
 q=q(T,X)\in \Com,\quad T,X\in \R.
\end{cases}
\ee
Note that in this equation (and after a suitable rescaling) $\epsilon$ is the parameter of bifurcation from (the integrable) cubic NLS \eqref{NLS0}. However, it is important to notice that, unless $\epsilon=0$, \eqref{SS0} represents a third order complex-valued model for the unknown $q$, with important differences with respect to \eqref{NLS0}.

\medskip

Following Sasa and Satsuma \cite{SS}, we have that under the change of variables
\[
\begin{aligned}
u(t,x)=&~{} q(T,X)e^{-i (X-T/(18\epsilon))/(6\epsilon)}, \\
 t=&~{} T, \quad x=X-T/(12\epsilon),
 \end{aligned}
\]
and assuming $\epsilon=1$, equation \eqref{SS0} reads now \cite[p. 114]{Yang}
\begin{equation}\label{SS}
\begin{aligned}
&~{} u_t  +  u_{xxx} +6|u|^2 u_x +3u(|u|^2)_x =0, \\
&~{} u= u(t,x)\in \Com, \quad t,x\in \R.
 \end{aligned}
\end{equation}
In this paper we will focus on this third order, complex-valued, {\bf modified KdV} (mKdV) model. In particular, this equation will retain several properties of the standard, scalar valued mKdV equation.

\medskip

Both equations, \eqref{NLS0} and \eqref{SS0}, are well-known integrable models, see \cite{ZS} and \cite{SS} respectively. NLS describes 
the propagation of pulses in nonlinear media and gravity waves in the ocean \cite{Dauxois}, and was proved integrable by Zakharov and Shabat \cite{ZS}. NLS \eqref{NLS0} with nonzero BC (1.2) is believed to describe the emergence of rogue or freak waves in deep sea \cite{Peregrine}, and also it is a well-known example of the mechanism known as modulational instability \cite{Peregrine,Akhmediev2}. On the other hand, SS was introduced by Sasa and Satsuma \cite{SS} as an integrable model for which the Lax pair is $3\times 3$ matrix valued, and it is closely related to another integrable model, the Hirota equation (see e.g. \cite{Yang} for additional details).

\medskip

Finally, in the case of \eqref{NLS0} with nonzero boundary conditions at infinity, note that the Stokes wave $e^{it}$ is a particular, non localized solution of \eqref{NLS0}.  A complete family of standing waves can be obtained by using the scaling, phase and Galilean invariances of \eqref{NLS0}: 
\be\label{standing_wave}
u_{c,v,\ga}(t,x):= \sqrt{c} \, \exp \Big(  ict + \frac i2 xv  -\frac i4 v^2 t + i\ga \Big).
\ee
This wave is another solution to \eqref{NLS0}, for any scaling $c>0$, velocity $v\in \R$, and phase $\ga \in \R$. However, since all these symmetries represent invariances of the equation, they will not be essential in our proofs, and we will assume in this paper $c=1$, $v=\ga=0$.

\medskip

Consequently, we will seek for solutions in the form of a Stoke wave, which means that we set
\be\label{u_w}
u(t,x) = e^{it}(1+ w(t,x)).
\ee
We will deal with solutions to \eqref{NLS} for which the \emph{modulational instability} phenomenon is present. Indeed, note that $w$ now solves \cite{Munoz2}
\be\label{NLS}
iw_t + w_{xx}   + 2 \re w  + w^2 + 2 |w|^2 + |w|^2 w=0, 
\ee
with initial data in a certain Sobolev space. The associated linearized equation for \eqref{NLS} is just\footnote{This equation is similar to the well-known linear Sch\"odinger $i\partial_t w + \partial_x^2 w=0$, but instead of dealing with the additional term $2 \re w$ only as a perturbative term, we will consider all linear terms as a whole for later purposes (not considered in this paper), in particular, \emph{long time existence and decay issues}, see e.g. \cite{GNT1,GNT2}.}
\be\label{LS}
i\partial_t w + \partial_x^2 w  + 2 \re w =0.
\ee
Written only in terms of $\phi =\re w$, we have the wave-like equation (compare with \cite{Goodman} in the periodic setting)
\be\label{phi_eqn}
\partial_t^2 \phi+\partial_x^4 \phi + 2 \partial_x^2 \phi =0.
\ee
This problem has some instability issues, as a standard frequency analysis shows: 
looking for a formal standing wave $\phi=e^{i(kx-\omega t)}$ solution to \eqref{phi_eqn}, one has
\[
\omega(k) = \pm |k| \sqrt{k^2 -2},
\]
which reveals that for small wave numbers ($|k|<\sqrt{2}$) the linear equation behaves in an ``elliptic'' fashion, and exponentially (in time) growing modes are present from small perturbations of the vacuum solution. A completely similar conclusion is obtained working in the Fourier variable. This singular behavior is not present if now the equation is defocusing, that is \eqref{NLS} with nonlinearity $-|u|^{2} u$.\footnote{Another model corresponds to the Gross-Pitaevskii equation: $i\partial_t u + \partial_x^2 u +u(1-|u|^2)=0$, for which the Stokes wave is modulationally stable.}

\medskip

Summarizing, in this paper we will focus on models \eqref{NLS0} and \eqref{SS} with zero boundary values at infinity, and on the model \eqref{NLS}, which represents \eqref{NLS0} with nonzero boundary conditions, in the form of a Stoke wave \eqref{Stokes}. Additionally, and appealing to physical considerations, we will only consider solutions to these models with \emph{finite energy}, in a sense to be described below.

\medskip

Concerning the well-posedness theory for the three models \eqref{NLS0}-(1.1a), \eqref{SS}, and \eqref{NLS}, we have the following result.

\begin{prop}[Local and global well-posedness for \eqref{NLS0}-(1.1a), \eqref{SS}, and \eqref{NLS}]
The Sasa-Satsuma equation \eqref{SS} is locally well-posed in $H^s$, $s>\frac14$, and globally well-posed if $s\geq 1$. Similarly, NLS with zero background \eqref{NLS0} is globally well-posed for $s\geq 0$, while NLS with nonzero background \eqref{NLS} is locally well-posed in $H^s$, $s>\frac12$. 
\end{prop}

The proof of this result in the case of Sasa-Satsuma \eqref{SS} follows easily from the arguments in Kenig-Ponce-Vega \cite{KPV}, and for \eqref{NLS} it was recently proved in \cite{Munoz2}. The proof of \eqref{NLS0} is standard, and is due to Ginibre and Velo \cite{GV}, Tsutsumi \cite{Tsutsumi} and Cazenave and Weissler \cite{CW}. See Cazenave \cite{Cazenave} for a complete account on the different NLS equations.

\subsection{$U(1)$ invariant Breathers}  In this paper we are interested in variational stability properties associated to particular but not less important exact solutions to \eqref{NLS0}-(1.1a),  \eqref{SS} and \eqref{NLS}, usually referred as \emph{breathers}.

\begin{defn} 
We will say that a particular smooth solution to \eqref{NLS0}-\emph{(1.1a)} or \eqref{NLS0}-\emph{(1.1b)}, 
or \eqref{SS}, is a {\bf breather} if modulo the invariances of the equation, it is periodic in time, but with nontrivial period. 
\end{defn}

This definition leaves outside of our paper standard solitons for \eqref{NLS0}:
\be\label{soliton}
\frac{\sqrt{c}e^{i \left( ct + \frac12 xv -\frac14  v^2 t +\gamma_0 \right)}}{\cosh(\sqrt{c}(x-vt-x_0))}, \quad c>0,~ v,x_0,\ga_0 \in\R,
\ee
which are time periodic solutions of \eqref{NLS0}, thanks to scaling and Galilean transformations, but its time period is trivial (its infimum equals zero). This last soliton is a well-known orbitally stable solution of NLS, see Cazenave-Lions \cite{CL}, Weinstein \cite{Weinstein}, and Grillakis-Shatah-Strauss \cite{GSS}.

\medskip

\medskip

\begin{itemize}
\item[(i)] {\bf The Sasa-Satsuma (SS) breather}. Let $\al,\bt>0$ be arbitrary but fixed parameters. Following \cite[eqns. (38)-(39)]{SS}, and \cite[eqns. (3-250)-(3-252)]{Yang}, an exact breather solution of  Sasa-Satsuma \eqref{SS} is given by the expression
\be\label{R}
B_{SS}(t,x):= Q_\beta(x+ \gamma t + x_2)e^{ i \Theta}, 
\ee
%
where the phase $\Theta$ and the scaled $Q_\beta$ obey
\[
\Theta:=\alpha (x + \delta t +x_1), \quad Q_\bt (x):=\beta Q(\beta x),
\] 
and the speeds $\gamma$ and $\delta$ are given by (compare with \cite{AM} for instance)
\be\label{gamma_delta}
\gamma:= 3\alpha^2-\beta^2, \quad \delta:= \alpha^2-3\beta^2.
\ee
Above, $Q$ is \emph{complex-valued}, exponentially decaying:
\begin{equation}\label{SolQ}
Q(x):=Q_{\eta}(x):= \frac{2(e^{x} +\eta e^{-x})}{e^{2x} +2 +  |\eta|^2 e^{-2x}},
\ee
and
\begin{equation}\label{speeds}
\eta:= \frac{\alpha}{\alpha +i\beta}.
\ee
It is well-known that the real-valued function $|Q|$ is {\bf single humped} when $|\eta|>1/2$ (i.e. $|\alpha|>\frac12\sqrt{\alpha^2+\beta^2}$), and {\bf double humped} when $0<|\eta|\leq 1/2$ (or $|\alpha|\leq \frac12\sqrt{\alpha^2+\beta^2}$), see \cite{Yang,Peli_Yang}. This mixed shape is in strong contrast with the standard NLS soliton \eqref{NLS0} given in \eqref{soliton}, which is only single humped. Moreover, from the formula in \eqref{R}-\eqref{SolQ}-\eqref{speeds}, one can clearly see that an increasingly small NLS soliton \eqref{soliton} is recovered in the limit $\eta \to 1$ (or $\beta\to 0$). See Fig. \ref{fig7} for more details. 

\medskip

Another important observation in the $SS$ breather is the fact that the single humped condition $|\alpha|>\frac12\sqrt{\alpha^2+\beta^2}$ leads to $3\al^2>\beta^2$, which is nothing but having $\ga>0$ (i.e., a SS breather of {\bf negative speed}). Similarly, the double-humped condition $|\alpha|\leq \frac12\sqrt{\alpha^2+\beta^2}$ means that $\ga\leq 0$, that is to say, the SS breather {\bf moves to the right}.

\medskip
The $B_{SS}$ solution is usually referred in the literature (see e.g. \cite{Yang,Peli_Yang} and references therein) as an \emph{embedded soliton}, because it is embedded in the continuous spectrum of the associated linear operator (see Remark \ref{embedded} for more details on this concept). From the techniques exposed in this paper, we will see that $B_{SS}$ fits perfectly the description associated to a {\bf breather solution}, including its stability characterization.

\medskip
The stability of the SS breather has been studied by Pelinovsky and Yang in \cite{Peli_Yang}. It was proved in this work that in the $\eta\to 1$ limit, the SS breather is linearly stable (single humped case). No other regime seems to be rigorously described in the literature, as far as we understand. Also, the nonlinear stability/instability of the SS breather seems a completely open question.

\medskip

\item[(ii)] {\bf The Satsuma-Yajima (SY) breather}. Let $c_1,c_2>0$, and $\ga_\pm:=c_2 \pm c_1$. The NLS equation with zero background \eqref{NLS0} has the standing, exponentially decaying breather \cite{SY}
%
%
\be\label{SY}
B_{SY}(t,x) := \frac{ 2\sqrt{2} \ga_+\ga_- e^{ic_1^2t}(c_1\cosh(c_2x) + c_2e^{i \ga_+\ga_- t} \cosh(c_1x))}
{\ga_-^2\cosh( \ga_+ x) + \ga_+^2\cosh(\ga_- x) + 4c_1c_2\cos( \ga_+\ga_- t)}, 
\ee
as solution which is a perturbation of the zero state, see Fig . \ref{SY_fig}. By invariances of the equation under time-space shifts, it is possible to give a more general form for \eqref{SY} involving shifts $x_1,x_2\in \R$ in the $t$ and $x$ variables, respectively. Note that by choosing $c_1=1$ and $c_2=3$, we recover the {\it original breather discovered by Satsuma-Yajima} \cite{SY}:
\be\label{SY0}
 \frac{4 \sqrt{2} e^{it}(\cosh(3x) +3e^{8it} \cosh x)}{\cosh(4x) + 4\cosh(2x) + 3\cos(8t)}. 
\ee
The SY breather has been observed in nonlinear optics as well as in quantum mechanics, 
and plays a key role in the description of the precise dynamics 
of optical and matter waves in nonlinear and non autonomous dispersive physical systems, driven by 
nonautonomous NLS and Gross-Pitaevskii (GP) models. For instance, two matter wave soliton solutions in a Bose-Einstein condensate 
reduce to the SY breather with a suitable constant selection (see \cite{Serkin} for further details). 
Moreover, in a hydrodynamical context, it has been reported the observation of the SY breather from a precise initial 
condition for exciting the two soliton solution, which gives rise to this SY breather, from the mechanical instruments generating the waves (\cite{CHOGDA}).

\medskip

It is also well-known that SY breathers are unstable \cite{Yang}. Their instability is simply based in the fact that there are explicit 2-solitons solutions 
(see \eqref{SYa1a2} in Appendix \ref{2solitonNLS} for example) arbitrarily close to the SY breather, but with completely different long-time behavior at infinity in time. This instability property is motivated, in terms of inverse scattering data, as the understanding of the 2-soliton and SY breather as objects described by 2-parameter ``complex-valued eigenvalues'', with no restriction at all, see \cite{Yang} for more details. On the contrary, the 2-soliton and mKdV breather are defined by using real-valued and complex-valued eigenvalues respectively, a distinction that avoids arbitrary closeness in any standard metric. 

\begin{figure}[h!] 
    \includegraphics[scale=0.25]{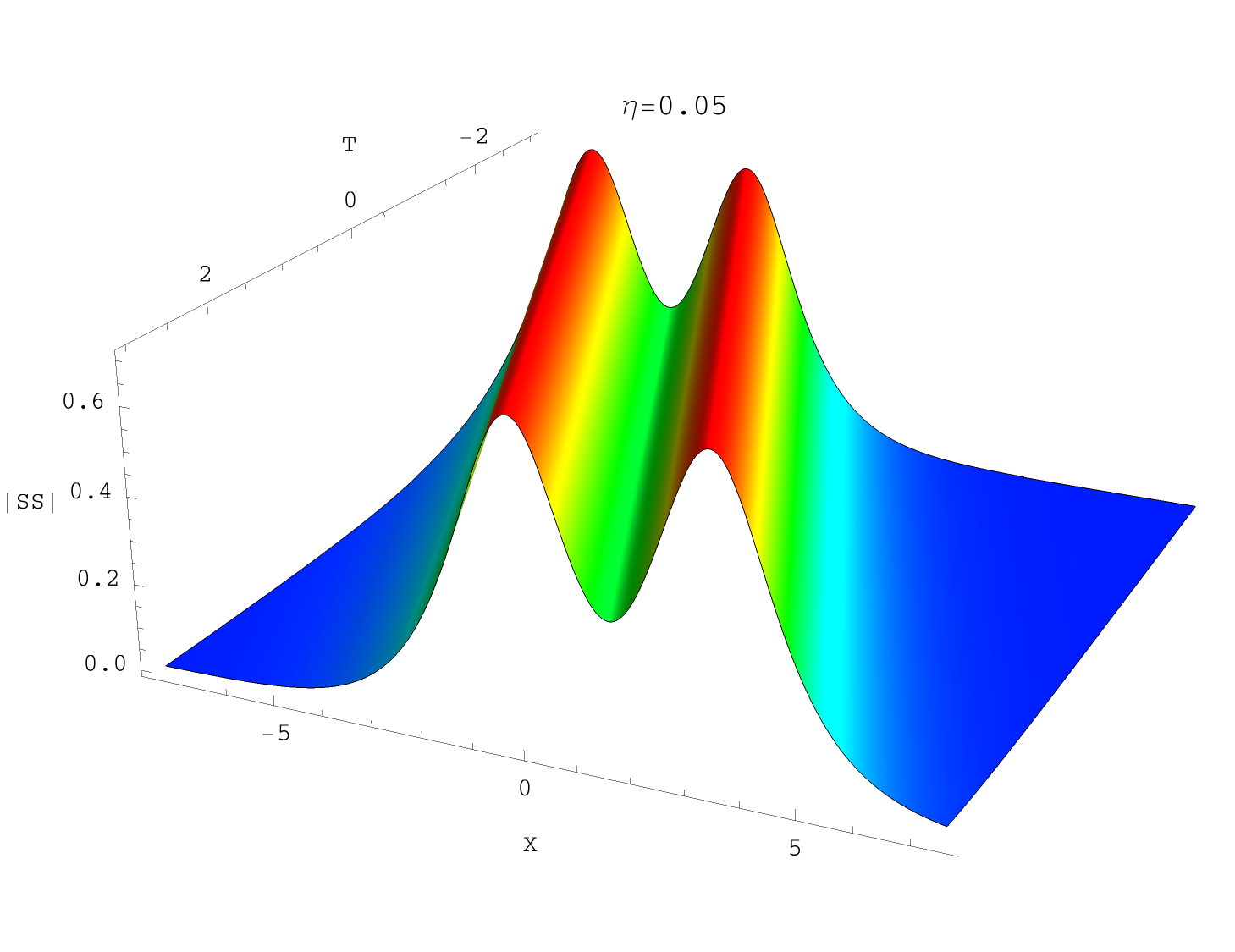}
    \includegraphics[scale=0.25]{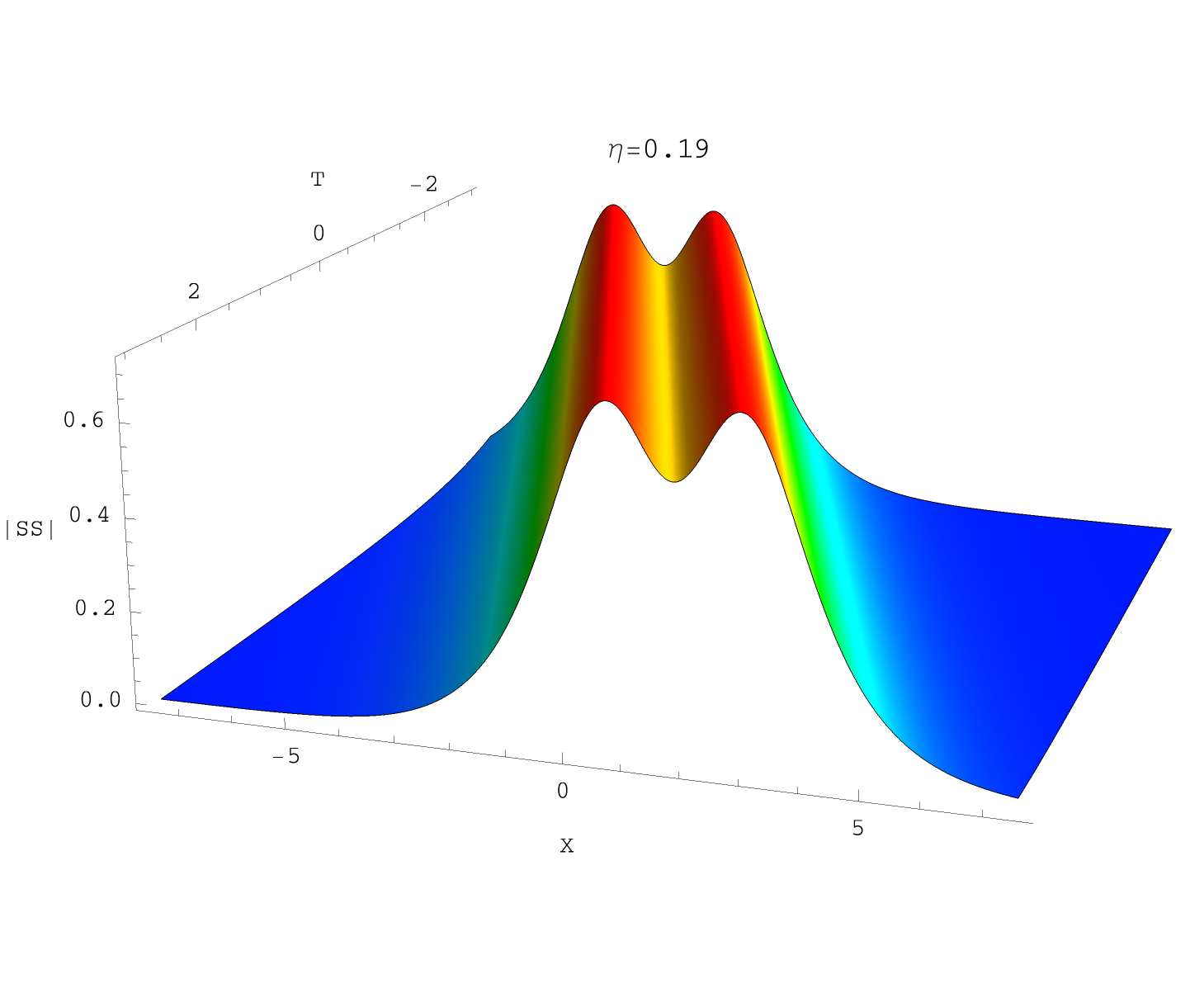}
    
    \includegraphics[scale=0.25]{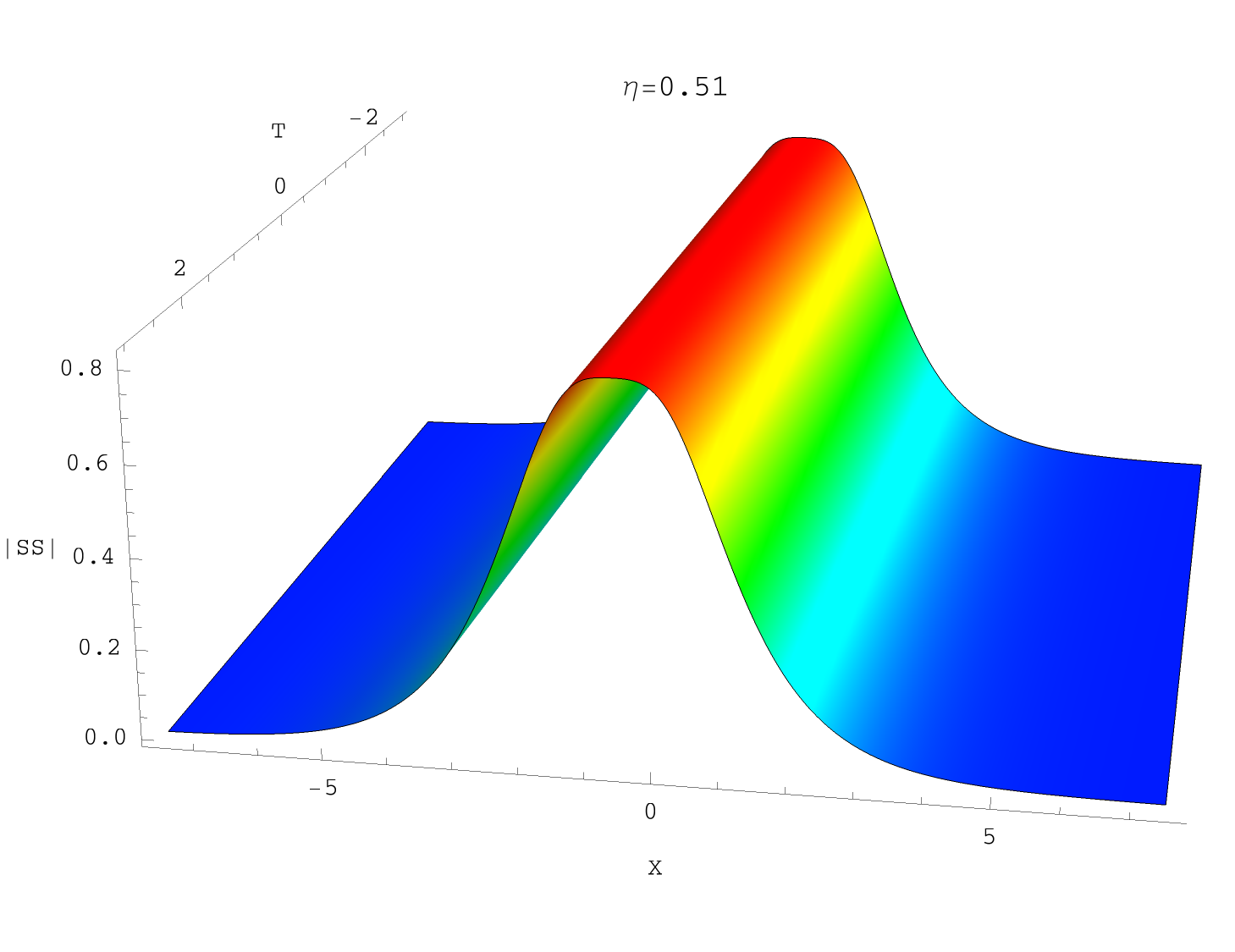}
    \includegraphics[scale=0.25]{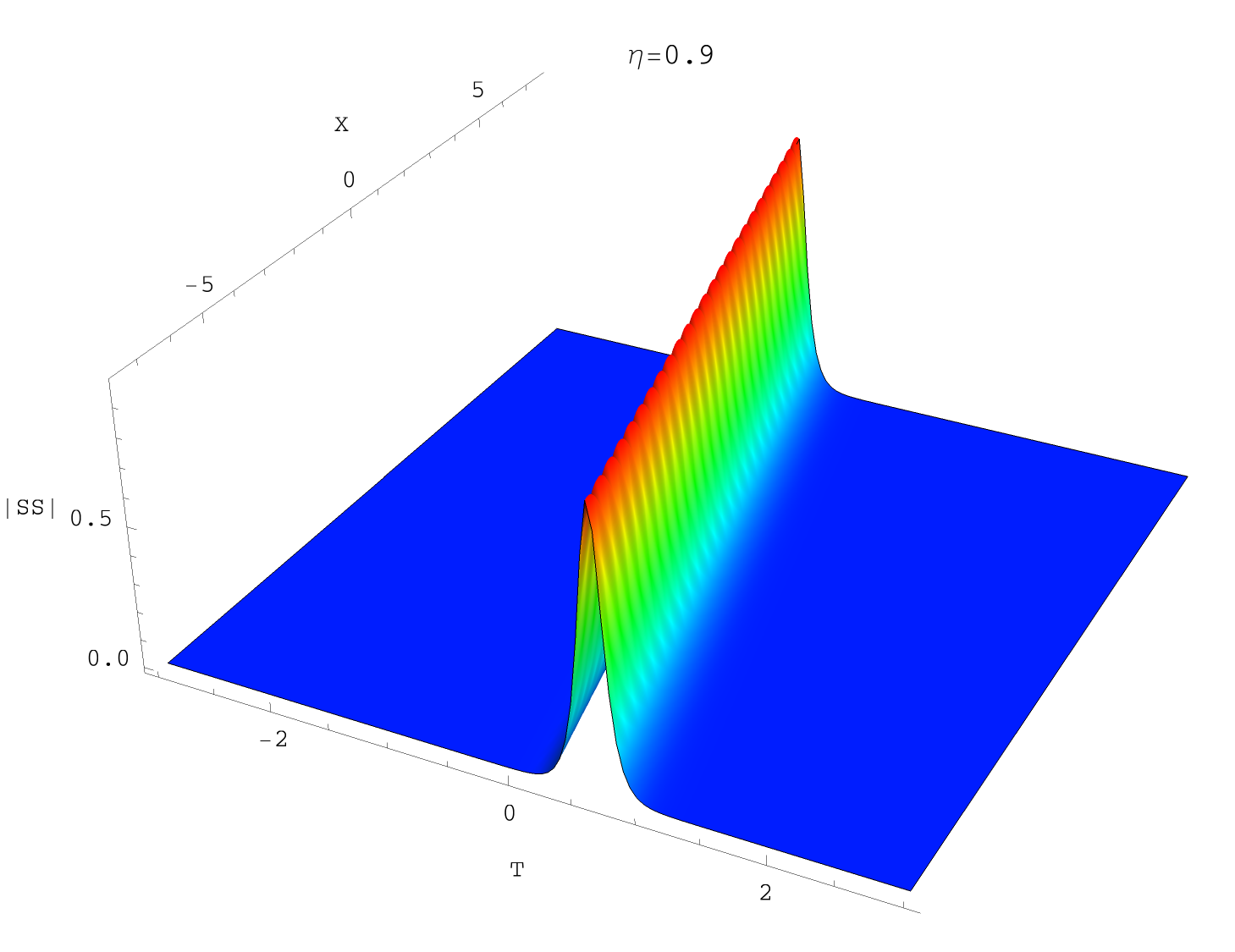}
  \caption{Absolute value of the SS  breather \eqref{SS}, for different values of the parameter $\eta$. \emph{Left above:} $|SS|$ with $\eta=0.05$; \emph{right above}: $|SS|$ with $\eta=0.19$; note that these are cases where the double hump is clearly devised. \emph{Left below:} $|SS|$ with $\eta=0.51$, and \emph{right below:} $|SS|$ with $\eta=0.9$. Note that for $\eta$ close to 1, one recovers the NLS soliton, and for $\eta$ close to zero, the breather decouples and two clearly defined humps, at equal distance for all time (of order $O(|\log \eta|)$), emerge in the dynamics.}
  \label{fig7} 
\end{figure}

\medskip

\begin{figure}[h!] 
              \includegraphics[width=7.0cm,height=4cm]{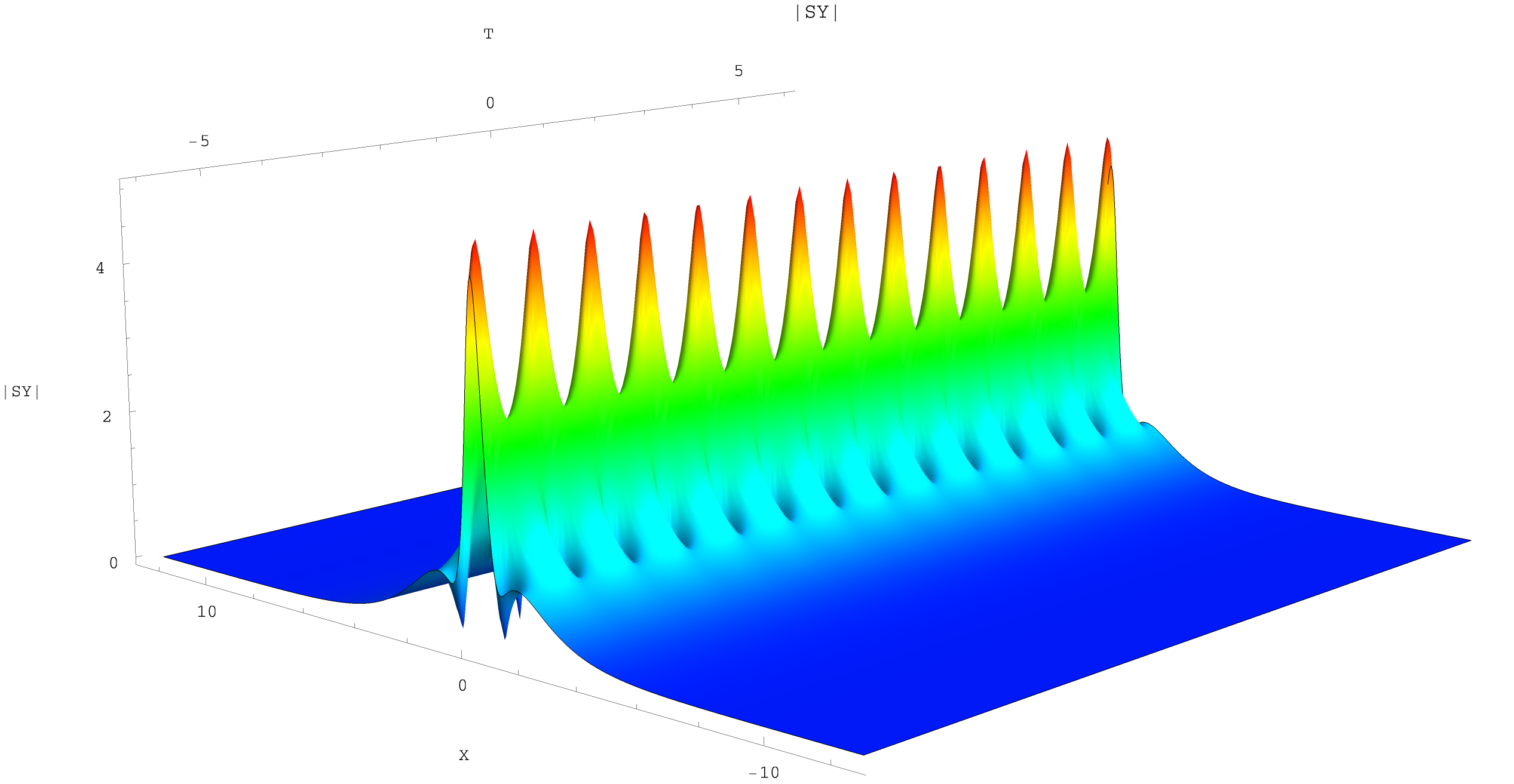}
              \includegraphics[width=7.0cm,height=4.5cm]{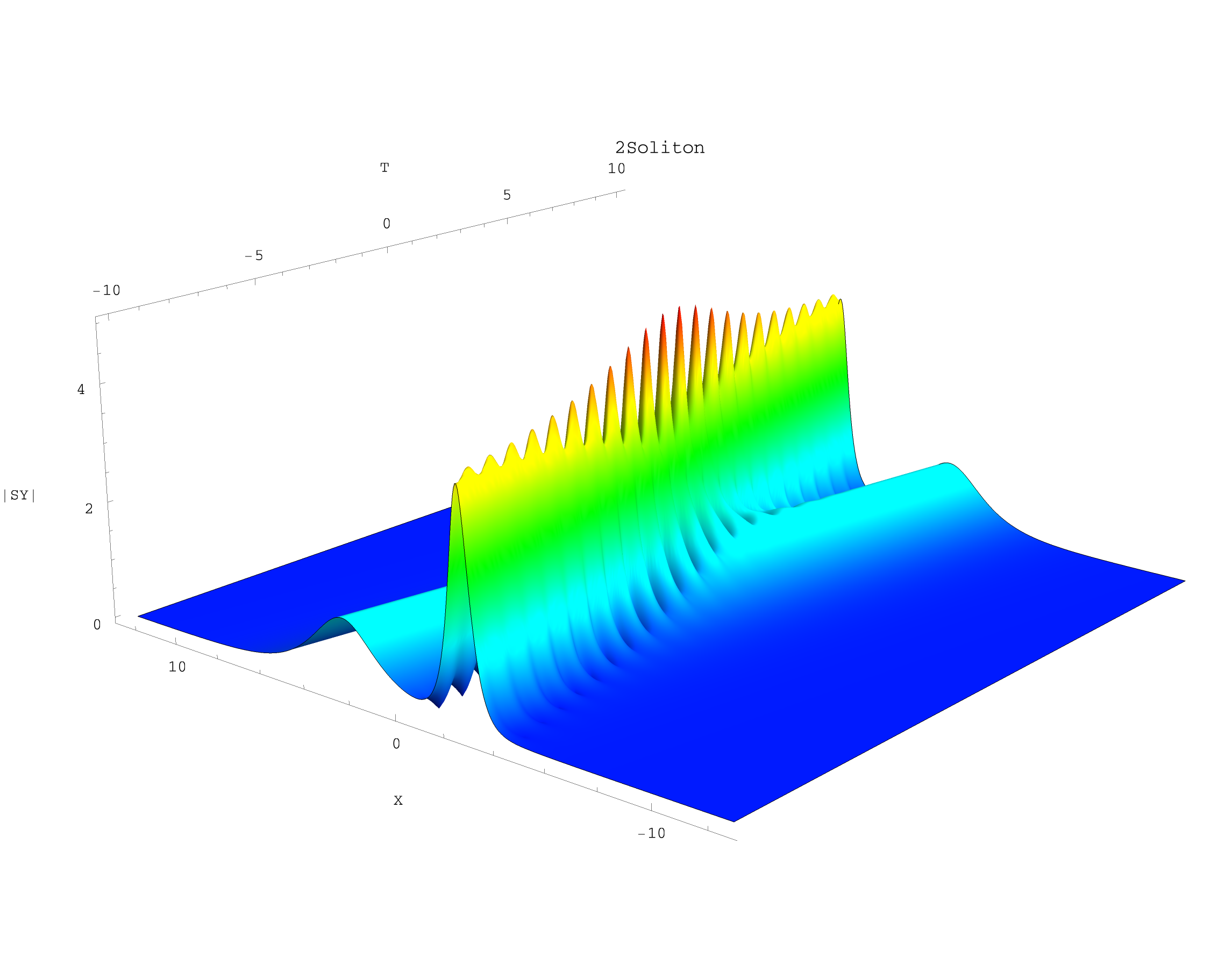}
              \caption{\emph{Left:} Absolute value of the SY  breather \eqref{SY}. Note the periodic in time behavior of this solution. \emph{Right:} 
              Absolute value of the double soliton \eqref{SYa1a2} close to (with $\alpha_1=0.1,~\alpha_2=0$) the SY breather \eqref{SY}. 
              The SY breather \eqref{SY} is part of the more complex family of NLS 2-solitons, and very small perturbations of the SY  breather 
              \eqref{SY} may lead to non breather solutions, like the one in \eqref{SYa1a2}. The left axis represents the $x$ variable, and the right axis, the $t$ variable.}\label{SY_fig}
  \end{figure}
%

\item[(iii)] {\bf The NLS case with nonzero background}. Finally, NLS with nonzero boundary condition, represented in \eqref{u_w}-\eqref{NLS}, possesses at least two important localized solutions characteristic of the \emph{modulational instability} phenomenon, which -roughly speaking- says that small perturbations of the exact Stokes solution $e^{it}$ are unstable and grow quickly. This unstable growth leads to a nontrivial competition with the (focusing) nonlinearity, time at which the solution is apparently stabilized.

\medskip

\begin{enumerate}

\medskip

\item[(iii.1)] {\bf The  Peregrine (P)  breather \cite{Peregrine}.} Given by
\be\label{P}
B_P(t,x):= 
e^{it}\left(1-\frac{4(1+2it)}{1+ 4t^2 +2x^2}\right),
\ee
which is a polynomially decaying (in space and time) perturbation of the nonzero background given by the Stokes wave $e^{it}$, which appears and disappears from nowhere \cite{Akhmediev2}. See Fig. \ref{P_KM} left for details. Some interesting connections have been made between the Peregrine soliton \eqref{P} and the intensely studied subject of \emph{rogue waves} in ocean \cite{Zakharov,Shrira,Akhmediev2,Kliber0} (see also \cite{BPSS} for an alternative explanation to the rogue wave phenomenon). Very recently, Biondini and Mantzavinos \cite{BM} showed, using inverse scattering techniques, the existence and long-time behavior of a global solution to \eqref{NLS} in the {integrable} case $(p=3)$, but under certain exponential decay assumptions at infinity, and a \emph{no-soliton} spectral condition (which, as far as we understand, does not define an open subset of the space of initial data). 

\medskip

Note that, because of time and space invariances in NLS, for any $t_0,x_0\in\R,$ $B_P(t-t_0,x-x_0)$ is also a Peregrine breather.

\medskip
\medskip

\item[(iii.2)] {\bf The Kuznetsov-Ma  (KM)  breather.}  The final object that we will consider in this paper is the Kuznetsov-Ma (KM) breather \cite{Kuznetsov,Ma}, given by the compact expression \cite{Akhmediev}
\be\label{KM}
\begin{aligned}
B_{KM}(t,x) : =  
e^{it}\Bigg[  1- \sqrt{2}\beta \frac{(\beta^2 \cos(\al t)  + i\al \sin(\al t)) }{ \al \cosh(\beta x) - \sqrt{2} \beta \cos(\al t)}\Bigg],& ~  \\
 \al :=  (8a(2a-1))^{1/2}, \quad \beta := (2(2a-1))^{1/2}, \quad a>\frac12 .& ~ \\
\end{aligned}
\ee
Notice that in the formal limit $a \downarrow \frac12$ one recovers the Peregrine breather.   See Fig. \ref{P_KM} right for details. Note that $B_{KM}$ is a Schwartz perturbation of the Stokes wave, and therefore a smooth classical solution of \eqref{NLS}.  It has been also observed in optical fibre experiments, see Kliber et al. \cite{Kibler} and references therein for a complete background on the mathematical problem and its physical applications.  
\end{enumerate}
\end{itemize}

          \begin{figure}[h!]
              \includegraphics[width=0.45\linewidth]{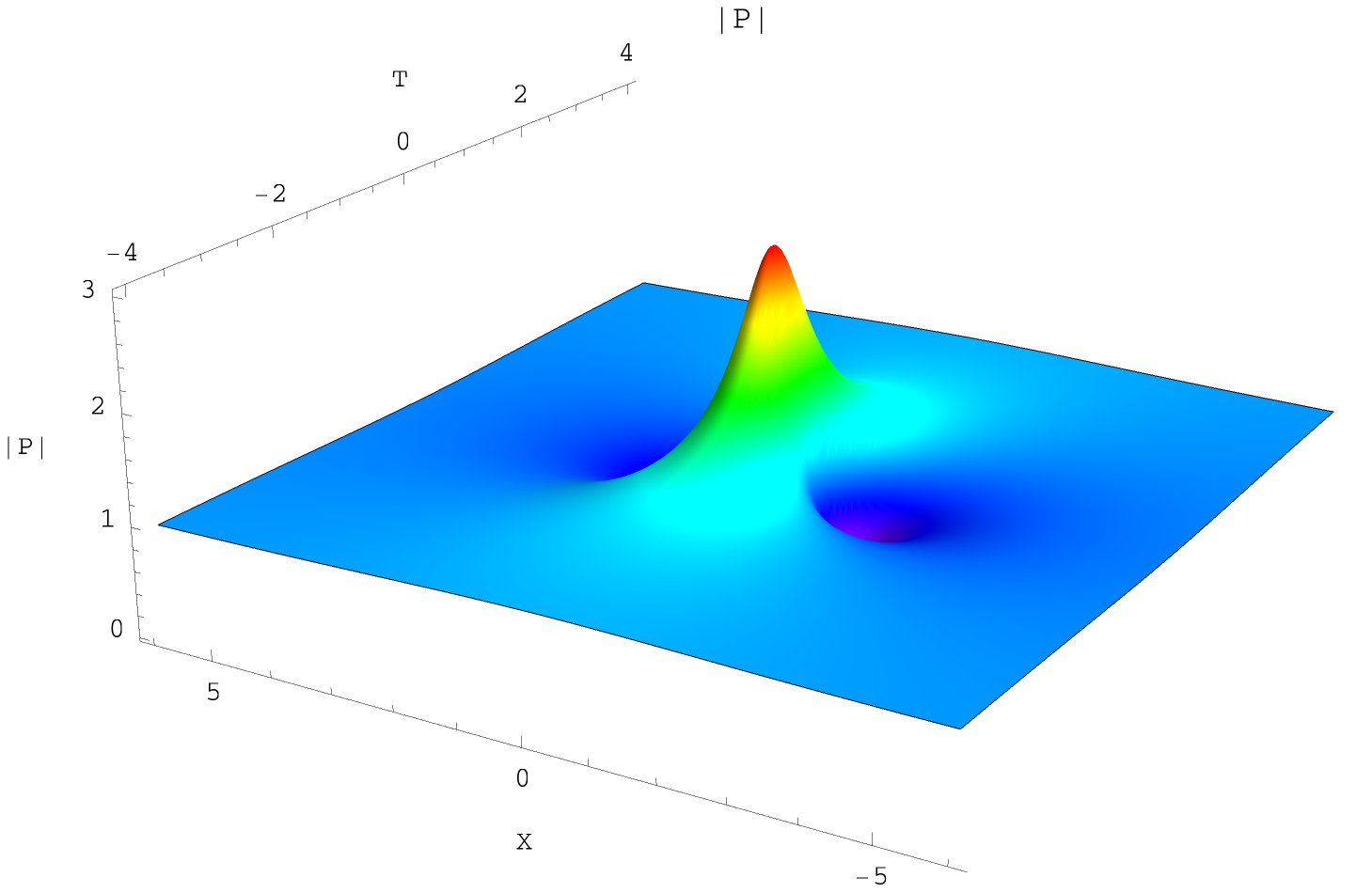}
              \includegraphics[width=0.45\linewidth]{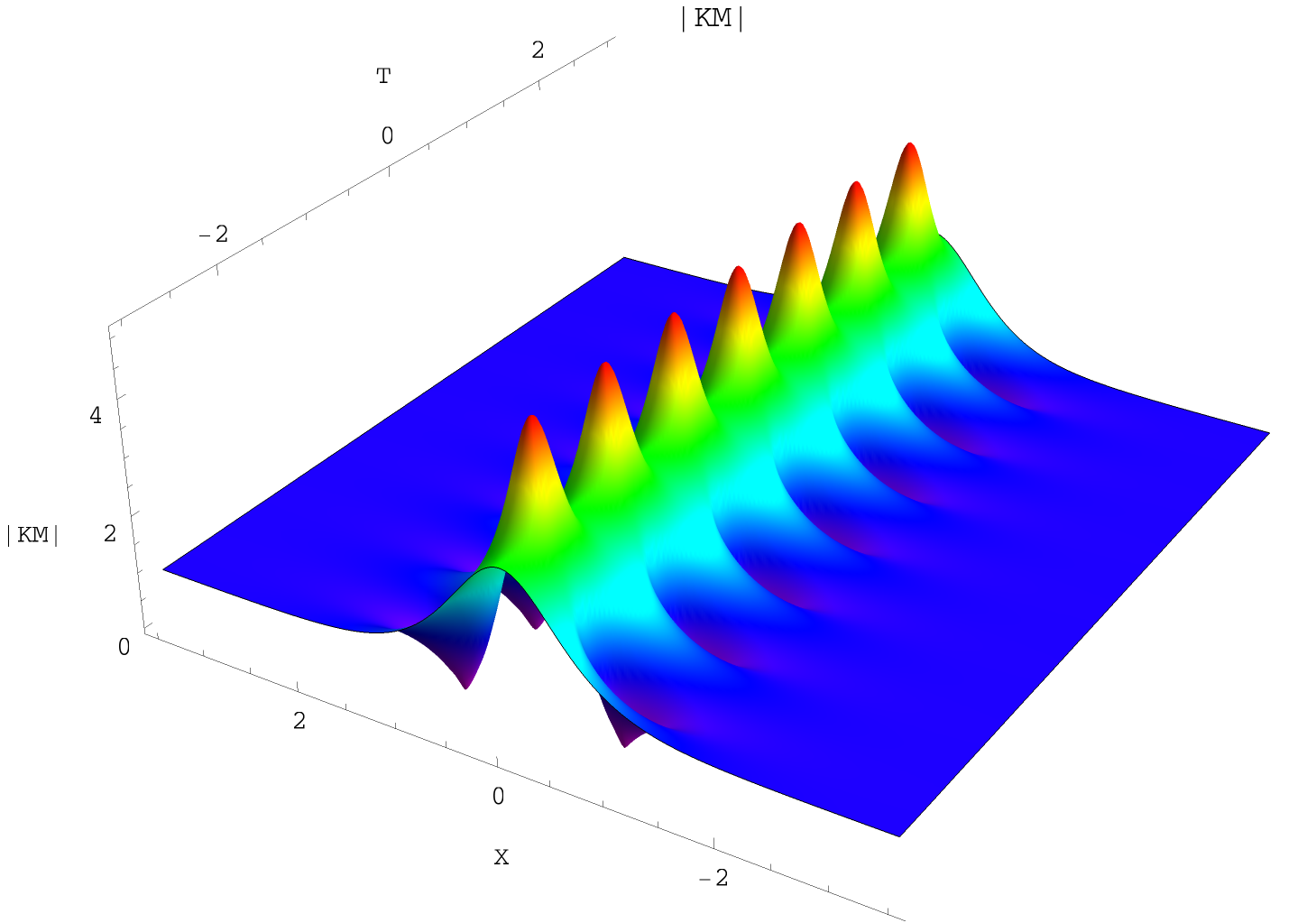}
              \caption{\emph{Left:} Absolute value of the P breather \eqref{P}. Note the localized character in space and time. \emph{Right:}  Absolute value of the KM breather \eqref{KM}. The left axis represents the $x$ variable, and the right one the $t$ variable.}\label{P_KM}
          \end{figure}

%
%
  
\medskip


Using a simple argument coming from the modulational instability of the equation \eqref{NLS}, in \cite{Munoz2} it was proved for the first time, and in a rigorous form, that both $B_{KM}$ and $B_{P}$ are {\bf unstable} with respect to perturbations in Sobolev spaces $H^s$, $s>\frac12$. Previously, Haragus and Klein \cite{KH} showed numerical instability of the Peregrine breather, giving a first hint of its unstable character. The proof of this result uses the fact that Peregrine and Kuznetsov-Ma breathers are in some sense converging to the background final state (i.e. they are {\bf asymptotically stable}) in the whole space norm $H^s(\R)$, a fact forbidden in Hamiltonian systems with conserved quantities and stable solitary waves. A further extension of this result, valid for periodic perturbations of the \emph{Akhmediev breather}, was proved in \cite{AFM}. Please see more details on the Akhmediev breather in \cite{AFM}.

\bigskip

\section{Main results}

The results in this paper can be characterized in two principal guidelines: a first one concerning a variational characterization for each breather above considered, and a second one related to stability and instability properties associated to that characterization.

\subsection{Variational characterization} Our first result is the following variational characterization of $B_{SS}$, $B_{SY}$, $B_P$ and $B_{KM}$ in \eqref{R}-\eqref{SY}-\eqref{P}-\eqref{KM}. 

We will also identify each dispersive model in this paper with its respective breather solution. Indeed, let
\[
SS=\hbox{Sasa-Satsuma \eqref{SS}}, \quad SY=\hbox{Satsuma-Yajima \eqref{NLS0}}, 
\]
and
\[
KM=\hbox{Kuznetsov-Ma \eqref{NLS}}, \quad P=\hbox{Peregrine \eqref{NLS}}. 
\]
Our first result is the following variational characterization of all these breather solutions. We will prove that, essentially, all of them satisfy the same nonlinear fourth order ODE, up to particular constants.

\begin{thm}[Elliptic equations satisfied by $U(1)$ breather solutions]\label{TH1}
Let $B=B_{X}$ be any of the solutions defined in \eqref{R}-\eqref{SY}-\eqref{P}-\eqref{KM}, with $X\in\{SS,SY,KM,P\}$. 
Then we have

\medskip

\ben 
\item  For $X=SS$,   $B=B_X$ satisfies
\be\label{perEcBp}
\begin{aligned}
& B_{(4x)} +8B_x^2 \bar B +14|B|^2 B_{xx}+6B^2 \bar B_{xx}  + 12|B_x|^2 B + 24 |B|^4 B \\
& \qquad - 2(\bt^2-\al^2) (B_{xx} + 4 |B|^2 B) +(\al^2+\bt^2)^2 B=0.
\end{aligned}
\ee

\medskip
\item If $X=SY$ and  $B=B_{SY}$,  
\be\label{Ec_SY}
\begin{aligned}
& B_{(4x)} + 3B_x^2 \bar B + 4 |B|^2B_{xx} + B^2 \bar B_{xx} + 2|B_x|^2 B   + \frac32 |B|^4B \\
&\qquad  - (c_2^2+c_1^2)(B_{xx} +  |B|^2 B) + c_2^2c_1^2B =0.
\end{aligned}
\ee

\medskip

\item For $X=KM$ and $\beta$ as in \eqref{KM}, $B=B_{KM}$ solves  
\be\label{Ec_KM}
\begin{aligned}
& B_{(4x)} + 3B_x^2 \bar B +(4 |B|^2-3) B_{xx}+ B^2 \bar B_{xx}  + 2  |B_x|^2 B  \\
&\qquad  + \frac32 (|B|^2-1)^2 B - \beta^2(B_{xx} +  (|B|^2-1) B) =0.
\end{aligned}
\ee
In particular, for $X=P$ one has that $B=B_P$ satisfies the limiting case
\be\label{Ec_P}
\begin{aligned}
& B_{(4x)} + 3B_x^2 \bar B +(4 |B|^2-3) B_{xx}+ B^2 \bar B_{xx}  + 2  |B_x|^2 B \\
&\qquad  + \frac32 (|B|^2-1)^2 B  =0.
\end{aligned}
\ee
\een
\end{thm}

\begin{rem}[Equivalence between $SS$ and $SY$ breathers]
Note that, except by some particular constants, $SS$ and $SY$ breathers satisfy {\bf the same} variational, fourth order elliptic equation. This fact reveals a deep connection between the $SS$ and $NLS$ integrable models. The case of $KM$ and $P$ breathers slightly differs from the previous cases because of suitable modifications appearing from their nonzero boundary value at infinity.
\end{rem}

\begin{rem}[New connections between $KM$ and $P$ breathers]
Note that the elliptic equation for the $P$ breather \eqref{Ec_P} is directly obtained by the formal limit $\beta\to 0$ in the $KM$ elliptic equation \eqref{Ec_KM}. This is in concordance with the expected behavior of the $KM$ breather as $a\to \frac12^+$, see \eqref{KM}.
\end{rem}

Theorem \ref{TH1} will be a particular consequence of the following variational characterization of each breather above mentioned. Recall that for $m\in\N$, the vector space $H^m(\R;\Com)$ corresponds to the Hilbert space of complex-valued functions $f:\R\to\Com$, with $m$ derivatives in $L^2(\R;\Com)$, endowed with the standard norm.

\begin{thm}[Variational characterization]\label{TH1a}
Each breather mentioned in Theorem \ref{TH1} is \emph{critical point} of a real-valued functional of the form
\be\label{H}
\mathcal H_X[u] := F_X[u] + m_X E_X[u] + n_X M_X[u], 
\ee
where
\ben
\item $F_X$, $E_X$ and $M_X$ are respective $H^2$, $H^1$ and $L^2$ based {\bf conserved quantities} for the dispersive model $X$ around the zero background or the Stokes wave $e^{it}$, depending on the particular limit value of the breather at infinity. Here, $E_X$ and $M_X$ corresponds to suitable energy and mass, respectively; 
\smallskip
\item $\mathcal H_X$ is well-defined for $u \in B_X+H^2(\R;\Com)$;
\smallskip
\item This functional is conserved for $H^2$ perturbations of the respective dispersive model $X$.
\smallskip
\item $m_X,n_X\in\R$ are well-chosen parameters, depending only on the {\bf nontrivial internal parameters} of the breather $B_X$; in particular: 
\smallskip
\ben
\item For $X=SS$, one has $m_X=- 2(\bt^2-\al^2) $ and $n_X=(\al^2+\bt^2)^2$.
\item For $X=SY$, one has $m_X=(c_2^2+c_1^2)$ and $n_X=c_2^2c_1^2$.
\item For $X=KM$, one has $m_X=+ \beta^2$ and $n_X=0$.
\item For $X=P$, one has $m_X=n_X=0$.
\een
\smallskip
\item Each breather $B_X$ is a critical point for the functional $\mathcal H_X,$ in the sense that for $X\in \{SS,SY,KM,P\}$,
\be\label{Hp}
\mathcal H_X'[B_X](z)=0, \quad \hbox{for all } ~ z\in H^2(\R;\Com). 
\ee
\een
\end{thm}

\begin{rem}
Theorem \ref{TH1a} states that all $U(1)$ breathers considered in this paper (and possibly several others not considered here 
by length considerations, such as Davey-Stewartson \cite{KS1,KS2} and the Manakov system \cite{Yang}) satisfy {\bf the same variational characterization}. This property exactly coincides in the $SS$ case with the classical mKdV characterization \cite{AM}; 
however, in the remaining $SY$, $KM$ and $P$ cases, it certainly differs in the choice of respective constants for the construction of $\mathcal H_X$.
\end{rem}

\begin{rem}
Theorem \ref{TH1a} also reveals that $KM$ and $P$ breathers obey, in some sense, {\bf degenerate variational characterizations}. More precisely, 
the $KM$ breather characterization does not require the use of the $L^2$ based mass term $M_{KM}$, and even worse, the $P$ breather does not require the mass and the energy $M_P$ and $E_P$, respectively:
\[
F_P'[B_P]\equiv 0.
\]
The absence of these two quantities may be related to the fact that 
\[
M_P[B_P]=E_P[B_P]=0, \qquad \hbox{(see Remark \ref{CL}),}
\]
meaning a particular form of instability (recall that mass and energy terms are somehow convex terms aiding to the stability of solitonic structures). We would like to further stress the fact that the variational characterization of the famous Peregrine breather is in $H^2$, since mass and energy are useless. See also Remark \ref{CLmom} for more about the zero character of $KM$ and $P$ conservation laws.
\end{rem}

\begin{rem}
We believe that  Theorem \ref{TH1a} describes for the first time, as far as we understand, the variational characterization of the Peregrine breather. It also describes in simple terms the connection between the Kuznetsov-Ma and Peregrine breathers.
\end{rem}

\begin{rem}
Theorem \ref{TH1} will be a (not so direct) consequence of the critical point character of each breather in Theorem \ref{TH1a}, identity \eqref{Hp}. Section \ref{Sect:4} is devoted to the proof of this fact.
\end{rem}

The proof of Theorem \ref{TH1} is simple, variational and follows previous ideas presented in \cite{AM} for the case of mKdV breathers, and \cite{AMP1} for the case of the Sine-Gordon breather (see also \cite{MP} for a recent improvement of this last result, based in \cite{AM1}).  The main differences are in the complex-valued nature of the involved breathers, and the nonlocal character of the $KM$ and $P$ breathers. Some special attention must be put to find the constants $m_X$ and $n_X$ above, a task that required some time and a large amount of computations, but finally we have found each of them. 

\medskip

\subsection{Stability and instability results} Next, we establish some stability and instability properties for the considered breathers. As usual, we start out with the $SS$ case. In this paper, we show nonlinear stability of this breather.

\begin{thm}[Nonlinear stability of the SS breather]\label{TH2}
The  SS breather \eqref{R} is orbitally stable in $H^2(\R;\Com)$. 
\end{thm}

A more precise statement of stability is given in Theorem \ref{TH2a}. The proof of Theorem \ref{TH2} follows the ideas in \cite{AM}, but the proofs are considerably harder, because of the complex-valued character of the involved linearized operator around the breather solution. After some nontrivial preliminary results, we prove that this linear operator is nondegenerate and has only a unique negative eigenvalue, a property shared by the mKdV breather. Recall that the mKdV breather is real-valued, and proofs are considerably simpler in that case. Theorem \ref{TH2} is, as far as we understand, the first rigorous nonlinear stability result for a $U(1)$ symmetry breather.

\medskip

Our proof does work even in the {\bf double humped case}, despite the fact that in this case the linearized operator $\mathcal H_{SS}''[B_{SS}]$ has a more complex structure. No such nonlinear stability result was known in the literature, even in the single humped case. 

\medskip

Now we consider the SY breather. Recall that it is well-known that the SY breather is unstable, see e.g. \cite{Yang}. However, this lack of stability is only mild, in the sense that the SY breather \eqref{SY} is instead part of a larger family of 2-soliton states $B_{SY,gen}$, given by a complicated formula, see \eqref{SYa1a2}. This larger family is indeed, stable, as it was proved by Kapitula \cite{Kap}. Further details on the variational structure of the full 2-soliton family, in the spirit of Theorem \ref{TH1}, can be found in Appendix \ref{2solitonNLS}. On the other hand, the construction of $N$-solitons in the nonintegrable NLS cases has been carried out for the first time by Martel and Merle \cite{MManihp}, and more recently by Nguyen \cite{Nguyen1}. Note that in this last reference, a breather like solution such as the SY breather \eqref{SY} has not yet been constructed. The stability of these nonintegrable  $N$-soliton solutions has been addressed in $H^1$ and for some particular nonlinearities (essentially supercritical), see \cite{MMT2}. Finally, nonexistence of NLS breathers with the oddness parity property and any nonlinearity has been recently proved in \cite{MEM}. 

\medskip

Finally, we consider the case of $KM$ and $P$ breathers. Recall that \emph{both are unstable}, see \cite{Munoz2}. In this paper we further improve the results in \cite{Munoz2} by showing the following nonlinear instability property:

\begin{thm}[Direction of instability of the Peregrine breather]\label{TH4}
Let $B=B_P$ be a Peregrine breather, critical point of the functional $\mathcal H_P$ defined in \eqref{H}. Then the following is satisfied. Let $z_0\in H^2$ be any sufficiently small perturbation. Then, as $t\to -\infty$,
\be\label{Insta_P}
\begin{aligned}
\mathcal H_P'[B_P](z_0)= &~{} 0, \quad  \hbox{ but}\\
\mathcal H_P''[B_P](z_0,z_0)= &~{} \frac12 \int  (|w_{x}|^2 -|w|^2 - w^2 )(t) +O(\|z_0\|_{H^1}^3) +o_{t\to +\infty}(1),
\end{aligned}
\ee
where $w=w(t):= e^{-it} \partial_x z_0\in H^1$. 
\end{thm}

\begin{rem}
The previous result gives a precise expression for the lack of stability in Peregrine breathers. Essentially, the continuous spectrum of the second derivative of the Lyapunov functional $\mathcal H_P$ stays below zero, a phenomenon that induces exponential growth in time for arbitrary perturbations of the associated linear dynamics. 
\end{rem}


\begin{rem}
Theorem \ref{TH4} can be recast as an absence of spectral gap for the linearized dynamics; we will not pursue this fact in the Peregrine case, but instead we will exemplify this fact using the Kuznetsov-Ma breather KM.
\end{rem}

In the case of the KM breather, things are more complicated, and the previous result is not valid, since $B_{KM}$ does not decay to the Stokes wave at time infinity (recall that KM breather oscillates around a Schwartz perturbation of the Stokes wave). Instead, we will prove the following

\begin{thm}[Absence of spectral gap and instability of the KM breather]\label{TH5}
Let $B=B_{KM}$ be a Kuznetsov-Ma breather \eqref{KM}, critical point of the functional $\mathcal H_{KM}$ defined in \eqref{H}. Then for all $a>\frac12$ we have
\be\label{espectro}
\begin{aligned}
\mathcal H_P'[B_{KM}]= &~ {} 0,\\
\mathcal H_{KM}''[B_{KM}](\partial_x B_{KM}) =&~{} 0, \\
\inf \sigma_c (\mathcal H_{KM}''[B_{KM}]) <&~ {} 0.%
\end{aligned}
\ee
Here $\sigma_c$ stands for the continuum spectrum of the linear operator associated to $\mathcal H_{KM}''[B_{KM}]$.
\end{thm}

\begin{rem}
The above theorem shows that the $KM$ linearized operator $\mathcal H_{KM}''$ has at least one \emph{embedded eigenvalue}. This is not true in the case of linear, real-valued operators with fast decaying potentials, but since  $\mathcal H_{KM}''$ is a matrix operator, this is perfectly possible. Additionally, a similar result for the Peregrine case could be proved, but the polynomial decay in space of the Peregrine breather makes this result more complicated to establish for the moment.
\end{rem}

\begin{rem}
Note that classical stable solitons or solitary waves $Q$ easily satisfy the estimate $\inf \sigma_c (\mathcal H_Q''[Q])>0$, where $\mathcal H_Q''$ is the standard quadratic form associated to the energy-mass or energy-momentum variational characterization of $Q$. Even in the cases of the mKdV breather $B_{mKdV}$ \cite{AM} or Sine-Gordon breather $B_{SG}$ \cite{AMP1}, one has the gap  $\inf \sigma_c (\mathcal H_{B_{mKdV}}''[B_{mKdV}])>0$ and also $\inf \sigma_c (\mathcal H_{B_{SG}}''[B_{SG}])>0$. The $KM$ breather does not follow this property at all, another consequence of the modulational instability present in the NLS equation with nonzero boundary value at infinity.
\end{rem}

\begin{rem}
This result is in concordance with the fact that the KM breathers are unstable, as shown in \cite{Munoz2}. 
\end{rem}

\subsection*{Organization of this paper} This paper is organized as follows. In Section \ref{Sect:2} we establish some preliminary results needed for the proof of Theorems \ref{TH1} and \ref{TH1a}. Section \ref{Sect:3} deals with the proof of Theorem \ref{TH1a}, needed for the proof of Theorem \ref{TH1}. Section \ref{Sect:4} is devoted to the proof of Theorem \ref{TH1}. In Section \ref{Sect:5} we prove Theorem \ref{TH2}. 
Section \ref{Sect:7} is concerned with the proof of Theorem \ref{TH4}. Finally, Section \ref{Sect:8} deals with Theorem \ref{TH5}.


\subsection*{Acknowledgments} We would like to thank the Applied Mathematics Department of the University of Granada, the IMUS at University of Sevilla, Spain; 
the Departamento de Ingenier\'ia Matem\'atica (DIM) of U. Chile, and the Mathematics Department La Sapienza U., in Roma, Italy, 
where part of this work was done. We also thank the referees for their deep and careful reading of our manuscript, that helped to improve a previous version of this work.

\bigskip

\section{Preliminaries}\label{Sect:2}

The purpose of this section is to gather several results present in the literature, needed below. We first present a result for the Sasa-Satsuma breather.

\subsection{Non variational PDE in the SS case} The following results are essentially contained in \cite{Peli_Yang}. From \eqref{R} and \eqref{SS}, it is not difficult to see that the soliton profile $Q_\beta$ satisfies the ODE
\[
Q_\beta''' +3i\alpha Q_\beta'' +6 |Q_\beta|^2 Q_\beta' +6i\alpha |Q_\beta|^2 Q_\beta + 3Q_\beta(|Q_\beta|^2)' -\beta^2 Q_\beta' -3i\alpha\beta^2 Q_\beta=0.
\]
This equation can be rewritten as
\bel{edoQ2}
Q_\beta''' +9 Q_\beta \bar Q_\beta Q_\beta'  + 3Q_\beta^2 \bar Q_\beta'   -\beta^2 Q_\beta'  +3i\alpha \Big( Q_\beta'' -\beta^2 Q_\beta +2 Q_\beta ^2 \bar Q_\beta  \Big)=0.
\ee
Note that this is a third order equation, and it seems that it cannot be integrated one more time. This exact equation will be used to prove \eqref{perEcBp}. 

\begin{rem}\label{embedded}
Note that the term \emph{embedded soliton} comes from \eqref{edoQ2}. Unlike the standard NLS ODE $Q''-Q+Q^3=0$, \eqref{edoQ2}  in its linear form $Q_\beta'''    -\beta^2 Q_\beta'  +3i\alpha \Big( Q_\beta'' -\beta^2 Q_\beta   \Big)=0$ has ``continuous spectrum'' solutions of the form $e^{iax}$, $a\in\R$; see \cite{Peli_Yang} for more details about this concept.
\end{rem}

\subsection{Conserved quantities}\label{Energies} In this subsection we consider the conserved quantities needed for the proof of Theorem \ref{TH1} and the definition of $\mathcal H_X$ in \eqref{H}. In what follows, we adopt the subscript $X\in\{SS,SY,KM,P\}$ to denote the conservation laws needed according to the respective breather $B_X$.

\medskip

\noindent
{\it Sasa-Satsuma.} Recall the Sasa-Satsuma equation \eqref{SS}. The following quantities are invariant of the motion, on sufficiently regular solutions: the mass
\begin{equation}\label{M}
M_{SS}[u]:=\int |u|^2dx,
\end{equation}
the energy
\bel{E}
E_{SS}[u]:=\int \Big(|u_x|^2 - 2|u|^4\Big)dx ,
\end{equation}
and the $H^2$ based energy
\bel{F}
F_{SS}[u]:=\int \Big(|u_{xx}|^2 - 8|u|^2|u_x|^2 -3((|u|^2)_x)^2 +8|u|^6\Big)dx.
\end{equation}
For complement purposes, one has
\be\label{massBreSS}
M_{SS}[B_{SS}]=2  \beta,
\ee
and
\[
E_{SS}[B_{SS}]=-\frac23 +4m^2+2m^2 \sqrt{1+m^2} \log\left(\frac{2+m^2-2\sqrt{1+m^2}}{m^2}\right), \qquad m:=\frac{\alpha}{\beta}.
\]
These two identities are easily checked using e.g. Mathematica. For the mass $M_{SS}[B_{SS}]$, since up to translation and phase factor,
$B_{SS}(x) = \beta Q(\beta x),$ one can easily compute that

\be\label{massSolSS}
\int_{\R}|Q(x)|^2 dx = 2.
\ee
\medskip

\noindent
{\it Satsuma-Yajima.} It is known that the NLS \eqref{NLS} with zero boundary condition at infinity possesses the following formally conserved quantities: the classical mass
\be\label{Mass_NLS0}
M_{SY}[u]: =\int |u|^2,
\ee
and the focusing energy
\be\label{Energy_NLS0}
E_{SY}[u]:=\int |u_x|^2 - \frac12 \int |u|^4.
\ee
The additional $H^2$ based energy is given by the expression
\be\label{F_NLS0}
F_{SY}[u]:= \int \Big(|u_{xx}|^2 -3 |u|^2|u_x|^2 -2(\re (\bar u u_x))^2  + \frac12 |u|^6 \Big).
\ee

\noindent
{\it Peregrine and Kuznetsov-Ma.} For simplicity in the computations, it is convenient to write \eqref{NLS} for $w$ in terms of the function $u$ in \eqref{u_w}. With this choice, both for $X=KM$ and $P$, one has the mass
\be\label{Mass_NLS}
M_{X}[u]: =\int (|u|^2 -1),
\ee
the energy
\be\label{Energy_NLS}
E_{X}[u]:=\int |u_x|^2 - \frac12 \int (|u|^2-1)^2,
\ee
and the Stokes wave + $H^2$ perturbations conserved energy:
\be\label{F_NLS}
F_{X}[u]:= \int \Big(|u_{xx}|^2 -3 (|u|^2-1)|u_x|^2 -\frac12((|u|^2)_x)^2  + \frac12 (|u|^2-1)^3 \Big). 
\ee

\begin{rem}\label{CL}
In \cite{Munoz2}, it was computed the mass and energy \eqref{Mass_NLS}-\eqref{Energy_NLS} of the Peregrine \eqref{P} and Kuznetsov-Ma \eqref{KM} breathers. Indeed, one has
\[
M_P[B_P]= E_P[B_P]=0,
\]
(however, the $L^2$-norm of $B_P(t)$ is never zero, but converges to zero as $t\to +\infty$), and
\[
M_{KM}[B_{KM}] = 4\beta, \quad E_{KM}[B_{KM}] =-\frac 83\bt^3.
\]
Note that $P$ has same energy and mass as the Stokes wave solution (the nonzero background), a property not satisfied by the standard soliton on zero background. Also, compare the mass and energy of the Kuznetsov-Ma breather with the ones obtained in \cite{AM} for the mKdV breather.
\end{rem}

\begin{rem}[Momentum laws]\label{CLmom}
{\color{black} Another important conserved quantity here is the Momentum
\be\label{Momentum_0}
P_X[u]:= \ima\int \bar u u_x,
\ee
valid in the $X=SS,SY$ cases, and 
\be\label{Momentum_1}
P_X[u]:= \ima\int  (\bar u-e^{-it}) u_x,
\ee
for the $X=P,KM$ cases. Note that both quantities are well-defined 
and finite in the case of a breather $B_X$, and essentially measure the speed of each breather. It is not difficult to show (or using a symbolic computing software) that 
\be\label{Momentum_2}
 P_{SS}[B_{SS}]=-\al\sqrt{\al^2+\bt^2}\log\Big(\frac{1}{\al^2}\big(2\bt^2+\al^2+2\bt\sqrt{\al^2+\bt^2}\big)\Big), 
\ee
and
\be\label{Momentum_3}
P_{SY}[B_{SY}]=  P_{P}[B_{P}]= P_{KM}[B_{KM}]=0.
\ee
We can then conclude that, except for $SS$ breathers, which have nonzero momentum, $SY$, $KM$ and $P$ breathers are zero speed solutions. This is in concordance with the characterization of periodic in time breathers, for which 
\[
\frac{d}{dt} M_{SY}[u] = const. P_{SY}[u].
\]
Therefore, breathers must have zero momentum. See  \cite{MEM} for another point of view about this fact. Note instead that, under a suitable Galilean transformation, they must have nonzero momentum.
}
\end{rem}

\bigskip

\section{Higher energy expansions: Proof of Theorem \ref{TH1a}}\label{Sect:3}

\medskip

This section is devoted to the proof of Theorem \ref{TH1a}. In what follows, we consider real-valued parameters $m_X,n_X$, for each $X\in \{SS,SY,KM,P\}$ as follows:
\smallskip
\ben
\item For $X=SS$, one has $m_X=- 2(\bt^2-\al^2) $ and $n_X=(\al^2+\bt^2)^2$ (see \eqref{R}).
\smallskip
\item For $X=SY$, one has $m_X=(c_2^2+c_1^2)$ and $n_X=c_2^2c_1^2$. 
\smallskip
\item For $X=KM$, one has $m_X=+ \beta^2$ and $n_X=0$ (see \eqref{KM}).
\smallskip
\item For $X=P$, one has $m_X=n_X=0$ (see \eqref{P}).
\een
These are the parameters previously mentioned in Theorem \ref{TH1a}, item (4).

\medskip

Consider the Lyapunov functional $\mathcal H_X$ defined by
\[
\mathcal H_X[u] = F_X[u] + m_X E_X[u] + n_X M_X[u], 
\]
where $F_X$, $E_X$ and $M_X$ were introduced in Subsection \ref{Energies}. This is exactly the functional considered in Theorem \ref{TH1a}, and more specifically, \eqref{H}. Note that this functional is a linear combination of conserved quantities mass \eqref{M}-\eqref{Mass_NLS}, energy \eqref{E}-\eqref{Energy_NLS}, and the second energy in $F_X$ \eqref{F}-\eqref{F_NLS}.

\medskip

Consequently, items (1)-(4) in Theorem \ref{TH1a} are easily proved.

\medskip

It remains to prove item (5) in Theorem \ref{TH1a}, and the fact that breathers $B_X$ are critical points for $\mathcal H_X$. These last facts will be a consequence of the following Proposition, and Theorem \ref{TH1}.


\begin{prop}[Variational characterization of $SS, SY, KM$ and $P$ breathers]\label{Decomposition_Prop}
For each $X\in\{SS,SY,KM,P\}$, and for each $z\in H^2(\R)$, we have 
\be\label{decomp}
\mathcal H_{X}[ B_X + z]=  \mathcal H_{X}[ B_X] + \mathcal G_X[z] + \mathcal Q_X[z] + \mathcal N_X[z],
\ee
where
\begin{itemize}
\item $\mathcal H_{X}[ B_X] $ does not depend on time. Moreover,
\[
 \mathcal H_{X}[ B_P]=0.
\]
\item The linear term in $z$ is given as 
\be\label{lineal_lineal}
\mathcal G_X[z] = 2\re\int \bar{z} G[B_X],
\ee
with ($B=B_X$)
\be\label{perEcBp_0}
\begin{aligned}
 G[B_{SS}] :=& ~ B_{(4x)} +8B_x^2 \bar B +14|B|^2 B_{xx}+6B^2 \bar B_{xx}  +12  |B_x|^2 B + 24 |B|^4 B \\
& - 2(\bt^2-\al^2) (B_{xx} + 4 |B|^2 B) +(\al^2+\bt^2)^2 B;
\end{aligned}
\ee
\smallskip
\be\label{Ec_SY_0}
\begin{aligned}
  G[B_{SY}] := & ~B_{(4x)} + 3B_x^2 \bar B + 4 |B|^2B_{xx} + 2|B_x|^2 B + B^2 \bar B_{xx}   + \frac32 |B|^4B \\
&  - (c_2^2+c_1^2)(B_{xx} +  |B|^2 B) + c_2^2c_1^2 B ;
\end{aligned}
\ee
\smallskip
\be\label{Ec_KM_0}
\begin{aligned}
 G[B_{KM}] := & ~ B_{(4x)} + 3B_x^2 \bar B +(4 |B|^2-3) B_{xx} + 2  |B_x|^2 B + B^2 \bar B_{xx}   \\
&  + \frac32 (|B|^2-1)^2 B - \beta^2(B_{xx} +  (|B|^2-1) B) ;
\end{aligned}
\ee
and
\be\label{Ec_P_0}
\begin{aligned}
 G[B_{P}] := & ~ B_{(4x)} + 3B_x^2 \bar B +(4 |B|^2-3) B_{xx} + 2  |B_x|^2 B+ B^2 \bar B_{xx}  \\
& + \frac32 (|B|^2-1)^2 B .
\end{aligned}
\ee
\item The quadratic functional is given as 
\be\label{quad_quad}
\mathcal Q_{X}[z] := \re \int\bar{z}\mathcal{L}_{X}[z]dx
\ee
where 
\be\label{perEcBp_1}
\begin{aligned}
 \mathcal{L}_{SS}[z]:=&~ z_{4x} + 14|B|^2 z_{xx} +  6B^2\bar{z}_{xx} + (12B\bar{B}_x + 16\bar{B}B_{x})z_x + 12 BB_x\bar{z}_x\\
 &+\big(14\bar{B}B_{xx}+12 |B_x|^2 + 12B\bar{B}_{xx}+72|B|^4 \big)z\\
 &+ \big(14BB_{xx} + 8B_x^2 + 48|B|^2B^2 \big)\bar{z} -m_{SS}( z_{xx} +8|B|^2 z + 4B^2 \bar z) +n_{SS}z,\\ 
\end{aligned}
\ee
\smallskip
\be\label{Ec_SY_1}
\begin{aligned}
 \mathcal{L}_{SY}[z]:= &~ z_{4x} + 4|B|^2z_{xx} + B^2\bar{z}_{xx} + (2B\bar{B}_x+6\bar{B}B_x   ) z_x + 2BB_x\bar{z}_x \\
&  + \left(  2|B_x|^2 + 2B\bar{B}_{xx}      +4\bar{B}B_{xx} + \frac92|B|^4 \right)z + \left(3 B_x^2 + 4BB_{xx} + 3|B|^2B^2 \right)\bar{z} \\ 
&   - m_{SY}[z_{xx} +B^2\bar{z} + 2|B|^2z] + n_{SY} z,
\end{aligned}
\ee
\smallskip
\be\label{Ec_KM_1}
\begin{aligned}
  \mathcal{L}_{KM}[z]:= &~  z_{4x} +(4|B|^2-3)z_{xx} +B^2\bar{z}_{xx} + (2B\bar{B}_x  +  6 \bar BB_x) z_x +2 BB_x \bar z_x \\
& ~{}   + \Big(2 |B_x|^2+  2B\bar B_{xx} + 4\bar B B_{xx}\Big)z  +( 3B_x^2 +  4BB_{xx}) \bar{z} \\
& ~{}   + \frac32(|B|^2-1)^2 z + 6(|B|^2-1)B\re (B\bar z)   \\
&~{}- m_{KM}[z_{xx} +B^2\bar{z} + (2|B|^2-1)z],
\end{aligned}
\ee
\smallskip
and
\be\label{Ec_P_1}
\begin{aligned}
  \mathcal{L}_{P}[z]:=&~  z_{4x} +(4|B|^2-3)z_{xx} +B^2\bar{z}_{xx} + (2B\bar{B}_x  +  6 \bar BB_x) z_x +2 BB_x \bar z_x \\
& ~{}   + \Big(2 |B_x|^2+  2B\bar B_{xx} + 4\bar B B_{xx}\Big)z  +( 3B_x^2 +  4BB_{xx}) \bar{z} \\
& ~{}   + \frac32(|B|^2-1)^2 z + 6(|B|^2-1)B\re (B\bar z) .
\end{aligned}
\ee
\item Finally, assuming $\|z\|_{H^1}$ small enough, we have the nonlinear estimate
\be\label{Nonlinear}
|\mathcal N_X[z]|\lesssim \|z\|_{H^1}^3.
\ee
\end{itemize}
\end{prop}

\begin{rem}
Note that terms \eqref{perEcBp_0}-\eqref{Ec_P_0} precisely correspond to the nonlinear elliptic equations presented in Theorem \ref{TH1}; in that sense, once Theorem \ref{TH1} is proved, Theorem \ref{TH1a} is also completely proved.
\end{rem}

\begin{rem}
All linearized operators appearing from \eqref{quad_quad} contain terms in $z$ and $\bar z$. Consequently, these are $2\times 2$ matrix valued operators with fourth order components each, more demanding that the ones found in \cite{AMP1} for the Sine-Gordon case, which was composed by fourth and second order mixed terms only. 
\end{rem}

\begin{proof}[Proof of Proposition \ref{Decomposition_Prop}]
We proceed following standard steps. We will  prove \eqref{decomp} decomposing $\mathcal H_{X}[ B_X + z]$ into zeroth, first (linear in $z$), second (quadratic in $z$) and higher order terms (cubic or higher in $z$). The convention that we will use below is the following:
\begin{itemize}
\item Zeroth order terms will have the subscript ``0''. 
\item First order terms will have the subscript \emph{lin}.
\item Second order terms will have the subscript \emph{quad}.
\item Higher order terms will have the subscript \emph{non}.
\end{itemize}

\medskip

\noindent
{\it Step 1. Contribution of the mass terms.} Recall the masses \eqref{M}, \eqref{Mass_NLS0} and \eqref{Mass_NLS}. We have for $X=SS,SY$ and $B=B_X$, 
\[
M_{X}[B+z] = \int |B+z|^2 = \int |B|^2 + 2\re \int B\bar z + \int |z|^2.
\]
Similarly, for $X=KM, P$,
\[
\begin{aligned}
M_{X}[B+z] =&~{}  \int |B+z|^2 -1\\
=&~{} \int (|B|^2-1) + 2\re \int B\bar z + \int |z|^2.
\end{aligned}
\]
The linear and quadratic contributions here are the same for both equations. Therefore, if $X=SS,SY, KM, P,$
\be\label{M_Lineal}
\begin{aligned}
M_{X,0} :=& ~{}M_{X}[B],\qquad
M_{X,lin} := 2\re \int B\bar z\\
&~{} \text{and}\qquad M_{X,quad} := \int |z|^2.
\end{aligned}
\ee
Note that $M_{KM,lin} $ and $M_{P,lin} $ may not be necessarily well-defined, without adding cancelling terms (see below). 
As for the mass terms, there are no higher order contributions to the expansion of $\mathcal H_{X}[ B_X + z]$:
\be\label{M_X_non}
M_{SS,non} =M_{SY,non}=M_{KM,non}=M_{P,non}=0.
\ee

\medskip

\noindent
{\it Step 2. Contribution of the energy terms.} Recall the energies \eqref{E} and \eqref{Energy_NLS0}. If $X=SS$ and $B=B_X,$
\[
\begin{aligned}
E_{SS}[B+z] =&~  \int |B_x+z_x|^2 -2\int |B+z|^4 \\
=&~{} \int |B_x|^2 + 2\re \int B_x\bar z_x + \int |z_x|^2 - 2\int \left(|B|^2 + 2\re (B\bar z) +  |z|^2 \right)^2.
\end{aligned}
\]
Therefore, we have
\[
\begin{aligned}
E_{SS}[B+z] =&~ {} E_{SS}[B]  + 2\re \int \bar z (-B_{xx} ) + \int |z_x|^2 \\
&~ {} - 2\int \left( (2\re (B\bar z))^2 +  |z|^4 + 4|B|^2\re (B\bar z) + 2|B|^2|z|^2 +4|z|^2 \re(B\bar z) \right).
\end{aligned}
\]
Clearly $E_{SS,0}=E_{SS}[B]$. The linear contribution here is
\be\label{E_SS_Lineal}
E_{SS,lin}:=2\re \int \bar z (-B_{xx} -4 |B|^2 B),
\ee
and the quadratic contribution  is
\be\label{E_SS_Quad}
E_{SS,quad}=\int |z_x|^2  - 2\int \left( 2(\re (B\bar z))^2 + 2|B|^2|z|^2 \right).
\ee
Finally, the higher order contribution is given by
\be\label{E_SS_non}
E_{SS,non}= - 2\int \left( |z|^4  +4|z|^2 \re(B\bar z) \right).
\ee
Now, consider the energy in the Satsuma-Yajima (SY) case \eqref{Energy_NLS0}. If $X=SY$ and $B=B_X$,
\[
\begin{aligned}
E_{SY}[B+z] =&~ {} E_{SY}[B]  + 2\re \int \bar z (-B_{xx} ) + \int |z_x|^2 \\
&~ {} - \frac12\int \left( 4(\re (B\bar z))^2 +  |z|^4 + 4|B|^2\re (B\bar z) + 2|B|^2|z|^2 +4|z|^2 \re(B\bar z) \right),
\end{aligned}
\]
so that $E_{SY,0}:=E_{SY}[B]$, and the linear  contribution is
\be\label{E_LinealSY}
E_{SY,lin}:=2\re \int \bar z \left(-B_{xx} - |B|^2 B\right).
\ee
and the quadratic contribution  is given by
\be\label{E_SY_Quad}
E_{SY,quad}:=\int  |z_x|^2  - \frac12\int \left( (2\re (B\bar z))^2+ 2|B|^2|z|^2\right).
\ee
Finally, the higher order contributions are
\be\label{E_SY_Non}
E_{SY,non}:=- \frac12\int \left(  |z|^4 +4|z|^2 \re(B\bar z) \right).
\ee
Consider now the NLS case. The energy is given by \eqref{Energy_NLS}, and if $X=KM$ or $P$, and $B=B_X$, we have
\[
\begin{aligned}
E_{X}[B+z] =&~  \int |B_x+z_x|^2 - \frac12\int (|B+z|^2-1)^2 \\
=&~{} \int |B_x|^2 + 2\re \int B_x\bar z_x + \int |z_x|^2 - \frac12\int \left(|B|^2-1 + 2\re (B\bar z) +  |z|^2 \right)^2.
\end{aligned}
\]
Therefore, we have
\[
\begin{aligned}
& E_{X}[B+z] \\
&~ {}= E_{X}[B]  + 2\re \int \bar z (-B_{xx} ) + \int |z_x|^2 \\
&~ {} \quad - \frac12\int \left( (2\re (B\bar z))^2 +  |z|^4 + 4(|B|^2-1)\re (B\bar z) + 2(|B|^2-1)|z|^2 +4|z|^2 \re(B\bar z) \right).
\end{aligned}
\]
Consequently, $E_{X,0}:=E_{X}[B]$. The linear contribution here is
\be\label{E_NLS_Lineal}
E_{X,lin}:=2\re \int \bar z (-B_{xx} - (|B|^2-1) B),
\ee
and the quadratic contribution  is
\be\label{E_NLS_Quad}
E_{X,quad}:=\int |z_x|^2  - \frac12 \int \left( (2\re (B\bar z))^2+ 2(|B|^2-1)|z|^2\right).
\ee
Finally, the higher order contribution is
\be\label{E_NLS_Non}
E_{X,non}:=- \frac12\int \left(   |z|^4  +4|z|^2 \re(B\bar z) \right).
\ee

\bigskip

\noindent
{\it Step 3. Contribution of the second energy terms. The SS case.} We start by considering the case $X=SS$. Note that from \eqref{F},
\be\label{F_SS_deco}
\begin{aligned}
F_{SS}[B+z ]=&~{} \int \Big(|B_{xx}+ z_{xx}|^2 - 8|B+z|^2|B_x+z_x|^2 -3((|B+z|^2)_x)^2 +8|B+z|^6\Big)\\
=&~ {}\underbrace{\int \Big(|B_{xx}|^2 +|z_{xx}|^2 + 2\re (B_{xx}\bar z_{xx})\Big)}_{F_{SS,1}} \\
&~ {} \underbrace{-8\int \Big(|B|^2 + |z|^2 +2\re (B\bar z) \Big) \Big(|B_x|^2 +|z_x|^2 + 2\re(B_x \bar z_x) \Big)}_{F_{SS,2}} \\
&~ {} \underbrace{-3 \int \Big((B_x+z_x)(\bar B + \bar z) +(B+z)(\bar B_x + \bar z_x)\Big)^2}_{F_{SS,3}} \\
&~ {}\underbrace{+ 8\int \Big(|B|^2 + 2\re (B\bar z) +  |z|^2\Big)^3}_{F_{SS,4}}
\end{aligned}
\ee
We have
\[
F_{SS,1}=\int (|B_{xx}|^2 +|z_{xx}|^2) + 2\re \int \bar z B_{xxxx},
\]
hence $F_{SS,1,0} = \int |B_{xx}|^2$,
\be\label{F_SS_1_lin}
F_{SS,1,lin} =  2\re \int \bar z B_{xxxx},
\ee
and
\be\label{F_SS_1_quad}
F_{SS,1,quad} = \int |z_{xx}|^2.
\ee
Clearly 
\be\label{F_SS_1_non}
F_{SS,1,non}=0.
\ee
Analogously,
\be\label{F_SS_2}
F_{SS,2} =-8\int \Big(|B|^2 + |z|^2 +2\re (B\bar z) \Big) \Big(|B_x|^2 +|z_x|^2 + 2\re(B_x \bar z_x) \Big).
\ee
We have $F_{SS,2,0}=-8\int  |B|^2|B_x|^2$. The linear terms are
\be\label{F_SS_2_lin}
\begin{aligned}
F_{SS,2,lin} = &~{} -8\int \Big(2|B|^2\re(B_x \bar z_x)   +2|B_x|^2 \re (B\bar z) \Big) \\
 = &~{} -16 \re \int \bar z \Big(-(|B|^2 B_x )_x    + |B_x|^2  B \Big) = 16\re \int \bar z (B_x^2 \bar B + |B|^2 B_{xx}),
 \end{aligned}
\ee
and the quadratic terms, taken from \eqref{F_SS_2}, are
\[
\begin{aligned}
F_{SS,2,quad} = &~{} -8\int \Big(|B|^2|z_x|^2 + |B_x|^2|z|^2 +   B \bar z  B_x \bar z_x + B \bar z  \bar B_x z_x + \bar B z  B_x \bar z_x + \bar B z  \bar B_x  z_x  \Big) \\
= &~{} -8\re \int \Big(|B|^2|z_x|^2 + |B_x|^2|z|^2\Big)  -8 \int \Big(  B \bar z  \bar B_x z_x + \bar B z  B_x \bar z_x \Big) \\
&~{} + 4 \int \Big( (B   B_x)_x  \bar z^2 + (\bar B   \bar B_x)_x  z^2\Big) \\
= &~{} -8\re \int \Big(|B|^2 z_x \bar z_x + |B_x|^2 z\bar z -  ( \bar B  B_x  z)_x\bar  z + B\bar B_x z_x\bar z - (B   B_x)_x  \bar z^2\Big)\\
= &~{} 8\re \int  \bar z \Big( ( B \bar B z_x)_x  - |B_x|^2 z  +  ( \bar B  B_x  z)_x - B\bar B_x z_x  + (B   B_x)_x  \bar z \Big).
 \end{aligned}
\]
Simplifying,
\be\label{F_SS_2_quad}
\begin{aligned}
F_{SS,2,quad} = &~{} 2\re \int \bar z \Big( 4|B|^2z_{xx} +  8 \bar BB_x z_x+ 4\bar BB_{xx} z + (4 B_x^2 + 4 BB_{xx}) \bar{z}  \Big).
 \end{aligned}
\ee
Finally, the higher order terms are
\be\label{F_SS_2_non}
\begin{aligned}
F_{SS,2,non} = &~{} -8\int \Big( |z|^2 +2\re (B\bar z) \Big) |z_x|^2  -16 \int  |z|^2 \re(B_x \bar z_x)  \\
= &~{}  -8\int \Big( |z_x|^2  |z|^2+2 |z_x|^2\re (B\bar z)  -2  |z|^2 \re(B_x \bar z_x)\Big) .
 \end{aligned}
\ee
Now, we deal with $F_{SS,3}$:
\be\label{F_SS_3}
\begin{aligned}
F_{SS,3}=& ~ {}-3 \int \Big((B_x+z_x)(\bar B + \bar z) +(B+z)(\bar B_x + \bar z_x)\Big)^2\\
=& ~ {}-3 \int \Big( B_x \bar B + B_x \bar z + z_x \bar B + z_x \bar z  + \bar B_x  B + \bar B_x z + \bar z_x B + \bar z_x  z \Big)^2.
\end{aligned} 
\ee
The linear terms are
\[
\begin{aligned}
F_{SS,3,lin} =  & ~ {} -3 \int  B_x \bar B \Big( B_x \bar z + z_x \bar B + \bar B_x z + \bar z_x B \Big)\\
 & ~ {} -3 \int \Big( B_x \bar z + z_x \bar B + \bar B_x z + \bar z_x B  \Big)\Big(B_x \bar B  + \bar B_x  B  \Big)\\ 
 & ~ {} -3 \int  \bar B_x  B  \Big( B_x \bar z + z_x \bar B + \bar B_x z + \bar z_x B \Big) \\
=& ~ {} -6\re \int  B_x \bar B \Big( B_x \bar z + z_x \bar B + \bar B_x z + \bar z_x B \Big)\\
 & ~ {} -12  \re \int \Big( B_x \bar z + z_x \bar B \Big) \re( B_x \bar B) .
\end{aligned}
\]
Therefore,
\[
\begin{aligned}
F_{SS,3,lin} =& ~ {} -6\re \int  B_x \bar B \Big( B_x \bar z + z_x \bar B + \bar B_x z + \bar z_x B \Big)\\
& ~{}  -12  \re \int \bar z \Big( B_x \re( B_x \bar B)  -  (\re( B_x \bar B)  B)_x \Big) \\
=& ~ {} -12\re \int  B_x \bar B \Big( \re(B_x \bar z) + \re( \bar z_x B) \Big) \\
& ~{} -12  \re \int \bar z \Big( B_x \re( B_x \bar B)  -  (\re( B_x \bar B)  B)_x \Big) \\
=& ~ {} -12\re \int  \re(B_x \bar B) \Big( B_x \bar z +  \bar z_x B \Big)\\
&~{} -12  \re \int \bar z \Big( B_x \re( B_x \bar B)  -  (\re( B_x \bar B)  B)_x \Big) .
\end{aligned} 
\]
Collecting similar terms, we get
\[
\begin{aligned}
F_{SS,3,lin} =& ~ {}-12\re \int  \bar z  \Big( \re(B_x \bar B)B_x - ( \re(B_x \bar B)B)_x +   \re( B_x \bar B)B_x -  (\re( B_x \bar B)  B)_x  \Big)\\
=& ~ {}-24\re \int  \bar z  \Big( \re(B_x \bar B)B_x - ( \re(B_x \bar B)B)_x  \Big)\\
 =&~ {} 24\re \int  \bar z  (\re(B_x \bar B))_xB \\
=&~ {} 12\re \int  \bar z  (B_{xx} \bar B +B_x \bar B_x + \bar B_{xx}  B +B_x \bar B_x ) B,
\end{aligned}
\]
so that
\be\label{F_SS_3_lin}
F_{SS,3,lin}=  12\re \int  \bar z  (|B|^2 B_{xx} + B^2 \bar B_{xx}  + 2 B|B_x|^2).
\ee
The quadratic terms, taken from \eqref{F_SS_3}, are
\[
\begin{aligned}
& F_{SS,3,quad}\\
& = ~{}-3 \int   \Big(B_x^2\bar{z}^2 + \bar{B}_x^2z^2 + \bar{B}^2z_x^2 + B^2\bar{z}_x^2 + 4\bar{B}B_xz_x\bar{z} + 4B\bar{B}_xz\bar{z}_x\\
& \qquad\qquad   \qquad + 2|B_x|^2|z|^2 + 2|B|^2|z_x|^2 + 2 (2\re (B\bar z))(2\re (B_x\bar{z}_x))\Big) \\
& = ~{} -6\re\int \Big( B_x^2\bar{z}^2 + B^2\bar{z}_x^2  +2 B \bar B_x z\bar z_x+ 2 \bar B B_x z_x \bar z + |B_x|^2|z|^2 + |B|^2|z_x|^2 \Big) \\
& \qquad-6 \int \Big( B\bar z B_x\bar z_x + B\bar z \bar B_x z_x + \bar B z B_x\bar z_x + \bar B z \bar B_xz_x \Big)	\\
& = ~{} -6\re\int \Big( B_x^2\bar{z}^2 + B^2\bar{z}_x^2  -2( B \bar B_x z)_x\bar z + 2 \bar B B_x z_x \bar z + |B_x|^2|z|^2 + |B|^2|z_x|^2 \Big) \\
& \qquad +3 \int \Big( (B B_x)_x \bar z^2  + (\bar B  \bar B_x)_x z^2  \Big)- 6\re \int  \bar B B_x z \bar z_x - 6\re \int  B \bar  B_x   z_x \bar z	\\
& = ~{} 6\re\int \bar{z} \Big( -B_x^2\bar{z}  + (B^2\bar{z}_x)_x  + 2( |B_x|^2 z +B \bar B_{xx} z +B \bar B_x z_x ) - 2 \bar B B_x z_x   \Big) \\
& \qquad +6\re \int \bar z \Big(- |B_x|^2 z  + (|B|^2 z_x)_x  + (B B_x)_x \bar z  +  (\bar B B_x)_x z + \bar B B_x z_x  -B \bar  B_x   z_x \Big).
\end{aligned}
\]
Consequently,
\be\label{F_SS_3_quad}
\begin{aligned}
 F_{SS,3,quad} & = ~{}  2\re \int  \bar z\Big(3 |B|^2z_{xx} +3 B^2\bar{z}_{xx} + 6B\bar{B}_x z_x + 6BB_x \bar z_x \\
 &~{} \qquad \qquad \qquad  + (6 |B_x|^2+ 6 B\bar B_{xx} + 3\bar B B_{xx})z + 3 B B_{xx} \bar z \Big). 
\end{aligned}
\ee
Finally,
\be\label{F_SS_3_non}
\begin{aligned}
F_{SS,3,non}=& ~ {}-3 \int \Big(  (z_x \bar z )^2  + ( \bar z_x  z)^2 +2B_x \bar z^2  z_x  +2B_x \bar z_x  |z|^2 +2  z_x^2 \bar B  \bar z +2  |z_x|^2 \bar B    z \\
& ~{} \qquad \qquad + 2 z_x |z|^2  \bar B_x  + 2 |z_x|^2 \bar z B + 2 |z_x|^2 |z|^2 +2  \bar B_x z^2  \bar z_x   +2 \bar z_x^2 B  z \Big).
\end{aligned}
\ee
As for $F_{SS,4}$, we have
\[
\begin{aligned}
F_{SS,4}=&~{}  8\int \Big(|B|^2 + 2\re (B\bar z) +  |z|^2\Big)^3\\
=&~{} 8\int \Big( |B|^4 + 4(\re (B\bar z))^2 +  |z|^4  +  4|B|^2 \re (B\bar z) + 2|B|^2  |z|^2 + 4|z|^2 \re (B\bar z)  \Big)\\
& \qquad\qquad \times \Big(|B|^2 + 2\re (B\bar z) +  |z|^2\Big). 
\end{aligned}
\]
Expanding terms, we have that the linear terms are given by
\be\label{F_SS_4_lin}
\begin{aligned}
F_{SS,4,lin} = &~ {} 8 \int \Big(2|B|^4\re (B\bar z) +4|B|^4\re (B\bar z) \Big) \\
=&~ {} 48\int   |B|^4 \re (B\bar z) = 48\re \int \bar z |B|^4 B.
\end{aligned}
\ee
On the other hand, the quadratic terms are given by
\be\label{F_SS_4_quad}
\begin{aligned}
F_{SS,4,quad} = & ~{} 8 \int \Big(3|B|^2B^2\bar{z}^2 + 3|B|^2\bar{B}^2z^2 + 9 |B|^4|z|^2 \Big) \\
= & ~{} 2 \re\int \Big(24|B|^2B^2\bar{z}^2 + 36 |B|^4|z|^2 \Big).
\end{aligned}
\ee
Finally,
\be\label{F_SS_4_non}
\begin{aligned}
F_{SS,4,non}=&~{} 8\int \Big(  4(\re (B\bar z))^2 ( 2\re (B\bar z) +  |z|^2) +  |z|^4(|B|^2 + 2\re (B\bar z) +  |z|^2) \\
& \qquad\quad +  4|B|^2 |z|^2 \re (B\bar z) + 2|B|^2  |z|^2 (2\re (B\bar z) +  |z|^2) \\
& \qquad\quad+ 4|z|^2 \re (B\bar z)(|B|^2 + 2\re (B\bar z) +  |z|^2)  \Big).
\end{aligned}
\ee

\medskip

\noindent
{\it Step 4. Gathering terms.}  We conclude from \eqref{F_SS_1_lin}, \eqref{F_SS_2_lin}, \eqref{F_SS_3_lin} and \eqref{F_SS_4_lin} that the linear part $F_{SS,lin}$ of $F_{SS}$ is given by
\[
\begin{aligned}
F_{SS,lin}:=&~  F_{SS,1,lin} +F_{SS,2,lin}+F_{SS,3,lin}+F_{SS,4,lin}\\
 = &~ 2\re\Big( \int \bar z B_{xxxx} + 8\int \bar z (B_x^2 \bar B + |B|^2 B_{xx}) \\
& \qquad \qquad + 6 \int  \bar z  (|B|^2 B_{xx} +B^2 \bar B_{xx}  + 2 B|B_x|^2) +24 \int \bar z |B|^4 B \Big)\\
 =&~ 2\re \int \bar z \left( B_{xxxx} +8B_x^2 \bar B +14|B|^2 B_{xx}+6B^2 \bar B_{xx}  +12 |B_x|^2B + 24 |B|^4 B \right) .
\end{aligned}
\]
On the other hand, collecting terms in \eqref{F_SS_1_quad}, \eqref{F_SS_2_quad}, \eqref{F_SS_3_quad} and \eqref{F_SS_4_quad}, the quadratic part of $F_{SS}$ is given by
\[
\begin{aligned}
&~{} F_{SS,quad}:=  F_{SS,1,quad} +F_{SS,2,quad}+F_{SS,3,quad}+F_{SS,4,quad}\\
& ~{}=  \int |z_{xx}|^2 \\
&~{} \quad + 2\re \int \bar z \Big(   4|B|^2z_{xx} +  8 \bar BB_x z_x+ 4\bar BB_{xx} z + (4 B_x^2 + 4 BB_{xx}) \bar z \Big)\\  
&~{}\quad   +2\re \int  \bar z\Big( 3 |B|^2z_{xx} +3 B^2\bar{z}_{xx} + 6B\bar{B}_x z_x + 6BB_x \bar z_x \\
 &~{} \qquad \qquad \qquad  + (6 |B_x|^2+ 6 B\bar B_{xx} + 3\bar B B_{xx})z + 3 B B_{xx} \bar z \Big)\\
&~{} \quad +2 \re\int \bar z \Big(24|B|^2B^2\bar{z} + 36 |B|^4 z \Big) \\
 &~  =2\re \int \bar z \Big(\frac12 z_{4x} + 7|B|^2 z_{xx} + 3B^2\bar{z}_{xx}  +(6 B\bar B_x   +8 \bar{B}B_x)z_x   +6 BB_x\bar{z}_x 
\\
& \qquad \qquad \qquad  +(  7\bar B B_{xx}  +  6 B\bar B_{xx} + 6|B_x|^2     + 36 |B|^4)z +  (7BB_{xx} + 4 B_x^2 + 24|B|^4B^2) \bar z  \Big).
\end{aligned}
\]
Finally, from \eqref{F_SS_1_non}, \eqref{F_SS_2_non}, \eqref{F_SS_3_non} and \eqref{F_SS_4_non}, we get
\be\label{F_SS_non}
\begin{aligned}
F_{SS,non}:= &~ {}F_{SS,1,non} +F_{SS,2,non}+F_{SS,3,non}+F_{SS,4,non}\\
=&~{}  -8\int \Big( |z_x|^2  |z|^2+2 |z_x|^2\re (B\bar z)  -2  |z|^2 \re(B_x \bar z_x)\Big)\\
&~{} -3 \int \Big(  (z_x \bar z )^2  + ( \bar z_x  z)^2 +2B_x \bar z^2  z_x  +2B_x \bar z_x  |z|^2 +2  z_x^2 \bar B  \bar z +2  |z_x|^2 \bar B    z \\
& ~{} \qquad \qquad + 2 z_x |z|^2  \bar B_x  + 2 |z_x|^2 \bar z B + 2 |z_x|^2 |z|^2 +2  \bar B_x z^2  \bar z_x   +2 \bar z_x^2 B  z \Big)\\
& ~{} + 8\int \Big(  4(\re (B\bar z))^2 ( 2\re (B\bar z) +  |z|^2) +  |z|^4(|B|^2 + 2\re (B\bar z) +  |z|^2) \\
& \qquad\qquad +  4|B|^2 |z|^2 \re (B\bar z) + 2|B|^2  |z|^2 (2\re (B\bar z) +  |z|^2) \\
& \qquad\qquad+ 4|z|^2 \re (B\bar z)(|B|^2 + 2\re (B\bar z) +  |z|^2)  \Big).
\end{aligned}
\ee
We can also collect higher order terms in the Lyapunov expansion. Specifically we have that from \eqref{M_X_non}, \eqref{E_SS_non} and \eqref{F_SS_non},
\be\label{Nss}
\begin{aligned}
\mathcal N_{SS}[z] := &~{} F_{SS,non} + m_{SS}E_{SS,non} + n_{SS} M_{SS,non} \\
 =&~{}  \int \Bigg(-8\re(B\bar z)|z_x|^2 - 8\re(B_x\bar z_x)|z^2| - 8|z|^2|z_x|^2 - 12\re(B_xz_x\bar z^2)\\
& \qquad \quad - 12\re(Bz\bar z_x^2) - 6\re(z_x^2\bar z^2) - 12\re(B_x\bar z_x)|z|^2 - 12\re(B\bar z)|z_x|^2 \\
& \qquad \quad - 6|z|^2|z_x|^2 + 24|B|^2(|z|^2 + 2\re(B\bar z))^2\\
& \qquad \quad + 8(|z|^2 + 2\re(B\bar z))^3 -2m_{SS}(|z|^4 + 2(2\re(B\bar z))|z|^2)\Bigg).
\end{aligned}
\ee
Clearly, in the case $\|z\|_{H^2}$ small, one has $|\mathcal N_{SS}[z]| \lesssim \|z\|_{H^1}^3$,  since $\|z\|_{L^\infty} \lesssim \|z\|_{H^1}$. Summarizing, we have the following expansion for the Lyapunov functional $\mathcal{H}_{SS}$:
\[
\begin{aligned}
\mathcal{H}_{SS}[B+z] =& ~ {}F_{SS}[B+z] +  m_{SS} E_{SS}[B+z] + n_{SS} M_{SS}[B+z] \nonu\\
=& ~ {}\mathcal{H}_{SS}[B] \\
&~ {} + 2\re \int \bar{z}\Big[B_{xxxx}+8B_x^2 \bar B +14|B|^2 B_{xx}+6B^2 \bar B_{xx} +12  |B_x|^2 B \\
&~ {} \qquad \qquad\qquad + 24 |B|^4 B - m_{SS}(B_{xx} +4 |B|^2 B)+ n_{SS}B\Big] \nonu\\
&~ {}+ 2\re\int \frac{\bar{z}}{2}\Big[z_{4x}  + 14|B|^2z_{xx} + 6B^2\bar{z}_{xx} - 14 B\bar B_x z_x -10\bar B B_x z_x - 16BB_x \bar{z}_x \\
&~ {} \qquad \qquad \qquad   - 14|B_x|^2z - 6B_x^2\bar{z}  + 48|B|^2B^2 \bar z  + 24|B|^4z \\
&~ {} \qquad \qquad \qquad - m_{SS}[z_{xx} +4B^2\bar{z} + 8|B|^2z] + n_{SS} z\Big]\\
&~ {}   + \mathcal N_{SS}[z] =:  \mathcal H_{SS}[ B] + \mathcal G_{SS}[z] + \mathcal Q_{SS}[z] + \mathcal N_{SS}[z],
\end{aligned} 
\]
with  $\mathcal N_{SS}[z]$ presented in \eqref{Nss}. 

\subsection*{End of proof in the SS case} This finally proves \eqref{decomp}, \eqref{lineal_lineal}-\eqref{perEcBp_0}, \eqref{quad_quad}-\eqref{perEcBp_1} and \eqref{Nonlinear} in the $SS$ case.

\medskip

Since the SY case is somehow standard and close to SS, we will prefer to prove in full detail the more complicated case of KM and P breathers; the remaining SY case will be at the end of the proof.

\medskip

\noindent
\emph{Step 5. Contribution of the second energy terms. The case of Kuznetsov-Ma and Peregrine.} Let $X=KM$ or $X=P.$ Now we deal with the contribution in $F_{X}$, given in \eqref{F_NLS}.
Compared with $F_{SS}$, there are minor differences, that we explain below. First of all, we also have the decomposition

\begin{align}\label{F_X_deco}
& F_{X}[B+z ] \nonu \\
&= ~{} \int \Big(|B_{xx}+ z_{xx}|^2 - 3(|B+z|^2-1)|B_x+z_x|^2 -\frac12((|B+z|^2)_x)^2 +\frac12(|B+z|^2-1)^3\Big) \nonu \\
&= ~ {}\underbrace{\int \Big(|B_{xx}|^2 +|z_{xx}|^2 + 2\re (B_{xx}\bar z_{xx})\Big)}_{F_{X,1}} \nonu \\
& \qquad ~ {} \underbrace{-3\int \Big((|B|^2-1)+ |z|^2 +2\re (B\bar z) \Big) \Big(|B_x|^2 +|z_x|^2 + 2\re(B_x \bar z_x) \Big)}_{F_{X,2}}\nonu  \\
& \qquad ~ {} \underbrace{-\frac12 \int \Big((B_x+z_x)(\bar B + \bar z) +(B+z)(\bar B_x + \bar z_x)\Big)^2}_{F_{X,3}} \nonu \\
& \qquad ~ {}\underbrace{+ \frac12\int \Big((|B|^2-1) + 2\re (B\bar z) +  |z|^2\Big)\Big((|B|^2-1) + 2\re (B\bar z) +  |z|^2\Big)^2}_{F_{X,4}}.
\end{align}
Consequently the zeroth, linear, quadratic and nonlinear parts  $F_{X,1,0}$, $F_{X,1,lin}$, $F_{X,1,quad}$ and $F_{X,1,non}$ described above, compared with \eqref{F_SS_1_lin}, \eqref{F_SS_1_quad} and \eqref{F_SS_1_non}, rest unchanged and we have $F_{X,1,0} = \int |B_{xx}|^2$,
\be\label{F_X_1}
F_{X,1,lin} =  2\re \int \bar z B_{xxxx}, \qquad F_{X,1,quad} = \int |z_{xx}|^2, \qquad F_{X,1,non}=0.
\ee
The term $F_{X,2,lin}$ is analogous to $F_{SS,2,lin}$ in \eqref{F_SS_2_lin}, except by a constant 3 (instead of 8) in front of it, and also the asymptotic constant equals 1. In fact,
 we have
\be\label{F_X_2_lin}
\begin{aligned}
F_{X,2,lin} = &~{} -3\int \Big(2(|B|^2-1)\re(B_x \bar z_x)   + 2|B_x|^2 \re (B\bar z) \Big) \\
 = &~{} -6 \re \int \bar z \Big(-((|B|^2-1) B_x )_x    + |B_x|^2  B \Big) \\
 = &~{} 6\re \int \bar z (B_x^2 \bar B + (|B|^2-1) B_{xx}).
 \end{aligned}
\ee
Also, the term $F_{X,2,quad}$ is analogous to $F_{SS,2,quad}$ in \eqref{F_SS_2_quad}, except by a constant 3 (instead of 8) in front of it and the asymptotic constant 1. In fact, we have
\be\label{F_X_2_quad}
\begin{aligned}
F_{X,2,quad} = &~{} -3\int \Big((|B|^2-1)|z_x|^2 + |B_x|^2|z|^2 + 2\re(B \bar z)\times 2\re(B_x \bar z_x)  \Big) \\
= &~{} 2\re \int \bar z \Big( \frac32 (|B|^2-1)z_{xx} +  3 \bar BB_x z_x+  \frac32 \bar BB_{xx} z + \frac32 ( B_x^2 +  BB_{xx}) \bar{z}  \Big).
 \end{aligned}
\ee
Finally, the nonlinear term $F_{X,2,non}$ is given by
\be\label{F_X_2_non}
\begin{aligned}
F_{X,2,non} = &~{} -3\int |z|^2  \Big(|z_x|^2 + 2\re(B_x \bar z_x) \Big)-6\int \re (B\bar z)  |z_x|^2 .
 \end{aligned}
\ee
Similarly, the term $F_{X,3,lin}$ is analogous to $F_{SS,3,lin}$ in \eqref{F_SS_3_lin}, except by a constant $\frac12$ (instead of 3) in front of it. We have first
\be\label{F_X_3}
\begin{aligned}
F_{X,3}=& ~ {}-\frac12 \int \Big((B_x+z_x)(\bar B + \bar z) +(B+z)(\bar B_x + \bar z_x)\Big)^2\\
=& ~ {} -\frac12 \int \Big( B_x \bar B + B_x \bar z + z_x \bar B + z_x \bar z  + \bar B_x  B + \bar B_x z + \bar z_x B + \bar z_x  z \Big)^2,
\end{aligned} 
\ee
and the linear contribution is given by
\be\label{F_X_3_lin}
F_{X,3,lin}=  2\re \int  \bar z  (|B|^2 B_{xx} + B^2 \bar B_{xx}  + 2 B|B_x|^2).
\ee
On the other hand, the quadratic contribution $F_{X,3,quad}$ from \eqref{F_X_deco} is analogous to $F_{SS,3,quad}$ in \eqref{F_SS_deco}-\eqref{F_SS_3_quad}, except by a constant 
$\frac12$ (instead of 3) in front of it. Therefore, the quadratic term is given by
\be\label{F_X_3_quad}
\begin{aligned}
 F_{X,3,quad} & = ~{}  2\re \int  \bar z\Big(\frac12 |B|^2z_{xx} +\frac12 B^2\bar{z}_{xx} + B\bar{B}_x z_x + BB_x \bar z_x \\
 &~{} \qquad \qquad \qquad  + ( |B_x|^2+  B\bar B_{xx} + \frac12\bar B B_{xx})z + \frac12 B B_{xx} \bar z \Big). 
\end{aligned}
\ee
The term $F_{X,3,non}$ is given now by
\be\label{F_X_3_non}
\begin{aligned}
F_{X,3,non}= &~{}  -\frac12 \int \Big(  ( z_x \bar z)^2+( \bar z_x  z)^2 \Big) - \int B_x\Big(    z_x \bar z^2 +   \bar z_x  |z|^2 \Big)\\
  &~{}  - \int \bar B \Big(   z_x^2    \bar z +   |z_x|^2  z \Big)  - \int \Big(  \bar B_x z_x |z|^2 + B  |z_x|^2 \bar z  +  |z_x|^2  | z|^2 \Big)\\
   &~{}  - \int \Big(   \bar B_x  \bar z_x  z^2 + B  \bar z_x^2    z \Big).
\end{aligned}
\ee
Finally, the term $F_{X,4,lin}$ requires more care than the others. We have this time ($X=KM,P$)
\[
F_{X,4}= \frac12\int \Big((|B|^2-1)  + 2\re (B\bar z) +  |z|^2\Big)\Big((|B|^2-1)  + 2\re (B\bar z) +  |z|^2\Big)^2.
\]
(Compare with $F_{SS,4}$ in \eqref{F_SS_deco}.) First of all, we have
\[
\begin{aligned}
F_{X,4}=&~{} \frac12\int \Big( (|B|^2-1)^2 + 4(\re (B\bar z))^2 +  |z|^4  +  4(|B|^2-1) \re (B\bar z)\\ 
& ~{} \qquad \qquad  + 2(|B|^2-1)  |z|^2 + 4|z|^2 \re (B\bar z)  \Big)\times\Big((|B|^2-1) + 2\re (B\bar z) +  |z|^2\Big). 
\end{aligned}
\]
Therefore, the linear terms are given by
\be\label{F_X_4_lin}
\begin{aligned}
F_{X,4,lin} 
=&~{}  \frac12 \int  3(|B|^2-1)^2 (2\re (B\bar z)) =  2\re\int \frac32(|B|^2-1)^2B\bar z. 
\end{aligned}
\ee
Moreover, the quadratic terms are given by
\[
\begin{aligned}
F_{X,4,quad}=&~{} \frac12\int \Big(  4(\re (B\bar z))^2   + 2(|B|^2-1)  |z|^2  \Big)(|B|^2-1)  +\frac12\int  (|B|^2-1)^2|z|^2 \\
& ~{} +\frac12\int \Big(   4(|B|^2-1) \re (B\bar z) \Big) 2\re (B\bar z) . 
\end{aligned}
\]
which simplify to
\be\label{F_X_4_quad}
\begin{aligned}
F_{X,4,quad} = &~{} \frac12 \int \Big( 3(|B|^2-1)^2 |z|^2 + 12(|B|^2-1)\re (B\bar z)\re (B\bar z) \Big) \\
&=\re \int \bar{z}\left( \frac32(|B|^2-1)^2 z + 6(|B|^2-1)B\re (B\bar z) \right).
\end{aligned}
\ee
Finally, $F_{X,4,non}$ is given by
\be\label{F_X_4_non}
\begin{aligned}
F_{X,4,non} = &~{}  2\int (|B|^2-1) \re (B\bar z)  |z|^2 \\
& ~{}+  \int \Big(  2(\re (B\bar z))^2  + (|B|^2-1)  |z|^2  \Big)\Big(2\re (B\bar z) +  |z|^2\Big)\\
& ~{}+  \frac12\int |z|^2 \Big(  |z|^2  + 4 \re (B\bar z)  \Big)\Big((|B|^2-1) + 2\re (B\bar z) +  |z|^2\Big).
\end{aligned}
\ee

\medskip

\noindent
\emph{Step 6. Gathering terms. The case of Kuznetsov-Ma and Peregrine.} From \eqref{F_X_1}, \eqref{F_X_2_lin}, \eqref{F_X_3_lin} and \eqref{F_X_4_lin} we conclude that the linear part of $F_{X}$, $X=NLS$ is given by
\[
\begin{aligned}
 F_{X,lin}:= & ~{} F_{X,1,lin} +F_{X,2,lin}+F_{X,3,lin}+F_{X,4,lin}\\
=& ~{} 2\re \int \bar z B_{xxxx} + 6\re \int \bar z (B_x^2 \bar B + (|B|^2-1) B_{xx})\\
&+ 2\re \int  \bar z  (|B|^2 B_{xx} + B^2 \bar B_{xx}  + 2 B|B_x|^2) +   2\re\int \frac32(|B|^2-1)^2B\bar z\\
=& ~{} 2\re \int \bar z \Big( B_{xxxx} + (4|B|^2-3) B_{xx} + 3B_x^2 \bar B +  2 B|B_x|^2 + B^2 \bar B_{xx} + \frac32(|B|^2-1)^2B\Big).
\end{aligned}
\]
On the other hand, collecting the terms in \eqref{F_X_1}, \eqref{F_X_2_quad}, \eqref{F_X_3_quad} and \eqref{F_X_4_quad}, the quadratic part of $F_{NLS}$ is given by 
\[
\begin{aligned}
 F_{X,quad}:= &~{}  F_{X,1,quad} +F_{X,2,quad}+F_{X,3,quad}+F_{X,4,quad}\\
=&~{} \int |z_{xx}|^2\\
&~{} + 2\re \int \bar z \left(\frac32 (|B|^2-1)z_{xx} +  3 \bar BB_x z_x+  \frac32 \bar BB_{xx} z + \frac32 ( B_x^2 +  BB_{xx}) \bar{z}  \right)\\
& ~{} +2 \re \int \Bigg( \frac12 |B|^2z_{xx} +\frac12 B^2\bar{z}_{xx} + B\bar{B}_x z_x + BB_x \bar z_x \\
 &~{} \qquad \qquad \qquad  + \Big( |B_x|^2+  B\bar B_{xx} + \frac12\bar B B_{xx}\Big)z + \frac12 B B_{xx} \bar z  \Bigg)\\
 &~{} + 2\re \int \bar{z}\left( \frac34(|B|^2-1)^2 z + 3(|B|^2-1)B\re (B\bar z) \right).
\end{aligned}
\]
Therefore,
\be\label{F_X_quad}
\begin{aligned}
 F_{X,quad} =&~{}2\re \int \frac{\bar z}2 \Bigg( z_{4x} +(4|B|^2-3)z_{xx} +B^2\bar{z}_{xx} + (2B\bar{B}_x  +  6 \bar BB_x) z_x +2 BB_x \bar z_x \\
& ~{} \qquad \qquad  \quad  + \Big(2 |B_x|^2+  2B\bar B_{xx} + 4\bar B B_{xx}\Big)z  +( 3B_x^2 +  4BB_{xx}) \bar{z} \\
& ~{} \qquad \qquad  \quad  + \frac32(|B|^2-1)^2 z + 6(|B|^2-1)B\re (B\bar z)    \Bigg).
\end{aligned}
\ee
Finally, we also collect the higher order terms in the Lyapunov expansion. Specifically we have that \eqref{F_X_1}, \eqref{F_X_2_non}, \eqref{F_X_3_non} and \eqref{F_X_4_non} leads to
\be\label{Nnls}
\begin{aligned}
\mathcal N_{X}[z] :=&~{}  F_{X,1,non} +F_{X,2,non}+F_{X,3,non}+F_{X,4,non} \\
=& ~{}\int \Bigg(-6\re(B\bar z)|z_x|^2 - 6\re(B_x\bar z_x)|z^2| - 3|z|^2|z_x|^2 - 2\re(B_xz_x\bar z^2) \\
&~{} \qquad - 2\re(Bz\bar z_x^2) - \re(z_x^2\bar z^2) - 2\re(B_x\bar z_x)|z|^2 - 2\re(B\bar z)|z_x|^2  - |z|^2|z_x|^2 \\
&~{} \qquad  + \frac32(|B|^2-1)(|z|^2 + 2\re(B\bar z))^2+ \frac12(|z|^2 + 2\re(B\bar z))^3 \\
&~{} \qquad  -\frac12m_{X}(|z|^4 + 2(2\re(B\bar z))|z|^2)\Bigg).
\end{aligned}
\ee
We conclude that Proposition \ref{Decomposition_Prop} in the $KM$ and $P$ cases (except for the proof of $\mathcal G_X[z] =0$) is deduced from the above representation. Indeed, we have \eqref{decomp} by gathering
\[
\begin{aligned}
\mathcal{H}_{X}[B+z] = & ~ F_{X}[B+z] +  m_{X} E_{X}[B+z] + n_{X} M_{X}[B+z] \nonu\\
=& ~{}\mathcal{H}_{X}[B] + \mathcal G_X[z] + \mathcal Q_X[z] + \mathcal N_X[z],
\end{aligned} 
\]
as desired, selecting $X=KM$ for the KM  breather,  $m_{KM}=\beta^2,~n_{KM}=0$; and selecting $X=P$ for the Peregrine breather, and $m_{P}=n_{P}=0$.

\medskip

%
%

\noindent
\emph{Step 7. The case of Satsuma-Yajima.} This case is very similar to the previous $KM/P$ cases, with some minor differences in constants. 
Let $m_{SY}=(c_2^2+c_1^2),~n_{SY}=c_2^2c_1^2$ as in the beginning of Section \ref{Sect:3}. Let also $X=SY$, $B=B_X$ and consider $F_{SY}[B+z]$ as in \eqref{F_NLS0}. First of all, note that the linear and quadratic contributions   $F_{SY,lin}$ and  $F_{SY,quad}$ from $F_{SY}[B+z]$ are as in the $KM/P$ cases, but removing the asymptotic constant $1$. Additionally, the higher order terms are given by
\be\label{Nsy}
\begin{aligned}
\mathcal N_{SY}[z]&:=\int \Big(-6\re(B\bar z)|z_x|^2 - 6\re(B_x\bar z_x)|z^2| - 3|z|^2|z_x|^2 - 2\re(B_xz_x\bar z^2) \\
&\qquad \qquad  - 2\re(Bz\bar z_x^2)- \re(z_x^2\bar z^2) - 2\re(B_x\bar z_x)|z|^2 - 2\re(B\bar z)|z_x|^2  - |z|^2|z_x|^2 \\
&\qquad \qquad + \frac32|B|^2(|z|^2 + 2\re(B\bar z))^2 + \frac12(|z|^2 + 2\re(B\bar z))^3 \\
& \qquad\qquad - \frac12m_{SY}(|z|^4 + 2(2\re(B\bar z))|z|^2)\Big).
\end{aligned}
\ee
Clearly we have the estimate $|\mathcal N_{SY}[z]|\lesssim \|z\|_{H^1}^3$ under small data assumptions. Finally, the expansion of the Lyapunov functional $\mathcal{H}_{SY}[B+z]$ is given by:
\[
\begin{aligned}
&~{} \mathcal{H}_{SY}[B+z] \\
&~{} = F_{SY}[B+z] +  m_{SY} E_{SY}[B+z] + n_{SY} M_{SS}[B+z] \nonu\\
&~{}= \mathcal{H}_{SY}[B]\\
& ~{}\quad + 2\re \int \bar{z}\Big[B_{xxxx} + 3 B_x^2 \bar B + 4 |B|^2 B_{xx}  + 2 B|B_x|^2 + B^2 \bar B_{xx} + \frac32 |B|^4 B \\
&~{} \qquad \qquad \qquad  - m_{SY}(B_{xx} + |B|^2 B) + n_{SY}B\Big]\nonu\\
&~{} \quad + 2\re\int \frac{\bar{z}}{2}\Big[ z_{4x} + 4|B|^2z_{xx} + B^2\bar{z}_{xx} + (2B\bar{B}_x+6\bar{B}B_x   ) z_x + 2BB_x\bar{z}_x \\
& ~{}\qquad \qquad \qquad + \left(  2|B_x|^2 + 2B\bar{B}_{xx}      +4\bar{B}B_{xx} + \frac92|B|^4 \right)z \\
&~{} \qquad \qquad \qquad + \left(3 B_x^2 + 4BB_{xx} + 3|B|^2B^2 \right)\bar{z}    - m_{SY}[z_{xx} +B^2\bar{z} + 2|B|^2z] + n_{SY} z \Big] \\
&~{}  \quad + \mathcal N_{SY}[z] =:\mathcal{H}_{SY}[B] +  \mathcal G_{SS}[z] + \mathcal Q_{SS}[z] + \mathcal N_{SS}[z],
\end{aligned} 
\]
with  $\mathcal N_{SY}[z]$ as defined in \eqref{Nsy}. This proves  in the SY case. The proof is complete.
\end{proof}

\bigskip

\section{Existence of critical points: Proof of Theorem \ref{TH1}}\label{Sect:4}

\medskip

In this section we prove Theorem \ref{TH1}. Recall that Theorem \ref{TH1} is a fundamental part to complete the proof of Theorem \ref{TH1a}.

\medskip

From Proposition \ref{Decomposition_Prop} (more precisely, \eqref{perEcBp_0}, \eqref{Ec_SY_0}, \eqref{Ec_KM_0} and \eqref{Ec_P_0}), we see that \eqref{perEcBp}, \eqref{Ec_SY}, \eqref{Ec_KM} and \eqref{Ec_P} are proved (and so Theorem \ref{TH1}) if we show in \eqref{lineal_lineal} that 
\be\label{Aux_00}
\mathcal G[B_X]\equiv 0,
\ee
for the choices of $m_X$ and $n_X$ given at the beginning of Section \ref{Sect:3}. Although these proofs are straightforward and painful, we present them in some detail to further checking by the reader. 

\subsection{Proof of \eqref{Aux_00} in the $SS$ case} First we have

\begin{lem}[Alternative form for \eqref{perEcBp}]\label{QODE} Let $m_{SS}= - 2(\bt^2-\al^2) $ and $n_{SS}=(\al^2+\bt^2)^2$, and let $B=B_{SS} = Q_\bt e^{i\Theta}$ be the breather solution \eqref{R} of \eqref{SS}. Then $B$ satisfies \eqref{perEcBp} if and only if $Q_\bt$ solves
\be\label{QODE0}
\begin{aligned}
&Q_{\bt}'''' + 4i\al Q_{\bt}'''  - 6\al^2 Q_{\bt}''  - 4i\al^3 Q_{\bt}' +\al^4 Q_\bt  + 8\bar{Q}_{\bt} Q_{\bt}'^2 \\
& + 14 Q_{\bt}  \bar Q_{\bt} Q_{\bt}'' +12 Q_\bt' \bar Q_\bt' Q_\bt   + 32i\al Q_{\bt}\bar Q_{\bt} Q_{\bt}' - 16 \al^2 Q_{\bt}^2 \bar Q_{\bt} + 6Q_{\bt}^2\bar{Q}_{\bt}'' \\
& + 24 Q_{\bt}^3 \bar Q_{\bt}^2 - m_{SS}\left( Q_\bt'' +2i\al Q_\bt' -\al^2 Q_\bt + 4Q_\bt^2 \bar Q_\bt  \right) +n_{SS} Q_\bt =0.
\end{aligned}
\ee
\end{lem}
\begin{proof}
See Appendix \ref{B} for a proof of this result.
\end{proof}

We continue with the proof of \eqref{Aux_00}. Replacing $m_{SS}$ and $n_{SS}$,
\be\label{QODE1}
\begin{aligned}
&Q_{\bt}'''' + 4i\al Q_{\bt}''' - 2(2\al^2+\bt^2) Q_{\bt}'' - 4i\al \bt^2 Q_{\bt}' +  \bt^2(4\al^2 +\bt^2) Q_{\bt}\\
& + 8\bar{Q}_{\bt} Q_{\bt}'^2 + 32i\al Q_{\bt}\bar Q_{\bt} Q_{\bt}'+ 14 Q_{\bt}  \bar Q_{\bt} Q_{\bt}''   \\
&  - 8(\al^2 + \bt^2)Q_{\bt}^2 \bar Q_{\bt} + 12 Q_{\bt}\bar Q_{\bt}'Q_{\bt}'  + 6Q_{\bt}^2\bar{Q}_{\bt}'' + 24 Q_{\bt}^3 \bar Q_{\bt}^2 =0.
\end{aligned}
\ee
From the third order ODE \eqref{edoQ2} satisfied by the  profile $Q_\bt$, we have
\[
\Bigg( Q_\beta''' +9 Q_\beta \bar Q_\beta Q_\beta'  + 3Q_\beta^2 \bar Q_\beta'   -\beta^2 Q_\beta'  +3i\alpha \Big( Q_\beta'' -\beta^2 Q_\beta +2 Q_\beta ^2 \bar Q_\beta  \Big) \Bigg)'=0.
\]
Therefore,
\[
\begin{aligned}
&Q_{\bt}''''  + 4 i\al Q_{\bt}''' - i\al Q_{\bt}'''  + 9\bar{Q}_{\bt}Q_{\bt}'^2 + 9 Q_\beta \bar Q_\beta Q_{\bt}''  + 15Q_{\bt} \bar Q_{\bt}' Q_{\bt}' \\
& \qquad + 3Q_{\bt}^2\bar{Q}_{\bt}''  -\bt^2Q_{\bt}'' + 3i\al Q_{\bt}'''  - 3i\al\bt^2Q_{\bt}'    + 12i\al Q_{\bt} Q_{\bt}' \bar Q_{\bt}   + 6i\al Q_{\bt}^2\bar{Q}_{\bt}' =0.
\end{aligned}
\] 
Using \eqref{edoQ2} and replacing above, we have
\[
\begin{aligned}
& Q_{\bt}'''' + 4 i\al Q_{\bt}'''   + 9\bar{Q}_{\bt}Q_{\bt}'^2 + 9 Q_\beta \bar Q_\beta Q_{\bt}''  + 15Q_{\bt} \bar Q_{\bt}' Q_{\bt}' + 3Q_{\bt}^2\bar{Q}_{\bt}'' \\
&  -\bt^2Q_{\bt}''  - 3i\al\bt^2Q_{\bt}'    + 12i\al Q_{\bt} Q_{\bt}' \bar Q_{\bt}    + 6i\al Q_{\bt}^2\bar{Q}_{\bt}' \\
&  + i\al \Big( 9 Q_\beta \bar Q_\beta Q_\beta'  + 3Q_\beta^2 \bar Q_\beta'   -\beta^2 Q_\beta'  +3i\alpha \Big( Q_\beta'' -\beta^2 Q_\beta +2 Q_\beta ^2 \bar Q_\beta  \Big)  \Big) =0.
\end{aligned}
\]
Namely
\[
\begin{aligned}
&Q_{\bt}'''' + 4 i\al Q_{\bt}'''  + 9\bar{Q}_{\bt}Q_{\bt}'^2 + 9 Q_\beta \bar Q_\beta Q_{\bt}''  + 15Q_{\bt} \bar Q_{\bt}' Q_{\bt}' + 3Q_{\bt}^2\bar{Q}_{\bt}''  \\
& -(3\al^2+ \bt^2) Q_{\bt}''   - 4 i\al\bt^2Q_{\bt}'   +21 i\al Q_{\bt} Q_{\bt}' \bar Q_{\bt}     + 9i\al Q_{\bt}^2\bar{Q}_{\bt}' +3\al^2\bt^2 Q_\bt  -6\al^2  Q_\beta^2 \bar Q_\beta=0.
\end{aligned}
\]
Comparing with \eqref{QODE1}, we just must show the following nonlinear identity satisfied by the soliton $Q_\bt$ \eqref{SolQ}:
\be\label{nlid1}
\begin{aligned}
& -(\al^2+ \bt^2) Q_{\bt}'' +   \bt^2(\al^2+\bt^2) Q_\bt - 2(\al^2 +4\bt^2)  Q_\beta^2 \bar Q_\beta + 11 i\al Q_{\bt} Q_{\bt}' \bar Q_{\bt}\\
&- 9i\al Q_{\bt}^2\bar{Q}_{\bt}' - \bar{Q}_{\bt}Q_{\bt}'^2  - 3Q_{\bt} \bar Q_{\bt}' Q_{\bt}'  + 3Q_{\bt}^2\bar{Q}_{\bt}'' + 5 Q_\beta \bar Q_\beta Q_{\bt}''
 + 24 Q_{\bt}^3 \bar Q_{\bt}^2 =0.
\end{aligned}
\ee
The proof of this nonlinear identity is direct but cumbersome: see Appendix \ref{App_nlid1} for a detailed proof.  This ends the proof of \eqref{Aux_00}.

\medskip

\subsection{Proof of \eqref{Aux_00} in the remaining cases}\label{Calculos0} The rest of proofs in the cases $SY$, $KM$ and $P$ (\eqref{Ec_SY}, \eqref{Ec_KM} and \eqref{Ec_P}) are similar to the above written, and add no new insights nor mathematical clues about the breathers themselves. For this reason, we have placed them in the Appendix \ref{Calculos}.

\bigskip

\section{Stability of the $SS$ breather. Proof of Theorem \ref{TH2}}\label{Sect:5}

This Section is devoted to the proof of Theorem \ref{TH2}. The proof requires several steps, that we represent in different subsections.

\medskip

Without loss of generality, using the scaling and space invariances of the equation, we assume $\beta=1$ and $x_2=0$.

\subsection{Continuous spectrum and nondegeneracy of the kernel} 
Let $B=B_{SS}$ be a SS breather as in \eqref{R}, and $\mathcal L_{SS}$ be the linear operator in \eqref{perEcBp_1}. By considering $z$ and $\bar z$ as independent variables, as usual, and with a slight abuse of notation, we can write $\mathcal{L}_{SS}$ as 
\begin{equation}\label{matrix_L}
 \mathcal{L}_{SS} = \left(\begin{matrix}\mathcal{L}_{SS,1} & \mathcal{L}_{SS,2} \\ \mathcal{L}_{SS,3} & \mathcal{L}_{SS,4}\end{matrix}\right),
\end{equation}
where
\[
\begin{aligned}
\mathcal{L}_{SS,1} := &~{}  \partial_x^4  + (14|B|^2+m_{SS}) \partial_x^2 + (12B\bar{B}_x + 16\bar{B}B_{x}) \partial_x\\
& ~{}  +(14\bar{B}B_{xx}+12 |B_x|^2 + 12B\bar{B}_{xx}+72|B|^4  +8m_{SS}|B|^2 + n_{SS}),
\end{aligned}
\]
\[
\mathcal{L}_{SS,2} :=     6B^2 \partial_x^2  + 12 BB_x \partial_x + (14BB_{xx} + 8B_x^2 + 48|B|^2B^2 + 4m_{SS}B^2),
\]
\[
\begin{aligned}
\mathcal{L}_{SS,3} :=&~{}  6 \bar B^2 \partial_x^2  + 12 \bar B \bar B_x \partial_x + (14 \bar B\bar B_{xx} + 8\bar B_x^2 + 48|B|^2\bar B^2 + 4m_{SS} \bar B^2) \\
=&~{} \overline{\mathcal{L}}_{SS,2},
\end{aligned}
\]
and
\[
\begin{aligned}
\mathcal{L}_{SS,4} := &~{}  \partial_x^4  + (14|B|^2+m_{SS}) \partial_x^2 + (12\bar B B_x + 16 B \bar B_{x}) \partial_x\\
& ~{}  +(14B \bar B_{xx}+12 |B_x|^2 + 12\bar B B_{xx}+72|B|^4  +8m_{SS}|B|^2 + n_{SS}) \\
=&~{}\overline{ \mathcal{L}}_{SS,1}.
\end{aligned}
\]
Note that $ \mathcal{L}_{SS}$ is Hermitian as an operator defined in $H^2(\R;\Com)$ with dense domain $H^4(\R,\Com)$. Therefore, its spectrum is real-valued. We start with the following result, essentially proved in \cite{AM}.

\begin{lem}\label{Cont_Spec} The operator $\mathcal L_{SS}$ is a compact perturbation of the constant coefficients operator
\be\label{L0}
 \mathcal{L}_{SS,0}: = \left(\begin{matrix}  \partial_x^4  + m_{SS} \partial_x^2 +n_{SS} & 0 \\ 0 &  \partial_x^4  + m_{SS} \partial_x^2 +n_{SS}  \end{matrix}\right). 
\ee
In particular, the continuous spectrum of $\mathcal L_{SS}$ is the closed interval $[(\al^2 +\bt^2)^2,+\infty)$ in the case $\beta\geq \al$, and $[ 4\al^2 \bt^2 ,+\infty)$ in the case $\beta< \al$. 
\end{lem}

Now we study the kernel of $\mathcal L_{SS}$. We have directly from \eqref{perEcBp}
\[
 \left(\begin{matrix} \partial_{x_1} B \\ \partial_{x_1}\bar B \end{matrix}\right) ,\left(\begin{matrix} \partial_{x_2} B \\ \partial_{x_2}\bar B \end{matrix}\right)\in  \ker  \mathcal L_{SS} . 
\]
Note that $ \partial_{x_1} B= i\al B$, which is nothing but the instability direction associated to the $U(1)$ invariance.  Moreover, following the ideas in \cite{AM}, based on the 1-D character of the ODEs involved, we have

\begin{lem}[Nondegeneracy]\label{Nondege}
\[
\ker \mathcal L_{SS} =\spawn \left\{  \left(\begin{matrix} \partial_{x_1} B \\ \partial_{x_1}\bar B \end{matrix}\right) ,\left(\begin{matrix} \partial_{x_2} B \\ \partial_{x_2}\bar B \end{matrix}\right)\right\}.
\]
\end{lem}

\begin{rem}
The proof of this result follows the ideas in \cite{AM}, but not every vector valued linear operator around breathers will follow the same idea of proof. 
See \cite{AMP1} for a case where the argument in \cite{AM} does not apply. We will benefit here from the fact that the second component of $\mathcal L_{SS}[z]=0$ corresponds to the complex conjugate of the first one.
\end{rem}

\begin{proof}[Proof of Lemma \ref{Nondege}]
 Let $B_i := \partial_{x_i} B,~i=1,2,$ and  $z\in H^4$ be such that  $\mathcal L_{SS} [z]=0$, such that $\{z,B_1,B_2\}$ are linearly independent. 
For all large $x$ we have that $\mathcal L_{SS} $ behaves like $ \mathcal{L}_{SS,0}$ in \eqref{L0}, which determines the large $x$ 
behavior of solutions of $\mathcal L_{SS}[z]=0$. Fortunately, $\mathcal{L}_{SS,0}$ is a diagonal operator with the same components 
(this is not the case in \cite{AMP1}), so we only need to consider the first one, the second one being identical since it corresponds
to the complex conjugate. As in \cite{AM}, since  $\mathcal{L}_{SS,0}$ has constant coefficients, we have that such additional element 
of the kernel $z$ must have the large $x$ behavior
\begin{equation}\label{decaimiento_L0}
z(x) \sim e^{\pm x \pm i\al}.
\end{equation}
Among these, there are only two linearly independent possible behaviors as $x\to\infty$ representing localized data: $e^{-x}\sin(\al x)$ and $e^{-x}\cos(\al x)$, the same number as the set $\{B_1,B_2\}$. This implies that $\dim\ker\mathcal L_{SS} \leq 2$, a contradiction. This proves the result.
\end{proof}

%

\begin{lem}[Existence of negative directions]\label{negative}
Let $B=B_{SS}$ be a SS breather as in \eqref{R}, and $\mathcal L_{SS}$ be the linear operator in \eqref{perEcBp_1}. Then we have
\be\label{DaB}
\mathcal L_{SS} \partial_{\al} B= -4\al  (B_{xx} + 4 |B|^2 B) - 4\al(\al^2+\bt^2) B, 
\ee
and
\be\label{DbB}
\mathcal L_{SS} \partial_{\beta} B= 4\bt  (B_{xx} + 4 |B|^2 B) - 4\beta(\al^2+\bt^2) B.
\ee
Additionally, we have
\be\label{DaB_DbB}
\mathcal L_{SS}\left(B_0 \right) =  -B,\qquad B_0:= \frac{ \beta \partial_{\al} B +\alpha \partial_{\beta} B}{ 8\al \bt (\al^2+\bt^2)},
\ee
and
\be\label{L_DaB_DbB}
\re \int  \overline{\left( \beta \partial_{\al} B +\alpha \partial_{\beta} B \right) }\mathcal L_{SS}\left( \beta \partial_{\al} B +\alpha \partial_{\beta} B \right) = - 4\al^2 \bt (\al^2+\bt^2) \int |Q|^2<0. 
\ee
\end{lem}

\begin{rem}
Lemma \ref{negative} shows that $\beta \partial_{\al} B +\alpha \partial_{\beta} B$ is a negative direction for the functional $\mathcal L_{SS}$. 
\end{rem}

\begin{proof}
The proofs of  \eqref{DaB} and \eqref{DbB} are direct from \eqref{perEcBp}. 
The proof of \eqref{DaB_DbB} follows from \eqref{DaB} and \eqref{DbB}. Finally, from \eqref{DaB_DbB}, \eqref{R}, \eqref{SolQ} and \eqref{massBreSS},
\[
\begin{aligned}
& \re \int  \overline{\left( \beta \partial_{\al} B +\alpha \partial_{\beta} B \right) } \mathcal L_{SS}\left( \beta \partial_{\al} B +\alpha \partial_{\beta} B \right) \\
 &\qquad = ~{}  - 8\al \bt (\al^2+\bt^2)\re  \int \overline{\left( \beta \partial_{\al} B +\alpha \partial_{\beta} B \right)} B\\
& \qquad= ~{} - 4\al \bt (\al^2+\bt^2) \left( \beta  \partial_{\al} \int  |B|^2  +\alpha \partial_{\beta} \int |B|^2 \right) \\
&\qquad= ~{} - 4\al^2 \bt (\al^2+\bt^2) \int |Q|^2.
\end{aligned}
\]
Since from \eqref{massSolSS},  the last integral is a  positive constant, this proves \eqref{L_DaB_DbB}.
\end{proof}

It turns out that the most important consequence of the previous result is the fact that $\mathcal L_{SS}$ possesses \emph{only one negative eigenvalue}. Indeed, in order to prove that result, we follow the Greenberg and Maddocks-Sachs strategy \cite{Gr,MS}, applied this time to the linear operator $\mathcal L_{SS}$. This time, we need some important changes.

\begin{lem}[Uniqueness criterium, see also \cite{Gr,MS}]\label{Wr1} Let $B=B_{SS}$ be any SS breather \eqref{R}, and let 
$B_1=\partial_{x_1}B$, $B_2=\partial_{x_2}B$ be the corresponding kernel of the operator $\mathcal L_{SS}$. Then $\mathcal L_{SS}$ has 
\[
\sum_{x\in \R} \dim   \left( \ker W[B_1, B_2] \cap \ker  W[\partial_x B_1, \partial_x B_2] \right)(x)
\] 
negative eigenvalues, counting multiplicity. Here,  $W[A_1, A_2]$ is the Wronskian matrix of the functions $A_1$ and $A_2$,
\be\label{WM}
W[A_1, A_2] (x) := \left[ \begin{array}{cc} A_1 &A_2 \\  \overline{A_1} &  \overline{A_2}  \end{array} \right] (x).
\ee
\end{lem}

\begin{proof}
This result is essentially contained in \cite[Theorem 2.2]{Gr}, where the finite interval case was considered. As shown in several articles (see e.g. \cite{MS,HPZ}), the extension to the real line is direct. Here we need some changes, 
that we sketch below. 

\medskip

Fix $\theta \in \R $. Let us consider  the eigenvalue problem
\be\label{Ltheta}
\mathcal L_{SS} z = \la(\theta) z, \quad z \in H_\theta,
\ee
where
\be\label{H_theta}
H_\theta := \{ z \in H^4((-\infty,\theta),\Com)   \ : \     z(\theta) =z_x(\theta)   =0 \}.
\ee
With a slight abuse of notation we will denote by $\mathcal L_{SS,\theta}$ the unbounded operator $\mathcal L_{SS}$ with domain $H_\theta$ and values in $L^2(\R)$. Clearly for any $\theta \in \R$, $\mathcal L_{SS,\theta}$ is self-adjoint. Moreover, its continuous spectrum is given by $\sigma_c(\mathcal L_{SS})$, see Lemma \ref{Cont_Spec}. Also, for any $\theta \in \R$, $\mathcal L_{SS,\theta}$ is bounded below. 

\medskip

For any $\theta \in \R$, the number of eigenvalues of $\mathcal L_{SS,\theta}$ is nonempty. We define by $n(\theta) \geq 1$ the number of eigenvalues of $\mathcal L_{SS,\theta}$.  Notice that $n(\theta)$ is never zero, since the first eigenvalue $\la_{1}(\theta)$ always exists.


%
%
\medskip

Recall that the set $(\la_j( +\infty))_{j}$ represent the eigenvalues of $\mathcal L_{SS}$ in $\R$. Our objective is to determine the number of indices $j$ such that $\la_j( +\infty)<0$.  We remark that we know that there is at least one and  at most a finite number of negative eigenvalues for $\mathcal L_{SS}$.

\medskip

Let  $\theta \in \R$ and $\la_1(\theta)\leq \la_2(\theta) \leq \cdots  \leq  \la_{n(\theta)}$, be the  eigenvalues of $\mathcal L_{SS,\theta}$, counted as many times according to their multiplicity. Note that $n(\theta)$ may vary but it is always finite and $\geq 1$. Moreover, by standard spectral arguments, the $\la_j(\theta)$ are {\bf continuous} and {\bf strictly} decreasing functions of $\theta$, with $\la_j(\theta) \geq \la_j( +\infty)$.

\medskip

Fix now $j \in \{1,\ldots, n(\theta) \}$. There is \emph{at most} one $\theta_j$ such that $\la_j(\theta_j) = 0$, and such $\theta_j$ exists if and only if $\la_j(+\infty)<0$.  Since the set of eigenvalues $\{\la_k(+\infty) : \la_k(+\infty) <0\}$ is finite and nonempty,  we conclude the number of negative eigenvalues of $\mathcal L_{SS}$ equals the number of points $\theta $ such that $\la_j(\theta) = 0$, where $j \in \{1,\ldots, n(\theta) \}$. 
And  for fixed $\theta \in \R$, the multiplicity of $0$ as an eigenvalue of $\mathcal L_{SS,\theta}$ is equal to the number of indices $j$ such that $ \la_j (\theta) =0$.

\medskip

Now, let us characterize 0 as an eigenvalue of $\mathcal L_{SS,\theta}$. Indeed, we have that $0$ is an eigenvalue of $\mathcal L_{SS,\theta}$ if and only if there are constants $c_1,c_2 \in \Com$, not all zero, such that 
\[
Z(x):= c_1 B_1(x) + c_2 B_2(x),  
\]
is nontrivial and belongs to $H_\theta$ (note that any other linearly independent element of the vector space $\mathcal L_{SS}Z=0$ is exponentially increasing as $x\to -\infty$, see \eqref{decaimiento_L0}). 
Consequently, one has
\[
\mathcal L_{SS} \left(\begin{array}{c} Z \\ \overline{Z} \end{array}  \right) = \left(\begin{array}{c} 0 \\ 0\end{array}  \right),
\]
with $\mathcal L_{SS}$ the matrix operator in \eqref{matrix_L}. By Lemma \ref{Nondege}, we have for all $x\in\R$,
\[
\left(\begin{array}{c} c_1 B_1 + c_2 B_2 \\ \bar c_1 \overline{B}_1 + \bar c_2 \overline{B}_2 \end{array}  \right) = \tilde c_1  \left(\begin{matrix} B_1 \\ \overline{B}_1 \end{matrix}\right)  + \tilde c_2 \left(\begin{matrix} B_2 \\ \overline{B}_2 \end{matrix}\right),
\]
for some $\tilde c_1, \tilde c_2\in \Com$. We conclude that $\tilde c_1= c_1\in\R$ and $\tilde c_2= c_2\in\R$. Since $Z\in H_\theta$,  for constant $c_1,c_2$ not both equal to zero and real-valued,
\be\label{49}
 \left(\begin{array}{ccc} B_1(\theta) & B_2(\theta) \\  \overline{B_1}(\theta) & \overline{B_2}(\theta) \end{array}  \right) \left(\begin{array}{c} c_1 \\ c_2  \end{array}  \right)  = \left(\begin{array}{c} 0 \\ 0 \end{array}  \right).
\ee
Additionally, taking space derivative and using the definition of $H_\theta$ in \eqref{H_theta} we have
\be\label{48}
 \left(\begin{array}{ccc}  \partial_x B_1(\theta) & \partial_x B_2(\theta) \\ \partial_x \overline{B_1}(\theta) & \partial_x \overline{B_2}(\theta)   \end{array}  \right) \left(\begin{array}{c} c_1 \\ c_2  \end{array}  \right)  = \left(\begin{array}{c} 0 \\ 0 \end{array}  \right).
\ee
Summing on $x\in\R$ we  conclude.  Notice that the sum is indeed finite, because of the finite number of negative eigenvalues of $\mathcal L_{SS}$.
\end{proof}

In what follows, we compute the Wronskian \eqref{49} in the explicit SS case. We easily have from \eqref{R},
\be\label{Wsimpl}
\det W[B_1, B_2](x) =  \det\left(\begin{array}{ccc} i\al Q_\eta e^{i\Theta} & Q_\eta' e^{i\Theta} \\  -i\al \overline{Q_\eta }e^{-i\Theta} & \overline{Q_\eta'} e^{-i\Theta} \end{array}  \right) =2i\al \re \left\{ \overline{Q}_\eta Q_\eta' \right\}.
\ee
We have, for $\eta=: a+ib = \frac{\al^2}{\al^2+\beta^2} -\frac{\al\bt i}{\al^2+\beta^2}$, and $u:=e^{2x}>0$,
\[
\re \left\{ \overline{Q}_\eta Q_\eta' \right\} =\frac{u \left(a^4+4 a^3 u -2 a^2 \left(u-b^2\right)-4 a u\left(u^2-b^2\right)+b^4-2 b^2 u+2 u^3 -u^4\right)}{\left(a^2+b^2+2u+u^2\right)^3}.
\]
Let us find a positive root $u$ for the term in the numerator above. First of all, we have
\be\label{polU4}
\begin{aligned}
& a^4+4 a^3 u -2 a^2 \left(u-b^2\right)-4 a u\left(u^2-b^2\right)+b^4-2 b^2 u+2 u^3 -u^4 \\
& \qquad = \frac{\al^4 (u-1) (u+1)^3+2 \al^2 \bt^2 u \left(u^3+1\right)+\bt^4 (u-2) u^3}{\left(\al^2+\bt^2\right)^2}.
\end{aligned}
\ee
The solutions to this equation equals zero (which will imply infinitely many $c_1,c_2$ as solutions to \eqref{49}) are
\[
u_{1,\pm}:= \frac{\pm \al}{\sqrt{\al^2+\bt^2}}, \quad u_{2,\pm}:= \frac{\bt^2-\al^2 \pm \beta \sqrt{\bt^2-3\al^2}}{\al^2+\bt^2}.
\]
Clearly $u_{1,-}$ is not a valid solution. Now, if $\bt^2-3\al^2<0$, the only valid positive root is $u_{1,+}=e^{2x}$. It is not difficult to see in this case that 
\[
 \dim \ker W[\partial_{x_1}B, \partial_{x_2}B] \left( \frac12\log u_{1,+} \right)  =1.
\]
Assume now  $\bt^2 > 3\al^2$ in \eqref{R}. We have now at least a second root, $u_{2,+}$, always positive. Additionally, $u_{2,-}>0$ means
\[
\begin{aligned}
(\bt^2-\al^2)^2 > \beta^2 (\bt^2-3\al^2) \iff  &~{} -2\bt^2 \al^2 +\al^4 > -3\al^2 \bt^2\\
 \iff  &~{} \bt^2  +\al^2 > 0,
\end{aligned}
\]
so both  $u_{2,\pm}$ are positive, therefore, three roots are present in this case. In all these cases, $ \dim \ker W[\partial_{x_1}B, \partial_{x_2}B] \left( \frac12\log u_{2,\pm} \right)  =1$.  

\medskip

Now we impose the second condition on the derivatives, i.e. \eqref{48}. Recall that from the previous part, at $u_{1,+}$ and $u_{2,\pm}$ there are infinitely many $c_1,c_2$ solutions of \eqref{49}, essentially a linear subspace of dimension 1. From \eqref{48} we have at $x=\theta$,
\be\label{50}
 \left(\begin{array}{ccc} i\al (Q_\eta' +i\al Q_\eta) e^{i\Theta} & (Q_\eta'' +i\al Q_\eta')e^{i\Theta} \\  -i\al (\overline{Q_\eta'}-i\al \overline{Q_\eta})e^{-i\Theta} & (\overline{Q_\eta''} - i\al \overline{Q_\eta'})e^{-i\Theta} \end{array}  \right)\left(\begin{array}{c}  c_1 \\ c_2  \end{array}  \right)  = \left(\begin{array}{c} 0 \\ 0 \end{array}  \right).
 \ee
A necessary condition to satisfy the previous equation with $c_1,c_2$ not both zero is that at $x=\theta$ we have
\[
  (Q_\eta' +i\al Q_\eta) (\overline{Q_\eta''} - i\al \overline{Q_\eta'}) + (Q_\eta'' +i\al Q_\eta') (\overline{Q_\eta'}-i\al\overline{Q_\eta}) =0.
\]
The previous identity simplifies to
\[
  2\re\{ Q_\eta'' \overline{Q_\eta'} \}  + 2\al  \ima \{Q_\eta'' \overline{Q_\eta}\} +2\al^2 \re \{ Q_\eta'\overline{Q_\eta}\}=0.
\]
From \eqref{Wsimpl} we have the last term in the previous identity equals zero. 
On the other hand, after some computations (see Appendix \ref{Calculos00}), one has
\be\label{reimQu1}
\re\{ Q_\eta'' \overline{Q_\eta'} \}\Big|_{x=\frac12\log u_{1,+}} =\ima \{Q_\eta'' \overline{Q_\eta}\}\Big|_{x=\frac12\log u_{1,+}}=0.
\ee
 However, 
\be\label{reimQu2}
\re\{ Q_\eta'' \overline{Q_\eta'} \} \left( \frac12\log u_{2,\pm}\right)   + \al  \ima \{Q_\eta'' \overline{Q_\eta}\} \left( \frac12\log u_{2,\pm}\right) \neq 0.
\ee

 For explicit computations of these real and imaginary parts, see Appendix \ref{Calculos00}.

\medskip

 Finally, we consider the degenerate case $\bt^2 = 3\al^2$. In this case $u_{1,+}=u_{2,\pm} = \frac12$ is the only (triple) root. We have from \eqref{speeds} $\eta= \frac14(1-i\sqrt{3})$ and we easily have 
\[
 \dim \ker W[B_1,B_2] \left( \frac12\log \frac12 \right)  =1, \quad  \dim \ker W[ \partial_x B_1, \partial_x B_2] \left( \frac12\log \frac12 \right)  =1,
\]
because at least one element in the matrix is nonzero. This concludes the proof.

\medskip

The following result summarizes our findings:

\begin{lem}[Negative eigenvalues of $\mathcal L_{SS}$]
Let $B=B_{SS}$ be a SS breather with parameters $\al,\bt>0$, and let $\mathcal L_{SS}$ be the associated linearized operator \eqref{perEcBp_1}. Then  $\mathcal L_{SS}$ has always only one negative eigenvalue.
\end{lem}


\medskip

In what follows, we define as $B_{-1}$ the unique eigenfunction associated to the unique negative eigenvalue, such that $\|B_{-1}\|_{L^2}=1$. We have

\begin{prop}[Coercivity]\label{PropOrtog} Let $B=B_{SS}$ be a Sasa-Satsuma breather, and $\partial_{x_1}B, \partial_{x_2}B$ the corresponding kernel of the associated operator $\mathcal L_{SS}$. There exists $\mu_0>0$, depending on $\al,\bt$ only, such that, for any  $ z\in H^2(\R)$ satisfying
\be\label{OrthoK}
\re\int \overline{\partial_{x_1}B}  z =\re \int \overline{\partial_{x_2}B} z =0,
\ee
one has
\be\label{Coer}
\re \int \overline{z}\mathcal L_{SS} z \geq \mu_0\| z\|_{H^2(\R)}^2 -\frac{1}{\mu_0}\Big(\re \int  z \overline{B} \Big)^2.
\ee
\end{prop}

\begin{proof}
For the sake of simplicity, we denote $B_j:= \partial_{x_j}B$.  Indeed, it is enough to prove that, under the conditions (\ref{OrthoK}) and the additional orthogonality condition $\re \int  z \bar B =0$, one has
\[
\re \int \overline{z}\mathcal L_{SS} z \geq \mu_0\| z\|_{H^2(\R)}^2.
\]
Indeed, note that from (\ref{DaB_DbB}), the function $B_0$ satisfies $\mathcal L_{SS}[B_0] = - B$, and from (\ref{L_DaB_DbB}),
\be\label{posB0}
\re\int \overline{B}_0 B = -\re\int \overline{B}_0 \mathcal L_{SS}[B_0]  > 0.
\ee
The next step is to decompose $z$ and $B_0$ in $\spawn (B_{-1}, B_1,B_2)$ and the corresponding orthogonal subspace. One has
\[
z =\tilde z +  m B_{-1}, \quad  B_0=  b_0 +  n  B_{-1} + p_1 B_1 + p_2 B_2, \quad m,n, p_1,p_2\in \Com,
\]
where 
\[
\begin{aligned}
\re\int  \tilde z  \overline{B}_{-1} =\re\int  \tilde z \overline{B}_1=\re\int  \tilde z \overline{B}_2 =&~{}0, \\
\re\int    b_0 \overline{B}_{-1} =\re\int   b_0 \overline{B}_1=\re\int   b_0 \overline{B}_2 =&~{}0. 
\end{aligned}
\]
Note in addition that 
\[
\re\int  \overline{B}_{-1}B_1 =\re\int  \overline{B}_{-1}B_2 =0.
\]
From here and the previous identities we have
\be\label{Qzz}
\begin{aligned}
\re \int \overline{z}\mathcal L_{SS} z =&~{} \re \int (\mathcal L_{SS} \tilde z - m\la_0^2 B_{-1})\overline{(\tilde z +m B_{-1}) } \\
=&~{} \re \int \overline{\tilde z}\mathcal L_{SS} \tilde z -m^2 \la_0^2. 
\end{aligned}
\ee
Now, since $\mathcal L_{SS}[B_0] =-B$ (see \eqref{DaB_DbB}), one has 
\be\label{zB}
\begin{aligned}
0  =&~{}\re \int  \overline{z} B = -\re \int   \overline{z}\mathcal L_{SS}[B_0]  =\re \int  \mathcal L_{SS}[\tilde z +m B_{-1}]  \overline{B}_0 \\
 = &~{} \re \int  (\mathcal L_{SS}[\tilde z] -m\la_0^2 B_{-1})\overline{( b_0 +n B_{-1} + p_1 B_1 + p_2 B_2)} \\
 =&~{}   \re \int  \mathcal L_{SS}[\tilde z]  \overline{b_0} -mn\la_0^2.
\end{aligned}
\ee
On the other hand, 
\be\label{B0B}
\begin{aligned}
\re \int  \overline{B}_0 B =&~{}   -\re \int  \overline{B}_0 \mathcal L_{SS}[B_0]  \\
= &~{} - \re \int  \overline{( b_0 + n B_{-1})} (\mathcal L_{SS}[b_0] -n\la_0^2 B_{-1}) \\
=&~{} -\re \int    \overline{b}_0 \mathcal L_{SS} b_0 + n^2 \la_0^2. 
\end{aligned}
\ee
Replacing (\ref{zB}) and (\ref{B0B}) into (\ref{Qzz}), we get
\be\label{deco}
\re \int \overline{z}\mathcal L_{SS} z = \re \int \overline{\tilde z}\mathcal L_{SS} \tilde z -\frac{\displaystyle{\Big(\re \int  \mathcal L_{SS}[\tilde z] \overline{b}_0\Big)^2}}{\displaystyle{\re\int  B_0\overline{B} + \re \int \overline{b_0}\mathcal L_{SS} b_0}}.  
\ee
Note that both quantities in the denominator are positive. Additionally, note that if $\tilde z =\la b_0$, with $\la\neq 0$, then
\[
\Big(\re\int  \mathcal L_{SS}[\tilde z] \overline{b}_0 \Big)^2 = \re \int \overline{\tilde z}\mathcal L_{SS} \tilde z ~ \re \int \overline{b_0}\mathcal L_{SS} b_0.
\]
In particular, if $\tilde z =\la b_0$,
\be\label{condit}
\frac{\displaystyle{\Big(\re \int  \mathcal L_{SS}[\tilde z] \overline{b}_0\Big)^2}}{\displaystyle{\re\int  B_0\overline{B} + \re \int \overline{b_0}\mathcal L_{SS} b_0 }} \leq a\,   \re \int \overline{\tilde z}\mathcal L_{SS} \tilde z, \quad 0<a<1.
\ee
In the general case, using the orthogonal decomposition induced by the scalar product $(\mathcal L_{SS} \cdot ,\cdot)_{L^2} $ on $\spawn (B_{-1}, B_1,B_2)$, we get the same conclusion as before. Therefore, we have proved (\ref{condit}) for all possible $\tilde z$. 

\medskip

Finally, replacing in (\ref{deco}) and (\ref{Qzz}), $\re \int \overline{z}\mathcal L_{SS} z \geq (1-a) \re \int \overline{\tilde z}\mathcal L_{SS} \tilde z  \geq 0$,  and $ \re \int \overline{\tilde z}\mathcal L_{SS} \tilde z  \geq m^2 \la_0^2$. We have, for some $C>0$,
\[
\begin{aligned}
\re \int \overline{z}\mathcal L_{SS} z  \geq &~{}  (1-a) \re \int \overline{\tilde z}\mathcal L_{SS} \tilde z   \\
\geq &~{} \frac 12 (1-a) \re \int \overline{\tilde z}\mathcal L_{SS} \tilde z + (1-a)m^2 \la_0^2 \\
\geq &~{}  \frac 1C(2 \|\tilde z\|_{H^2(\R)}^2 +  2m^2 \|B_{-1}\|_{H^2(\R)}^2) \geq  \frac 1C\|z\|_{H^2(\R)}^2.
\end{aligned}
\]
\end{proof}
%
%
%
%
%
%

\subsection{End of proof} We shall prove now the following explicit version of Theorem \ref{TH2}:

\begin{thm}[Explicit nonlinear stability of SS breathers]\label{TH2a} Let $B=B_{SS}$ be a SS breather. 
Assume that $u_0\in H^2(\R;\Com)$ is such that
\[
\| u_0 - B\|_{H^2}<\eta,
\]
for some $\eta$ sufficiently small. Then there exists $K>0$ and shifts $x_1(t),x_2(t)\in\R$ as in \eqref{R} such that 
\[
\sup_{t\in\R }\| u(t) - B(t; x_1(t),x_2(t))\|_{H^2}<K\eta.
\]
Moreover, one has $\sup_{t\in\R } |x_j'(t)| \lesssim K\eta.$
\end{thm}

We prove the theorem only for positive times, since the negative time case is completely analogous. From the continuity of the SS flow for $H^2(\R)$ data, there exists a time $T_0>0$ and continuous parameters $x_1(t), x_2(t)\in \R$, defined for all $t\in [0, T_0]$, and such that the solution $u(t)$ of the Cauchy problem for the SS equation \eqref{SS}, with initial data $u_0$, satisfies
\be\label{F0}
\sup_{t\in [0, T_0]}\big\| u(t) - B_{SS}(t; x_1(t),x_2(t)) \big\|_{H^2(\R)}\leq 2 \eta.
\ee
The idea is to prove that $T_0 =+\infty$. In order to do this, let $K^*>2$ be a constant, to be fixed later. Let us suppose, by contradiction, that the \emph{maximal time of stability} $T^*$, namely
\begin{align}\label{Te}
T^* &:=   \sup \Big\{ ~T>0 \, \big| \, \hbox{ for all } t\in [0, T], \, \hbox{ there exist } \tilde x_1(t), \tilde x_2(t) \in \R  \hbox{ such that } \nonu \\
&  \qquad \qquad \sup_{t\in [0, T]}\big\| u(t) - B_{SS}(t; \tilde x_1(t), \tilde x_2(t)) \big\|_{H^2(\R)}\leq K^* \eta \Big\},
\end{align}
is finite. It is clear from (\ref{F0}) that $T^*$ is a well-defined quantity. Our idea is to find a suitable contradiction to the assumption $T^*<+\infty.$

\medskip

By taking $\eta_0$ smaller, if necessary, we can apply a well known theory of modulation for the solution $u(t)$.

\begin{lem}[Modulation and orthogonality] \label{Mod} 
Let $B=B_{SS}$ be a SS breather as in \eqref{R}. There exists $\eta_0>0$ such that, for all $\eta\in (0, \eta_0)$, the following holds. There exist $C^1$ functions $x_1(t)$, $x_2(t) \in \R$, defined for all $t\in [0, T^*]$, and such that 
\be\label{z}
z(t) := u(t) - {B} (t), \quad B(t,x) := B(t,x; x_1(t),x_2(t)) 
\ee
satisfies, for $t\in [0, T^*]$,
\be\label{Ortho}
\re \int  \overline{\partial_{x_1} B}(t; x_1(t),x_2(t)) z(t) =\re \int \overline{\partial_{x_2} B} (t; x_1(t),x_2(t)) z(t)=0.
\ee
Moreover, one has
\be\label{apriori}
\|z(t)\|_{H^2(\R)} + |x_1'(t)| + |x_2'(t)| \leq K K^* \eta, \quad \|z(0)\|_{H^2(\R)} \leq K\eta,
\ee
for some constant $K>0$, independent of $K^*$. 
\end{lem}

\begin{proof}
For simplicity, we denote $B_j:= \partial_{x_j}B$. The proof of this result is a classical application of the Implicit Function Theorem.  Let
\[
J_j(u(t), x_1,x_2) :=\re \int (u(t,x) - B (t,x; x_1,x_2)) \overline{B_j}(t,x; x_1,x_2)dx, \quad j=1,2.
\]
It is clear that $J_j(B(t ;x_1,x_2),x_1,x_2) \equiv 0,$ for all $x_1,x_2\in \R$. On the other hand, one has for $j,k=1,2$,
\[
\partial_{x_k} J_j(u(t), x_1,x_2)\Big|_{(B(t),0,0)} =  -\re \int \overline{B_k (t,x; 0,0)} B_j(t,x; 0,0)dx.
\]
Let $J$ be the $2\times 2$ matrix with components $ J_{j,k} := (\partial_{x_k}J_j)_{j,k=1,2}$. From the identity above, one has
\[
\det J = -\Big[ \int |B_1|^2  \int |B_2|^2 -  (\re \int \overline{B_1}B_2)^2\Big](t;0,0),
\]
which is different from zero from the Cauchy-Schwarz inequality and the fact that $B_1$ and $B_2$ are not parallel for all time. Therefore, in a small $H^2$ neighborhood of $B(t; 0,0)$, $t\in [0,T^*]$ (given by the definition of (\ref{Te})), it is possible to write the decomposition (\ref{z})-(\ref{Ortho}).

\medskip

Now we look at the bounds (\ref{apriori}). The first bounds are consequence of the decomposition itself and the equations satisfied by the derivatives of the scaling parameters, after taking time derivative in (\ref{Ortho}) and using that $\det J\neq 0$. 
\end{proof}

From the conservation laws for $\mathcal H_{SS}$ and Proposition \ref{Decomposition_Prop},
\be\label{Hut}
\mathcal H_{SS}[u](t) = \mathcal H_{SS}[B](t) + \mathcal Q_{SS}[z](t) + \mathcal N_{SS}[z](t).
\ee 
Note that $|\mathcal N_{SS}[z](t)|\leq K \|z(t)\|_{H^1(\R)}^3$. On the other hand, by the translation invariance in space,
\[
 \mathcal H_{SS}[B](t) =\mathcal H_{SS}[B](t=0) =\hbox{constant}.
\]
Indeed, from (\ref{R}), we have
\[
B(t,x; x_1(t), x_2(t)) = B(t-t_0(t), x-x_0(t)),
\]
for some specific $t_0,x_0$. Since $\mathcal H_{SS}$ involves integration in space of polynomial functions on $B, B_x$ and $B_{xx}$, we have 
\[
\begin{aligned}
 \mathcal H_{SS}[B(t, \cdot ; x_1(t),x_2(t))] = &~{}\mathcal H_{SS}[B(t -t_0(t), \cdot -x_0(t); 0,0)] \\
 =&~{}   \mathcal H_{SS}[B(t-t_0(t), \cdot ; 0,0)].
 \end{aligned}
\] 
Finally, $ \mathcal H_{SS}[B(t-t_0(t), \cdot ; 0,0)] = \mathcal H_{SS}[B(\cdot , \cdot ; 0,0)] (t-t_0(t))$. Taking time derivative,
\[
\begin{aligned}
 \partial_t \mathcal H_{SS}[B(t, \cdot ; x_1(t),x_2(t))] = &~{}  \mathcal H_{SS}'[B(\cdot , \cdot ; 0,0)] (t-t_0(t)) \times (1-t_0'(t)) \\
 \equiv &~{} 0,
  \end{aligned}
\]
hence $\mathcal H_{SS}[B]$ is constant in time. Now we compare (\ref{Hut}) at times $t=0$ and $t\leq T^*$. We have 
\[
\begin{aligned}
\mathcal Q_{SS}[z](t) \leq  &~{}  \mathcal Q_{SS}[z](0) + K\|z(t)\|_{H^2(\R)}^3 +K\|z(0)\|_{H^2(\R)}^3 \\
\leq &~{} K\|z(0)\|_{H^2(\R)}^2+K\|z(t)\|_{H^2(\R)}^3. 
 \end{aligned}
\]
Additionally, from (\ref{OrthoK})-(\ref{Coer}) applied this time to the time-dependent function $z(t)$, which satisfies (\ref{Ortho}), we get
\begin{align}
\|z(t)\|_{H^2(\R)}^2&  \leq     K\|z(0)\|_{H^2(\R)}^2 + K\|z(t)\|_{H^2(\R)}^3 + K\left| \re\int  B(t) \overline{z}(t)\right|^2\nonu\\
& \leq  K\eta^2 + K(K^*)^3 \eta^3 +K\left| \re\int  B(t) \overline{z}(t)\right|^2. \label{Map}
\end{align}

\noindent
{\bf Conclusion of the proof.} Using the conservation of mass $M_{SS}$ in (\ref{M}), we have, after expanding $u=B+z$,
\begin{align*}
\left| \re \int   B(t) \overline{z}(t) \right| &  \leq  K\left| \re \int   B(0) \overline{z}(0) \right| +K\|z(0)\|_{H^2(\R)}^2+ K\|z(t)\|_{H^2(\R)}^2 \nonu \\
& \leq    K (\eta + (K^*)^2 \eta^2), \quad \hbox{ for each $t\in [0, T^*]$.}
\end{align*}
Replacing this last identity in (\ref{Map}), we get
\[
\|z(t)\|_{H^2(\R)}^2 \leq  K \eta^2 ( 1+ (K^*)^2 \eta^3) \leq \frac 12 (K^*)^2 \eta^2 ,
\]
by taking $K^*$ large enough. This last fact contradicts the definition of $T^*$ and therefore the stability property holds true. 

\bigskip

\section{Instability of the Peregrine bilinear form. Proof of Theorem \ref{TH4}}\label{Sect:7}

We start out with a simple lemma.
\begin{lem}
Let $X=P$ or $KM$,  $B=B_{X}$ be the Peregrine and Kuznetsov-Ma breathers from \eqref{P}-\eqref{KM}, and $F=F_{X}$ given by \eqref{F_NLS}. Then we have
\be\label{F_B}
F[B_P]=0, \qquad  F[B_{KM}]= \frac45 \beta^5.
\ee
\end{lem}

\begin{proof}
We deal first with the Peregrine case. Since $F$ is a conserved quantity, we have from \eqref{P} that $F[B_{P}]=\lim_{t\to +\infty} F[B_P] = \lim_{t\to +\infty} F[e^{it}].$ Now, from \eqref{F_NLS}
\[
F[e^{it}] =  \int 0 = 0.
\]
This proves the first identity in \eqref{F_B}. Now we deal with $F[B_{KM}]$. Since $F$ is conserved, we can assume $t=\frac{\pi}{2\alpha}$. Then we have from \eqref{KM} and \eqref{F_NLS},
%
\[
\begin{aligned}
F[B_{KM}] = &~ {} F\left[ e^{i \frac{\pi}{2\alpha}} \left( 1-  \frac{ i\sqrt{2}\beta  }{ \cosh(\beta x) } \right) \right] \\
=&~{} F\left[ 1-  \frac{ i\sqrt{2}\beta  }{  \cosh(\beta x) } \right] \\
= &~ {}  \widetilde F_{KM}\left[  \frac{ \sqrt{2}\beta  }{  \cosh(\beta x) } \right] =\beta^5 \widetilde F_{KM}\left[  \frac{ \sqrt{2}  }{  \cosh x } \right], 
\end{aligned}
\]
where
\be\label{tilde_F_NLS}
\widetilde F_{KM}[u]:= \int \Big( u_{xx}^2 -5 u^2 u_x^2   + \frac12 u^6 \Big). 
\ee
After some lengthy computations, we see that $ \widetilde F_{KM}\left[  \frac{ \sqrt{2}  }{  \cosh x } \right] =\frac45$, so that \eqref{F_B} is proved.
\end{proof}

\begin{rem}
Note that from Remark \ref{CL} we also have $M_P[B_P]=E_P[B_P]=0$. Additionally, $M_{KM}[B_{KM}] = 4\beta$ and $E_{KM}[B_{KM}] =-\frac 83\bt^3$. Consequently, $H_{KM}$ defined in \eqref{H} satisfies
\[
\begin{aligned}
\mathcal H_{KM}[B_{KM}] = &~{} F_{KM}[B_{KM}]  + m_{KM} E_{KM}[B_{KM}]  + n_{KM} M_{KM}[B_{KM}] \\
=&  ~{} \frac45 \beta^5 +\frac{8}{3}\beta^5, 
\end{aligned}
\]
which is strictly positive for $\beta>0$. 
\end{rem}

\begin{lem}
Let $B=B_P$ be the Peregrine breather \eqref{P} and $z\in H^2(\R;\Com)$ be a small perturbation. We have
\be\label{F_P}
\mathcal H_P[B+z]=\frac12 \int  (|z_{xx}|^2 -|z_x|^2 - (e^{it}\bar{z}_x)^2 ) +O(\|z\|_{H^1}^3) +o_{t\to +\infty}(1).
\ee
\end{lem}

\begin{proof} From Proposition \ref{Decomposition_Prop} in the $X=P$ case, we have
\[
\mathcal H_{P}[ B + z]=  \mathcal H_{P}[ B] + \mathcal G_P[z] + \mathcal Q_P[z] + \mathcal N_P[z],
\]
and $\mathcal H_P[B]=0$. From \ref{Aux_00}, we have $\mathcal G_P[z]=0$. Therefore,
\[
\mathcal H_P[B +z] =  \mathcal Q_P[z] + \mathcal N_P[z],
\]
where $\mathcal N_P$ satisfies \eqref{Nonlinear}. Recall that $\mathcal Q_P$ is given by \eqref{quad_quad}-\eqref{Ec_P_1}. More precisely, we have
\[
\begin{aligned}
\mathcal Q_{P}[z] =&~ \frac12\re \int \bar{z} \Bigg( z_{4x}  + \frac32(|B|^2-1)^2 z + 6(|B|^2-1)B\re (B\bar z)+3(|B|^2-1)z_{xx} \\
& ~{} \qquad \qquad -4 |B_x|^2z -6 BB_x\bar{z}_x -4 B\bar{B}_xz_x   -B_x^2\bar{z}  + B^2 \bar z_{xx}   + |B|^2 z_{xx}  \Bigg).
\end{aligned}
\]
Write $B_P= e^{it} + \widetilde B_P$,  where $\lim_{t\to \pm \infty} \| \widetilde B_P(t)\|_{L^\infty}=0$. We claim
\be\label{Final_P}
\mathcal Q_P[z] = \mathcal Q_{e^{it} }[z] + o_{t\to +\infty}(1).
\ee
Assuming this property, we can conclude \eqref{F_P}, since 
\[
\mathcal Q_{e^{it}}[z] = \frac12 \re\int \bar{z}( z_{4x} + z_{xx}   + e^{2it}\bar{z}_{xx}) = \frac12\re \int  (|z_{xx}|^2 -|z_x|^2 - (e^{it}\bar{z}_x)^2 ).
\]
It only remains to prove \eqref{Final_P} which follows easily from Cauchy-Schwarz and the fact that $\lim_{t\to \pm \infty} \| \partial_x^k \widetilde B_P(t)\|_{L^\infty}=0$ for all $k\geq 0$.
\end{proof}
In what follows, we make the change of variables $w:=e^{-it} z_x$. We have from \eqref{F_P},
\be\label{Final_PP}
\mathcal H_P[B+z]=\frac12\re \int  (|w_{x}|^2 -|w|^2 - w^2 ) +O(\|z\|_{H^1}^3) +o_{t\to +\infty}(1).
\ee
 From \eqref{Final_PP} we have that for $z=z_0\in H^2$ fixed small and $w=e^{-it} \partial_x z_0$, 
\[
\frac{d^2}{ds^2}\mathcal H_P[B+ s z_0]\Bigg|_{s=0} =\frac12\re \int  (|w_{x}|^2 -|w|^2 - w^2 ) +O(\|z_0\|_{H^1}^3)  +o_{t\to +\infty}(1).
\]
This concludes the proof of Theorem \ref{TH4}. In particular, the second derivative functional for the Peregrine breather has negative continuous spectrum bounded by $-1$.

\bigskip

\section{Proof of Theorem \ref{TH5}. The Kuznetsov-Ma case}\label{Sect:8}

We start with the following result.

\begin{lem}[Essential spectrum]
Let $\mathcal{L}_{KM}$ be the linear operator in \eqref{Ec_KM_1} associated to the $KM$ breather \eqref{KM}. Then $\mathcal{L}_{KM}$ is a compact perturbation of the constant (in $x$) coefficients operator with dense domain $H^4(\R;\Com)$
\be\label{Ec_KM_1_new}
\begin{aligned}
  \mathcal{L}_{KM,0}[z]:= &~ z_{4x} + z_{xx} +   e^{2it}\bar{z}_{xx}- \beta^2 (z_{xx} +e^{2it}\bar{z} + z).
\end{aligned}
\ee
\end{lem}

The proof of this result is direct in view of the spatial exponential decay of the KM breather to the Stokes wave, and the Weyl's Theorem.

 \begin{lem}
Let $a>\frac12$ be any fixed parameter in \eqref{KM}, and $\beta$ given in \eqref{KM} as well. Then we have
\be\label{sigma_c}
\sigma_c( \mathcal{L}_{KM,0})=
\begin{cases}  
[-2\beta^2,\infty) & \beta \geq \sqrt{2},\\
  [ -\frac14(2-\beta^2)^2-2\beta^2,\infty) & \beta \in (0, \sqrt{2}). \end{cases}
\ee
\end{lem}

\begin{proof}
Let $\la\in\R$ be such that $ \mathcal{L}_{KM,0}z=\la z$. In matrix form, we have
\[
\left( \begin{matrix} \partial_x^4 -(\beta^2-1)\partial_x^2 -\beta^2 & & -e^{2it} (-\partial_x^2 + \beta^2) \\ -e^{-2it} (-\partial_x^2 + \beta^2) & &  \partial_x^4 -(\beta^2-1)\partial_x^2 -\beta^2 \end{matrix}\right) \left( \begin{matrix} z \\ \bar z \end{matrix}\right) =\la  \left( \begin{matrix} z \\ \bar z \end{matrix}\right).
\]
Let us diagonalize the matrix operator on the LHS. In Fourier variables we have
\[
\left( \begin{matrix} \xi^4 + (\beta^2-1)\xi^2 -\beta^2 & & -e^{2it} (\xi^2 + \beta^2) \\ -e^{-2it} (\xi^2 + \beta^2) & &  \xi^4 +(\beta^2-1)\xi^2 -\beta^2 \end{matrix}\right),
\]
for which the diagonal operators $ \mathcal L_{KM,0,\pm} $ are in Fourier variables
\[
\begin{aligned}
\mathcal F\left( \mathcal L_{KM,0,\pm} \right):= &~{} \xi^4 +(\beta^2-1)\xi^2 -\beta^2 \pm (\xi^2 +\beta^2) \\
=&~{}  \begin{cases} \xi^4 +\beta^2\xi^2 \\ \xi^4 +(\beta^2-2)\xi^2 -2\beta^2 .
 \end{cases}
 \end{aligned}
\]
Consider now the operator $\mathcal L_{KM,0,-}= \partial_x^4 -(\beta^2-2)\partial_x^2 -2\beta^2$. If $\beta^2\geq 2$, then $\sigma_c(\mathcal L_{KM,0,-}) = [-2\beta^2,\infty)$, proving the first part in \eqref{sigma_c}. If now $0<\beta^2<2$, we have after a simple computation that $\sigma_c(\mathcal L_{KM,0,-}) = [ -\frac14(2-\beta^2)^2-2\beta^2,\infty)$. The proof is complete.
\end{proof}

\subsection{End of proof of Theorem \ref{TH5}} We have that \eqref{espectro} is a direct consequence of \eqref{sigma_c}, and  $\mathcal H_{KM}''[B_{KM}](\partial_x B_{KM}) = 0$ is also a consequence of Theorems \ref{TH1a} and \ref{TH1}.


%
%
%
%

\appendix

\bigskip
\bigskip

\section{Proof of \eqref{QODE0}}\label{B}

\medskip

Let $B_{SS}=B = Q_\bt e^{i\Theta}$ be the soliton solution \eqref{R} of \eqref{SS}. Then we have
\[
\begin{aligned}
&B_x = Q_\bt' e^{i\Theta} + i\al B,\\
&B_{xx} = Q_\bt'' e^{i\Theta} +2i\al Q_\bt' e^{i\Theta} -\al^2 B,\\
&B_{xxx} = Q_\bt''' e^{i\Theta} + 3i\al Q_\bt'' e^{i\Theta} -3\al^2Q_\bt' e^{i\Theta} -i\al^3B,\\
&B_{xxxx} = Q_\bt'''' e^{i\Theta} + 4i\al Q_\bt''' e^{i\Theta} -6\al^2 Q_\bt'' e^{i\Theta} -4i\al^3 Q_\bt' e^{i\Theta} + \al^4 B.
\end{aligned}
\]
\noindent
Now, substituting the above derivatives  in LHS of \eqref{perEcBp}, we have
\[
\begin{aligned}
& B_{(4x)} +8B_x^2 \bar B +14|B|^2 B_{xx}+6B^2 \bar B_{xx}  +12  |B_x|^2 B + 24 |B|^4 B \\
& \qquad - m_{SS}(B_{xx} + 4 |B|^2 B) + n_{SS} B \\
& \quad =e^{i\Theta}\Big( Q_\bt'''' + 4i\al Q_\bt'''  -6\al^2 Q_\bt''  -4i\al^3 Q_\bt'  + \al^4 Q_\bt + 8\bar Q_\bt(Q_\bt'+i\al Q_\bt)^2\\
& \qquad  + 14 Q_\bt\bar Q_\bt(Q_\bt''+2i\al Q_\bt' -\al^2Q_\bt) + 12 Q_\bt(Q_\bt'+i\al Q_\bt)(\bar Q_\bt' -i\al \bar Q_\bt)\\
& \qquad + 6Q_\bt^2(\bar Q_\bt''-2i\al\bar Q_\bt'-\al^2\bar Q_\bt) + 24Q_\bt^3\bar Q_\bt^2\\
& \qquad - m_{SS}(Q_\bt''+2i\al Q_\bt'-\al^2 Q_\bt + 4Q_\bt Q_\bt^2\bar Q_\bt)
+n_{SS} Q_\bt\Big)
\end{aligned}
\]
\noindent
Expanding and simplifying we get
\[
\begin{aligned}
&  B_{(4x)} +8B_x^2 \bar B +14|B|^2 B_{xx}+6B^2 \bar B_{xx}  +12  |B_x|^2 B + 24 |B|^4 B \\
& \qquad - m_{SS}(B_{xx} + 4 |B|^2 B) + n_{SS} B \\
& \quad = e^{i\Theta}\Big( Q_\bt'''' + 4i\al Q_\bt'''  -6\al^2 Q_\bt''  -4i\al^3 Q_\bt'  + \al^4 Q_\bt + 8\bar Q_\bt(Q_\bt'^2+2i\al Q_\bt Q_\bt' -\al^2Q_\bt^2)\\
& \qquad + 14 Q_\bt\bar Q_\bt(Q_\bt''+2i\al Q_\bt' -\al^2Q_\bt) + 12 Q_\bt(Q_\bt'\bar Q_\bt' -i\al Q_\bt'\bar Q_\bt + i\al\bar Q_\bt'Q_\bt + \al^2Q_\bt\bar Q_\bt)\\
& \qquad + 6Q_\bt^2(\bar Q_\bt''-2i\al\bar Q_\bt'-\al^2\bar Q_\bt) + 24Q_\bt^3\bar Q_\bt^2 \\
& \qquad - m_{SS}(Q_\bt''+2i\al Q_\bt'-\al^2 Q_\bt + 4Q_\bt Q_\bt^2\bar Q_\bt) +n_{SS} Q_\bt\Big).
\end{aligned}
\]
This implies that 
\[
\begin{aligned}
&  B_{(4x)} +8B_x^2 \bar B +14|B|^2 B_{xx}+6B^2 \bar B_{xx}  +12  |B_x|^2 B + 24 |B|^4 B \\
& \qquad - m_{SS}(B_{xx} + 4 |B|^2 B) + n_{SS} B \\
& \quad = e^{i\Theta}\Big( Q_\bt'''' + 4i\al Q_\bt'''  -6\al^2 Q_\bt''  -4i\al^3 Q_\bt'  + \al^4 Q_\bt + 8\bar Q_\bt Q_\bt'^2 +32i\al Q_\bt\bar Q_\bt Q_\bt'\\
& \qquad - 16\al^2 Q_\bt^2\bar Q_\bt + 14Q_\bt\bar Q_\bt Q_\bt'' + 12 Q_\bt Q_\bt'\bar Q_\bt' + 6Q_\bt^2\bar Q_\bt''\\
& \qquad  + 24Q_\bt^3\bar Q_\bt^2 - m_{SS}(Q_\bt''+2i\al Q_\bt'-\al^2 Q_\bt + 4Q_\bt Q_\bt^2\bar Q_\bt)+n_{SS} Q_\bt\Big),
\end{aligned}
\]
which is nothing but \eqref{QODE0}. 

\bigskip

\section{Proof of \eqref{nlid1}}\label{App_nlid1}

\medskip

Denote
\[
 Q_\bt=\frac{2\bt(e^{\bt x} + \eta e^{-\bt x})}{D},\quad D:=2 + e^{2\bt x} +  |\eta|^2 e^{-2\bt x}, \quad \eta=\frac{\al}{\al + i\bt}.
\]

Now, substituting $Q_\bt$, expanding and collecting similar terms, we rewrite the nonlinear identity \eqref{nlid1} as follows:

\[
\begin{aligned}
\eqref{nlid1} =&~{}\frac{1}{D^5}\sum_{n=-4}^{3}A_{(2n+1)}e^{(2n+1)\bt x}  =  \frac{1}{D^5}\Big(A_{7}e^{7\bt x} + A_{5}e^{5\bt x} + A_{3}e^{3\bt x}+ A_{1}e^{\bt x}+ A_{-1}e^{-\bt x} \\
& ~{} \qquad + A_{-3}e^{-3\bt x}+ A_{-5}e^{-5\bt x} + A_{-7}e^{-7\bt x}\Big),
\end{aligned}
\]
 \noindent
where
\[
 \begin{aligned}
 A_7:= & ~ {}\frac{2 \beta ^3 (8 \beta +i \alpha ) \left(\alpha ^2+\beta ^2\right) }{\beta -i \alpha }-16 \beta ^3 \left(\alpha ^2+4 \beta ^2\right) \\
 &~{} +2 \beta ^3 
\left(\alpha ^2+\beta ^2\right) \left(8+\frac{\alpha }{\alpha +i \beta }\right)-16 i \alpha  \beta ^4 +32 \beta ^5 =0,
\end{aligned}
\]                 
\[
 \begin{aligned}
A_5:= & ~ {} -\frac{32 \beta ^5 (\beta -i \alpha )^3 \left(5 \alpha ^2+4 i \alpha  \beta +20 \beta ^2\right)}{(\alpha +i \beta )^4 (\beta +i \alpha )}\\
& ~ {} +40 \beta ^3 \left(\alpha ^2+2 i \alpha  \beta +2 \beta ^2\right) +8 \beta ^3 \left(9 \alpha ^2-2 i \alpha  \beta +6 \beta ^2\right)\\
&-\frac{16 \beta ^3 \left(\alpha ^2+4 \beta ^2\right) \left(7 \alpha ^2-i \alpha  \beta +4 \beta ^2\right) }{\alpha ^2+\beta ^2} \\
& ~ {} -\frac{16 i \alpha ^2 \beta ^4\left(5 \alpha ^3-21 i \alpha ^2 \beta +5 \alpha  \beta ^2-21 i \beta ^3\right) }{\left(\alpha ^2+\beta ^2\right)^2}+768 \beta ^5 =0,
\end{aligned}
\]   
\[
\begin{aligned}
A_3:= & ~ {} \frac{8 \beta ^3 \left(\alpha ^2+\beta ^2\right) \left(-21 i \alpha ^3+34 \alpha ^2 \beta 
+6 i \alpha  \beta ^2+8 \beta ^3\right) }{(\beta -i \alpha )^2 (\beta +i \alpha )}\\
 & ~ {} +\frac{8 \beta ^3 \left(\alpha ^2+\beta ^2\right) \left(-21 i \alpha ^3+14 \alpha ^2 \beta -14 i \alpha  \beta ^2+8 \beta ^3\right) }{(\beta -i \alpha )^2 (\beta +i \alpha )}\\
&~{}-\frac{16 \beta ^3 \left(\alpha ^2+4 \beta ^2\right) \left(-21 i \alpha ^3+15 \alpha ^2 \beta -8 i \alpha  \beta ^2+4 \beta ^3\right) 
}{(\beta -i \alpha )^2 (\beta +i \alpha )} \\
&~{} +\frac{32 \beta ^5 \left(\alpha ^2+\beta ^2\right)^2 \left(39 i \alpha ^3-42 \alpha ^2 
\beta +8 i \alpha  \beta ^2+4 \beta ^3\right) }{(\alpha +i \beta )^4 (\beta +i \alpha )^3}\\
&~{}+\frac{16 i \alpha  \beta ^4 
\left(-9 \alpha ^4+104 i \alpha ^3 \beta +19 \alpha ^2 \beta ^2+80 i \alpha  \beta ^3+4 \beta ^4\right) }{\left(\alpha ^2
+\beta ^2\right)^2}\\
&~{} +\frac{768 \alpha  \beta ^5 (\beta +5 i \alpha ) }{(\alpha +i \beta ) (\beta +i \alpha )}=0,
\end{aligned}
\] 
\[
\begin{aligned}
A_1:= & ~ {} \frac{1536 \alpha ^2 \beta ^5 \left(5 \alpha ^2-2 i \alpha  \beta +\beta ^2\right) }{\left(\alpha ^2+\beta ^2\right)^2}\\
& ~ {} -\frac{16 \alpha  \beta ^3 \left(\alpha ^2+4 \beta ^2\right) \left(35 \alpha ^3-15 i \alpha ^2 \beta +20 \alpha  \beta ^2-4 i \beta ^3\right) }{\left(\alpha ^2+\beta ^2\right)^2}\\
&-\frac{16 i \alpha ^2 \beta ^4 \left(5 \alpha ^3-205 i \alpha ^2 \beta -84 \alpha  \beta ^2-76 i \beta ^3\right) }{\left(\alpha ^2+\beta ^2\right)^2}\\
& ~ {} +\frac{4 \beta ^3 \left(63 \alpha ^4-28 i \alpha ^3 \beta +56 \alpha ^2 \beta ^2-16 i \alpha  \beta ^3+8 \beta ^4\right) }{\alpha ^2+\beta ^2}\\
& ~ {} -\frac{4 \beta ^3 \left(-77 \alpha ^4+12 i \alpha ^3 \beta -24 \alpha ^2 \beta ^2-16 i \alpha  \beta ^3+8 \beta ^4\right) }{\alpha ^2+\beta ^2}\\
& ~ {} +\frac{32 \alpha  \beta ^5 \left(85 i \alpha ^4+112 \alpha ^3 \beta -47 i \alpha ^2 \beta ^2-28 \alpha  \beta ^3+8 i \beta ^4\right) }{(\alpha +i \beta )^2 (\beta +i \alpha )^3}=0,
\end{aligned}
\]
\[
\begin{aligned}
A_{-1}: = &~ {}-\frac{1536 i \alpha ^3 \beta ^5 \left(5 \alpha ^2+2 i \alpha  \beta +\beta ^2\right) }{(\beta -i \alpha )^3 (\beta +i \alpha )^2}\\
&~ {} -\frac{16 \alpha ^2 \beta ^3 \left(\alpha ^2+4 \beta ^2\right) \left(-35 i \alpha ^3+15 \alpha ^2 \beta -20 i \alpha  \beta ^2+4 \beta ^3\right) }{(\beta -i \alpha )^3 (\beta +i \alpha )^2}\\
&~ {}+\frac{16 i \alpha ^3 \beta ^4 \left(5 \alpha ^3+205 i \alpha ^2 \beta -84 \alpha  \beta ^2+76 i \beta ^3\right) }{(\alpha +i \beta ) \left(\alpha ^2+\beta ^2\right)^2}\\
&~ {} +\frac{4 \alpha  \beta ^3 \left(\alpha ^2+\beta ^2\right) \left(-63 i \alpha ^4+28 \alpha ^3 \beta -56 i \alpha ^2 \beta ^2+16 \alpha  \beta ^3-8 i \beta ^4\right) }{(\beta -i \alpha )^3 (\beta +i \alpha )^2}\\
&~ {} -\frac{4 \alpha  \beta ^3 \left(\alpha ^2+\beta ^2\right) 
\left(77 i \alpha ^4-12 \alpha ^3 \beta +24 i \alpha ^2 \beta ^2+16 \alpha  \beta ^3-8 i \beta ^4\right) }{(\beta -i \alpha )^3 (\beta +i \alpha )^2}\\
&~ {} +\frac{32 \alpha ^2 \beta ^5 \left(85 i \alpha ^5-27 \alpha ^4 \beta +65 i \alpha ^3 \beta ^2-19 \alpha ^2 \beta ^3-20 i \alpha  \beta ^4+8 \beta ^5\right) }{(\alpha +i \beta )^4 (\beta +i \alpha )^3}= 0,
\end{aligned}
\]
\\
\[
\begin{aligned}
A_{-3}:= &~\frac{768 \alpha ^4 \beta ^5 (\beta -5 i \alpha )}{(\beta -i \alpha )^3 (\beta +i \alpha )^2}
+\frac{8 \alpha ^3 \beta ^3 \left(\alpha ^2+\beta ^2\right) \left(21 i \alpha ^3+14 \alpha ^2 \beta +14 i \alpha  \beta ^2+8 \beta ^3\right) }{(\alpha -i \beta )^3 (\beta -i \alpha )^3}\\
& ~{} -\frac{16 \alpha ^3 \beta ^3 \left(\alpha ^2+4 \beta ^2\right) \left(21 i \alpha ^3+15 \alpha ^2 \beta +8 i \alpha  \beta ^2+4 \beta ^3\right) }{(\alpha -i \beta )^3 (\beta -i \alpha )^3}\\
& ~{} -\frac{8 \alpha ^3 \beta ^3 \left(\alpha ^2+\beta ^2\right) \left(21 i \alpha ^3+34 \alpha ^2 \beta -6 i \alpha  \beta ^2+8 \beta ^3\right) }{(\alpha +i \beta )^3 (\beta +i \alpha )^3}\\
&~{} -\frac{16 i \alpha ^4 \beta ^4 \left(-9 i \alpha ^3+95 \alpha ^2 \beta -76 i \alpha  \beta ^2+4 \beta ^3\right) }{(\alpha +i \beta ) (\beta +i \alpha ) \left(\alpha ^2+\beta ^2\right)^2}\\
&~{} +\frac{32 \alpha ^3 \beta ^5 \left(39 i \alpha ^4+3 \alpha ^3 \beta +50 i \alpha ^2 \beta ^2-12 \alpha  \beta ^3-4 i \beta ^4\right) }{(\alpha +i \beta )^4 (\beta +i \alpha )^3}
=0,\\
&\\
A_{-5}:=& ~{}\frac{768 \alpha ^5 \beta ^5 }{(\alpha -i \beta )^2 (\alpha +i \beta )^3}
-\frac{16 i \alpha ^7 \beta ^4 (-21 \beta +5 i \alpha ) }{(\beta -i \alpha )^2 (\beta +i \alpha ) \left(\alpha ^2+\beta ^2\right)^2}\\
&  ~{} +\frac{32 \alpha ^5 \beta ^5 \left(5 i \alpha ^2+4 \alpha  \beta +20 i \beta ^2\right) }{(\alpha +i \beta )^4 (\beta +i \alpha )^3}
+\frac{8 \alpha ^5 \beta ^3 \left(\alpha ^2+\beta ^2\right) \left(9 \alpha ^2+2 i \alpha  \beta +6 \beta ^2\right) }{(\alpha -i \beta )^3 (\alpha +i \beta )^4}\\
& ~{} -\frac{40 \alpha ^5 \beta ^3 \left(i \alpha ^2+2 \alpha  \beta +2 i \beta ^2\right) \left(\alpha ^2+\beta ^2\right) }{(\alpha +i \beta )^4 (\beta +i \alpha )^3}\\
& ~{} -\frac{16 \alpha ^5 \beta ^3 \left(\alpha ^2+4 \beta ^2\right) \left(7 \alpha ^2+i \alpha  \beta +4 \beta ^2\right) }{(\alpha -i \beta )^3 (\alpha +i \beta )^4}
=0,
\end{aligned}
\]
and
\[
\begin{aligned}
A_{-7}:=&~{} -\frac{32 i \alpha ^7 \beta ^5 }{(\alpha +i \beta )^4 (\beta +i \alpha )^3}
+\frac{16 i \alpha ^8 \beta ^4 }{(\alpha +i \beta ) \left(\alpha ^2+\beta ^2\right)^3}\\
&~{} -\frac{16 \alpha ^7 \beta ^3 \left(\alpha ^2+4 \beta ^2\right) }{(\alpha -i \beta ) (\alpha +i \beta )^2 \left(\alpha ^2+\beta ^2\right)^2}\\
&~  {}  -\frac{2 \alpha ^7 \beta ^3 (\alpha +8 i \beta ) }{\left(\alpha ^2+\beta ^2\right)^3}+\frac{2 \alpha ^7 \beta ^3 (9 \alpha -8 i \beta ) }{\left(\alpha ^2+\beta ^2\right)^3}=0,
\end{aligned}
\]
\noindent
and we conclude.

\section{Proofs of \eqref{Ec_SY}, \eqref{Ec_KM} and \eqref{Ec_P}}\label{Calculos}

This section continues and ends the proof mentioned in Subsection \ref{Calculos0}.

\subsection{Proof of \eqref{Ec_SY}}

We will use, for the sake of simplicity, the following notation for the SY breather solution \eqref{SY}:
\be\label{AppSY}
\begin{aligned}
&B_{SY}=\frac{M}{N}, \quad\text{with}\\
& M:= 2\sqrt{2} \ga_+\ga_- e^{ic_1^2t}(c_1\cosh(c_2x) + c_2e^{i \ga_+\ga_- t} \cosh(c_1x)),\\
& N:=\ga_-^2\cosh( \ga_+ x) + \ga_+^2\cosh(\ga_- x) + 2c_1c_2(e^{i \ga_+\ga_- t}+e^{-i \ga_+\ga_- t}).
\end{aligned}
\ee
\noindent
Now, we rewrite the identity \eqref{Ec_SY} in terms of $M,N$ in the following way
\be\label{AppSYMN}
\eqref{Ec_SY} = \frac1{N^5}\sum_{i=1}^{5}S_i,
\ee
with $S_i$ given explicitly by:
\be\label{AppSYs1}
\begin{aligned}
S_1= &~{} iN\Big(6MN_tN_x^2-2N(N_x(M_tN_x + 2MN_{xt})+N_t(2M_xN_x+MN_{xx}))\\
&~{}-N^3M_{xxt}+N^2(2N_xM_{xt}+2M_xN_{xt}+M_{xx}N_t+M_tN_{xx}+MN_{xxt})\Big),
\end{aligned}
\ee
\be\label{AppSYs2}
S_2=\bar{M}(NM_x-MN_x)^2,
\ee
\be\label{AppSYs3}
S_3=2M\bar{M}\Big(2MN_x^2+N^2M_{xx}-N(2M_xN_x+MN_{xx})\Big),
\ee
\be\label{AppSYs4}
S_4=2M(MN_x-M_xN)(N\bar{M}_x - \bar{M}N_x),
\ee
and
\be\label{AppSYs5}
\begin{aligned}
S_5=&~{} \frac32 M^3\bar{M}^2+ n N^4M \\
&~{} -m N^2\Big(M^2\bar{M} + N(NM_{xx}-2M_xN_x) + M(2N_x^2-NN_{xx})\Big) ,
\end{aligned}
\ee

\medskip
\noindent
where we skipped index \emph{SY} in parameters $m_{SY},~n_{SY}$ for simplicity. Now substituting the explicit functions $M,N$ \eqref{AppSY} in $S_i,~i=1,\dots,5$ and collecting terms, we get after lengthy manipulations that
\be\label{AppSYsum}
\begin{aligned}
&\sum_{i=1}^{5}S_i= \sum_{i=1}^{29}p_is_i,\\
\end{aligned}
\ee
\medskip
\noindent
where, labeling $r_1=\sinh(c_1x)\sinh(c_2x),~s_1=\cosh(c_1x), ~s_2=\cosh(c_2x),$
\be\label{AppSYsum2}
\begin{aligned}
& ~ s_3=s_1r_1,~s_4=s_1^3 ~s_5=s_1^3r_1,~ s_6=s_1^5, ~s_7=s_1r_1s_2^4,~s_8=s_2r_1,\\
& ~s_9=s_1^2s_2, ~s_{10}=r_1s_1^2s_2, ~s_{11}=s_1^4s_2, ~s_{12}=r_1s_1^4s_2, ~s_{13}=s_1s_2^2,\\
&~ s_{14}=r_1s_1s_2^2,~s_{15}=s_1^3s_2^2,~s_{16}=r_1s_1^3s_2^2, ~s_{17}=s_1^5s_2^2, ~s_{18}=s_2^3,\\
& ~s_{19}=r_1s_2^3,~s_{19}=s_1^2s_2^3, ~s_{20}=r_1s_1^2s_2^3,~s_{21}=s_1^4s_2^3, ~s_{22}=r_1s_1^4s_2^3,\\
&~s_{23}=s_1s_2^4, ~s_{24}=s_1^3s_2^4, ~s_{25}=r_1s_1^3s_2^4,~s_{26}=s_1^5s_2^4, ~s_{27}=s_2^5, \\
&~s_{28}=s_1^2s_2^5, ~s_{29}=s_1^4s_2^5, \\
\end{aligned}
\ee
\medskip
\noindent
and
\be\label{APPpi}
p_i=\sum_{j=0}^{L_i}a_{ij}(c_1,c_2,m,n)e^{j\cdot(2i\ga_+\ga_-t)},~~L_i\in\N,
\ee
\noindent
with $a_{ij}$ a polynomial in $c_1,c_2,m,n$. For instance, we have for the first term in \eqref{AppSYsum}, i.e.
\[
 p_1s_1=\Big(\sum_{j=0}^{4}a_{1j}e^{2j\cdot(2i\ga_+\ga_-t)}\Big)s_1,
\]
\noindent
with
\be\label{coefa11}
\begin{aligned}
&a_{10}=-16 c_1^4 c_2^4 (c_1^4 - 2 c_2^4 - c_1^2 m +   2 c_2^2 m - n),\\
& a_{11}=-224 c_1^4 c_2^4  (c_1^4 - 3 c_2^4 - c_1^2 m +    3 c_2^2 m - 2 n),\\
&a_{12}=1120 c_1^4 c_2^4  (c_2^4 - c_2^2 m + n),\\
& a_{13}=224 c_1^4 c_2^4  (c_1^4 + c_2^4 - c_1^2 m -    c_2^2 m + 2 n),\\
&a_{14}=16 c_1^4 c_2^4  (c_1^4 - c_1^2 m + n).
\end{aligned}
\ee
\medskip
\noindent
Then, imposing e.g. $a_{14}=0$, and substituting $n=c_1^2m-c_1^4$ into $a_{10}$, we get that
\[
 a_{10}=2(c_1^2-c_2^2)(c_1^2+c_2^2-m),
\]
therefore, $a_{10}=0$, if $m=c_1^2+c_2^2$ and then $n=c_1^2c_2^2$. In fact, substituting $n=c_1^2m-c_1^4$ into the
coefficients $a_{11},~a_{12},~a_{13}$, we get that
all them are proportional to the factor $(c_1^2+c_2^2-m)$, namely
\be\label{coefa11a}
\begin{aligned}
&a_{11}=3(c_1^2-c_2^2)(c_1^2+c_2^2-m),\\
&a_{12}= (c_1^2 - c_2^2)(-(c_1^2 + c_2^2) + m) ,\\                                                                                                                                              
&a_{13}=(c_1^2 - c_2^2)(-(c_1^2 + c_2^2) + m),
\end{aligned}
\ee
\medskip
\noindent
and then  when $m=c_1^2+c_2^2$ we get  $a_{11}=a_{12}=a_{13}=0$, and $p_1=0$. Now, selecting $n=c_1^2c_2^2$ and analyzing the rest of polynomials 
$p_i,~i=2,\dots29$, in \eqref{AppSYsum}-\eqref{APPpi}, it is easy to see  that all coefficients $a_{ij},~i=2,\dots29$  are proportional 
to the factor $(c_1^2+c_2^2-m)$, i.e.
\[
 a_{ij}=b_{ij}(c_1,c_2)\cdot(c_1^2+c_2^2-m),
\]
\noindent
with $b_{ij}$ a polynomial in $c_1,c_2$. Therefore, selecting $m=c_1^2+c_2^2$, we get $a_{ij}=0,~\forall i=2,\dots,29,\forall j=0,\dots, L_i$ and we conclude.

\subsection{Proof of \eqref{Ec_KM}}
The proof is similar to the one for \eqref{Ec_SY}. Let us use the following notation for the KM breather solution \eqref{KM}:
\be\label{AppKM}
\begin{aligned}
&B_{KM}=e^{it}\left(1-\frac{M}{N}\right), \quad\text{with}\\
& M:=\sqrt{2} \beta  \left(\beta ^2 \cos \left(\alpha  t\right)+i \alpha  \sin \left(\alpha  t\right)\right),\\
& N:=\alpha  \cosh ( \beta  x)-\sqrt{2} \beta  \cos \left(\alpha  t\right).
\end{aligned}
\ee
Now, we rewrite the identity \eqref{Ec_KM} in terms of $M,N$ in the following way
\be\label{AppKMMN}
\eqref{Ec_KM} = \frac{e^{it}}{N^5}\sum_{i=1}^{6}S_i,
\ee
\noindent
with $S_i$ given explicitly by:
\be\label{AppKMs1}
\begin{aligned}
S_1:= & ~{} -N\Big(6iMN_tN_x^2-2iN(N_x(M_tN_x+M(iN_x+2N_{xt})) + N_t(2M_xN_x+MN_{xx}))\\
& ~ {}  +N^3(M_{xx}-iM_{xxt})+N^2(-2M_x(N_x-iN_{xt}) \\
& ~ {} +i(2N_xM_{xt}+N_tM_{xx}+iMN_{xx}+M_tN_{xx}+MN_{xxt}))\Big),
\end{aligned}
\ee
\be\label{AppKMs2}
\begin{aligned}
&S_2:=-(\bar{M}-N)(NM_x-MN_x)^2,
\end{aligned}
\ee
\be\label{AppKMs3}
\begin{aligned}
&S_3:=2(\bar{M}-N)(M-N)\Big(-2MN_x^2-N^2M_{xx}+N(2M_xN_x+MN_{xx})\Big),
\end{aligned}
\ee
\be\label{AppKMs4}
\begin{aligned}
&S_4:=-2(M-N)(-NM_x+MN_x)(N\bar{M}_x-\bar{M}N_x),
\end{aligned}
\ee
\be\label{AppKMs5}
\begin{aligned}
&S_5:=-\frac32(M-N)(\bar{M}N+M(-\bar{M}+N))^2,
\end{aligned}
\ee
and
\be\label{AppKMs6}
\begin{aligned}
S_6:=&~{} N^2\Big(\bt^2M^2(\bar{M}-N) + N(\bt^2\bar{M}N-(3+\bt^2)(2M_xN_x-NM_{xx}))\\
&~{}+ M(-2\bt^2\bar{M}N+\bt^2N^2+2(3+\bt^2)N_x^2-(3+\bt^2)NN_{xx})\Big).
\end{aligned}
\ee
Now substituting the explicit functions $M,N$ \eqref{AppKM} in $S_i,~i=1,\dots,6$ and collecting terms, we get 
\[
\begin{aligned}
\sum_{i=1}^{6}S_i =&~{}  a_1\cosh^2(x\beta)\cos(t \alpha) +  a_2\cosh^4 (\bt x)\cos(\al t)  + a_3\cosh(\bt x)\cos^2(\al t)\\
&~{} + a_4 \cosh^3 (\bt x)\cos^2(\al t)  + a_5\cos^2 (\al t)\sin(\al t)   +  a_6 \cos^3(\al t) \\
&~{}+ a_7\cosh^2(\bt x)\cos^3 (\al t) + a_{8}\cosh(\bt x)\cos^3(\al t)\sin(\al t) \\
&~{}  + a_9\cosh(\bt x)\cos^4(\al t) + a_{10}\cos(\al t)^4\sin(\al t) +  a_{11} \cos^5 (\al t) ,                                                                                                                                                                                                                                                                                                                                                                                                            
\end{aligned}
\]
with coefficients $a_i,~i=1,\dots,11$ given as follows
\[
\begin{aligned}
&a_1:=4 \sqrt{2} \alpha ^4 \beta ^3 \left(\alpha ^2-\beta ^2 \left(\beta ^2+2\right)\right), \\
& a_2:=-\frac{1}{2}a_1,\quad a_3:=-\frac{7\bt}{\sqrt{2}\al}a_1, \quad a_4:=\frac{\bt}{\sqrt{2}\al}a_1,\\
&a_5:=3i\frac{\bt^2}{\al^2}a_1,\quad a_6:=\frac{\bt^2}{\al^2}(3\bt^2+5)a_1,\\
& a_7:=3\frac{\bt^2}{\al^2}a_1, \quad a_{8}:=-4i\frac{\bt^2}{\al^2}a_1,\\
&a_9:=4 \alpha  \beta ^4 \left(3 \alpha ^4-2 \alpha ^2 \beta ^2 \left(5 \beta ^2+8\right)+\beta ^4 \left(7 \beta ^4+24 \beta ^2+20\right)\right),\\
&a_{10}:=-4 i \sqrt{2} \alpha  \beta ^5 \left(3 \alpha ^4-2 \alpha ^2 \beta ^2 \left(3 \beta ^2+5\right)+\beta ^4 \left(3 \beta ^4+10 \beta ^2+8\right)\right),\\
& a_{11}:=-\frac{i}{\al}(\beta^2+1)a_{10}.
\end{aligned}
\]
Finally, using that $\beta =\sqrt{2 (2 a -1)}$ and $\alpha =\sqrt{8a  (2 a -1)}$, we have that all $a_i$ vanish, and we conclude.

\subsection{Proof of \eqref{Ec_P}}
This identity follows in the same way that the proof of identity \eqref{Ec_KM} above. We include it for the sake of completeness, but it can be formally obtained by a standard limiting procedure.

\medskip

Let us use the following notation for the Peregrine breather solution \eqref{P}:
\[\label{AppP}
\begin{aligned}
&B_P=e^{it}\left(1-\frac{M}{N}\right), \quad\text{with}\\
&M:=4(1+2it),\quad N:=1 + 4t^2 + 2x^2.
\end{aligned}
\]
\noindent
Now, we rewrite the identity \eqref{Ec_P} in terms of $M,N$ in the following way
\[\label{AppPMN}
\eqref{Ec_P} = \frac{e^{it}}{N^5}\sum_{i=1}^{6}S_i,
\]
\noindent
with $S_i$ given explicitly by:
\[\label{AppPs12345}
\begin{aligned}
 S_1=&~{} \eqref{AppKMs1}\\
 =&~{} 16(1 + 4t^2 + 2x^2)(3 - 80 t^4 + 32 i t^5 - 12 x^2 - 36 x^4\\
 &~ {} - 16 i t^3 (-5 + 2 x^2) +   8 t^2 (-1 + 34 x^2) - 6 i t (-3 + 28 x^2 + 4 x^4)),
\end{aligned}
\]
\[\label{AppPs2}
\begin{aligned}
S_2=&~{}\eqref{AppKMs2}=-256 (i - 2 t)^2 x^2 (-3 + 8 i t + 4 t^2 + 2 x^2),
\end{aligned}
\]
\[\label{AppPs3}
\begin{aligned}
S_3=&~{} \eqref{AppKMs3}\\
=&~{} 32(1 + 2 i t) (1 + 4 t^2 - 6 x^2) (-3 - 8 i t + 4 t^2 + 2 x^2) (-3 +  8 i t + 4 t^2 + 2 x^2),
\end{aligned}
\]
\[\label{AppPs4}
\begin{aligned}
S_4=&~{} \eqref{AppKMs4}=-512 (2 t-i) (2 t+i) x^2 \left(4 t^2-8 i t+2 x^2-3\right),
\end{aligned}
\]
\[\label{AppPs5}
\begin{aligned}
S_5=&~{}\eqref{AppKMs5}=96 \left(1 + 4 t^2-2 x^2\right)^2 \left(4 t^2-8 i t+2 x^2-3\right),
\end{aligned}
\]
and
\[\label{AppPs6}
\begin{aligned}
&S_6=-48 i  (2 t-i)(1 + 4t^2 + 2x^2)^2 \left(4 t^2-6 x^2+1\right).
\end{aligned}
\]
Now collecting terms, it is easy to see that we get a polynomial
\[
 \sum_{i=1}^{6}S_i= b_0 + b_2 x^2 + b_4 x^4 + b_6 x^6,
\]
where we have that $b_0=b_2=b_4=b_6=0$.


\medskip

\section{Real and imaginary parts \eqref{reimQu1} and \eqref{reimQu2}}\label{Calculos00}

First of all, we prove \eqref{reimQu1}. Having in mind \eqref{SolQ}, we explicitly compute $Q_\eta''\overline{Q_\eta'}$ and $Q_\eta'' \overline{Q_\eta}$ respectively, i.e.
\be\label{reQ}\begin{aligned}
&Q_\eta'' \overline{Q_\eta'}=-\frac{4 e^{2 x}}{\left(|\eta|^2+2 e^{2 x}+e^{4 x}\right)^5} 
\left(-\bar{\eta} |\eta|^2+\bar{\eta} e^{2 x} \left(-3 \eta+3 e^{2 x}+2\right)
+e^{4 x} \left(e^{2 x}-2\right)\right)\times D,
\end{aligned}\ee
and
\be\label{imaQ}
Q_\eta'' \overline{Q_\eta}=\frac{4 e^{2 x}}{\left(|\eta|^2+2 e^{2 x}+e^{4 x}\right)^4} 
\left(\bar{\eta}+e^{2 x}\right)\times D,
\ee
where
\[
D:=|\eta|^4\eta +3(3 \bar{\eta}-4)|\eta|^2\eta e^{2 x}-2 \eta e^{4 x} (\bar{\eta} (11 \eta-2)-2)+e^{6 x} ((4-22\bar{\eta}) \eta+4)
+3 (3 \eta-4) e^{8 x}+e^{10 x}.\]
Now, setting $r=e^{2x},$ substituting $\eta=a+ib,$  remembering that from \eqref{polU4} 
\[
r^4 = (a^2 + b^2)^2+ \left(4 a^3-2 a^2+4 a b^2-2 b^2\right)r+(2-4 a) r^3, 
\]
and taking the real and imaginary parts to \eqref{reQ} and \eqref{imaQ} respectively, we get
\be\label{reQ2}\begin{aligned}
\re\{ Q_\eta'' \overline{Q_\eta'} \}=&256 \left(2 a^2-3 a+1\right) b^2 \left(a^2+b^2\right)^2\\
&+128 (a-1) b^2 [15 a^4-16 a^3+2 a^2 \left(7 b^2+2\right)-16 a b^2-b^2 \left(b^2-4\right)]u\\
&-256 \left(2 a^2-3 a+1\right) b^2\left(a^2+b^2\right)u^2-128 (a-1) b^2\left(15 a^2-16 a-b^2+4\right)u^3,
\end{aligned}\ee
and
\be\label{imaQ2}\begin{aligned}
\ima \{Q_\eta'' \overline{Q_\eta}\}=& -32 (a-1) b \left(a^2+b^2\right)^2
+(-64 (1 - 3 a + 2 a^2) b (a^2 + b^2))u\\
&+ (32 (a-1) b \left(a^2+b^2\right))u^2.
\end{aligned}
\ee
\medskip
Now, evaluating at $x=\frac12\log u_{1,+},$ with $u_{1,+}=\frac{\alpha }{\sqrt{\alpha ^2+\beta ^2}},$ we rewrite \eqref{reQ2} and \eqref{imaQ2} as follows:
\be\label{reQ3}\begin{aligned}
\re\{ Q_\eta'' \overline{Q_\eta'} \}\Big|_{x=\frac12\log u_{1,+}}=&
\frac{128 (a-1) b^2}{\left(\alpha ^2+\beta ^2\right)^{3/2}} \left(\alpha ^2 \left(a^2+b^2-1\right)+\beta ^2 \left(a^2+b^2\right)\right)\\
&\times\left(2 (2 a-1) \left(a^2+b^2\right)\sqrt{\alpha ^2+\beta ^2}+\alpha  \left(15 a^2-16 a-b^2+4\right)\right),
\end{aligned}\ee
and
\be\label{imaQ3}\begin{aligned}
\ima \{Q_\eta'' \overline{Q_\eta}\}\Big|_{x=\frac12\log u_{1,+}}=&-\frac{32 (a-1) b}{\left(\alpha ^2+\beta ^2\right)^{3/2}}\left(\alpha ^2 \left(a^2+b^2-1\right)+\beta ^2 \left(a^2+b^2\right)\right)\\
&\times\left(\left(a^2+b^2\right) \sqrt{\alpha ^2+\beta ^2}+(4 a-2) \alpha \right).
\end{aligned}
\ee
Finally, substituting $a=\frac{\alpha ^2}{\alpha ^2+\beta ^2},~b = -\frac{\alpha  \beta }{\alpha ^2+\beta ^2},$ we easily get 
\[
 \re\{ Q_\eta'' \overline{Q_\eta'} \}\Big|_{x=\frac12\log u_{1,+}} =\ima \{Q_\eta'' \overline{Q_\eta}\}\Big|_{x=\frac12\log u_{1,+}}=0,
\]
since $\alpha ^2 \left(a^2+b^2-1\right)+\beta ^2 \left(a^2+b^2\right)=0$.

\medskip

Now, we proceed to prove \eqref{reimQu2}. Starting from \eqref{reQ2} and \eqref{imaQ2} respectively, and 
evaluating at $x=\frac12\log u_{2,\pm},$ with $u_{2,\pm}=\frac{-\alpha ^2+\beta ^2\pm\beta \sqrt{\beta ^2-3 \alpha ^2}}{\alpha ^2+\beta ^2},$ we get
\be\label{reimQu2App}\begin{aligned}
&\re\{ Q_\eta'' \overline{Q_\eta'} \} \left( \frac12\log u_{2,\pm}\right)   + \al  \ima \{Q_\eta'' \overline{Q_\eta}\} \left( \frac12\log u_{2,\pm}\right)\\
=& -\frac{4 (a-1) b ((a-1) \alpha +b)}{\left(\alpha ^2+\beta ^2\right)^3} 
\Bigg\{\alpha ^6 [15 a^6-72 a^5+a^4 \left(29 b^2+81\right)+a^3 \left(24-80 b^2\right)\\
&+a^2 \left(13 b^4+82 b^2-88\right)-8 a \left(b^4+5 b^2-6\right)-b^6+b^4+8 b^2-8]\\
&+\alpha ^4 \beta ^2 [45 a^6-104 a^5+a^4 \left(87 b^2+164\right)-8 a^3 \left(18 b^2+98\right)\\
&+3 a^2 \left(13 b^4+56 b^2+376\right)-8 a \left(5 b^4+2 b^2+72\right)-3 b^6+4 b^4-24 b^2+96]\\
&+\alpha ^2 \beta ^4 [45 a^6+8 a^5+a^4 \left(87 b^2-35\right)-8 a^3 \left(6 b^2-105\right)\\
&+a^2 \left(39 b^4-22 b^2-1376\right)-8 a \left(7 b^4+15 b^2-90\right)-3 b^6+13 b^4+64 b^2-120]\\
&+\beta ^6 [15 a^6+40 a^5+a^4 \left(29 b^2-118\right)+8 a^3 \left(2 b^2-18\right)\\
&+a^2 \left(13 b^4-108 b^2+352\right)-8 a \left(3 b^4-14 b^2+24\right)-b^6+10 b^4-32 b^2+32]\}\\
&\mp\frac{8 (a-1) b \beta \sqrt{\beta ^2-3 \alpha ^2} ((a-1) \alpha +b)}{\left(\alpha ^2+\beta ^2\right)^3}
\{2 (2 a-1) \alpha ^2 \beta ^2 \left(7 a^2-8 a-b^2+2\right) \left(2 a^2+2 b^2+9\right)\\
&+\alpha ^4 [28 a^5-31 a^4+4 a^3 \left(6 b^2-19\right)-6 a^2 \left(5 b^2-23\right)\\
&-4 a \left(b^4-5 b^2+18\right)+b^4-6 b^2+12]+\beta ^4[28 a^5-61 a^4+4 a^3 \left(6 b^2-18\right)\\
&+a^2 \left(176-58 b^2\right)-4 a \left(b^4-14 b^2+24\right)+3 b^4-16 b^2+16]
\Bigg\}.
\end{aligned}
\ee
Finally, substituting $a=\frac{\alpha ^2}{\alpha ^2+\beta ^2},~b = -\frac{\alpha  \beta }{\alpha ^2+\beta ^2},$ we get 
\be\label{reimQu2App2}\begin{aligned}
\re\{ Q_\eta'' \overline{Q_\eta'} \} & \left( \frac12\log u_{2,\pm}\right)   + \al  \ima \{Q_\eta'' \overline{Q_\eta}\} \left( \frac12\log u_{2,\pm}\right)\\
=& ~{} \frac{2 \alpha ^2 \beta ^5 (\beta +1) \sqrt{\beta ^2-3 \alpha ^2}}{\left(\alpha ^2+\beta ^2\right)^9} \\
&\times \Big(4 \beta  \sqrt{\beta ^2-3 \alpha ^2} \left(5 \alpha ^8-50 \alpha ^6 \beta ^2+109 \alpha ^4 \beta ^4-76 \alpha ^2 \beta ^6+16 \beta ^8\right)\\
& \qquad \mp\left(\alpha ^2-\beta ^2\right) \left(\alpha ^2+\beta ^2\right)^3 [\alpha ^8-34 \alpha ^6 \beta ^2+121 \alpha ^4 \beta ^4-84 \alpha ^2 \beta ^6
+16 \beta ^8]\Big).
\end{aligned}
\ee
Note how the RHS term is zero at $\beta = \sqrt{3} \alpha$, but we are in the case $\beta > \sqrt{3} \alpha$, the equality being treated below. Now, without loss of generality we take $\alpha=1$  and we set $t=\beta^2$ to rewrite \eqref{reimQu2App2} as follows:
\be\label{reimQu2App3}\begin{aligned}
&\re\{ Q_\eta'' \overline{Q_\eta'} \} \left( \frac12\log u_{2,\pm}\right)   + \al  \ima \{Q_\eta'' \overline{Q_\eta}\} \left( \frac12\log u_{2,\pm}\right)
=\frac{2t^{5/2} (\sqrt{t} +1) \sqrt{t-3}}{\left(1+t\right)^9}P(t),\\
&P(t):=P_1(t) + P_2(t),\\
& P_1(t) :=4 \sqrt{t-3} \sqrt{t} \left(16 t^4-76 t^3+109 t^2-50 t+5\right),\\
&P_{2,\mp}(t):=\mp(1-t) (t+1)^3 \left(16 t^4-84 t^3+121 t^2-34 t+1\right).
\end{aligned}
\ee
Finally plotting $P_1,P_{2,\mp}$, we clearly see that 
$\re\{ Q_\eta'' \overline{Q_\eta'} \} \left( \frac12\log u_{2,\pm}\right)   + \al  \ima \{Q_\eta'' \overline{Q_\eta}\} \left( \frac12\log u_{2,\pm}\right)\neq0$, see Fig. \ref{SY_fig} for details.  
\begin{figure}[h!] 
              \includegraphics[width=9.0cm,height=6.5cm]{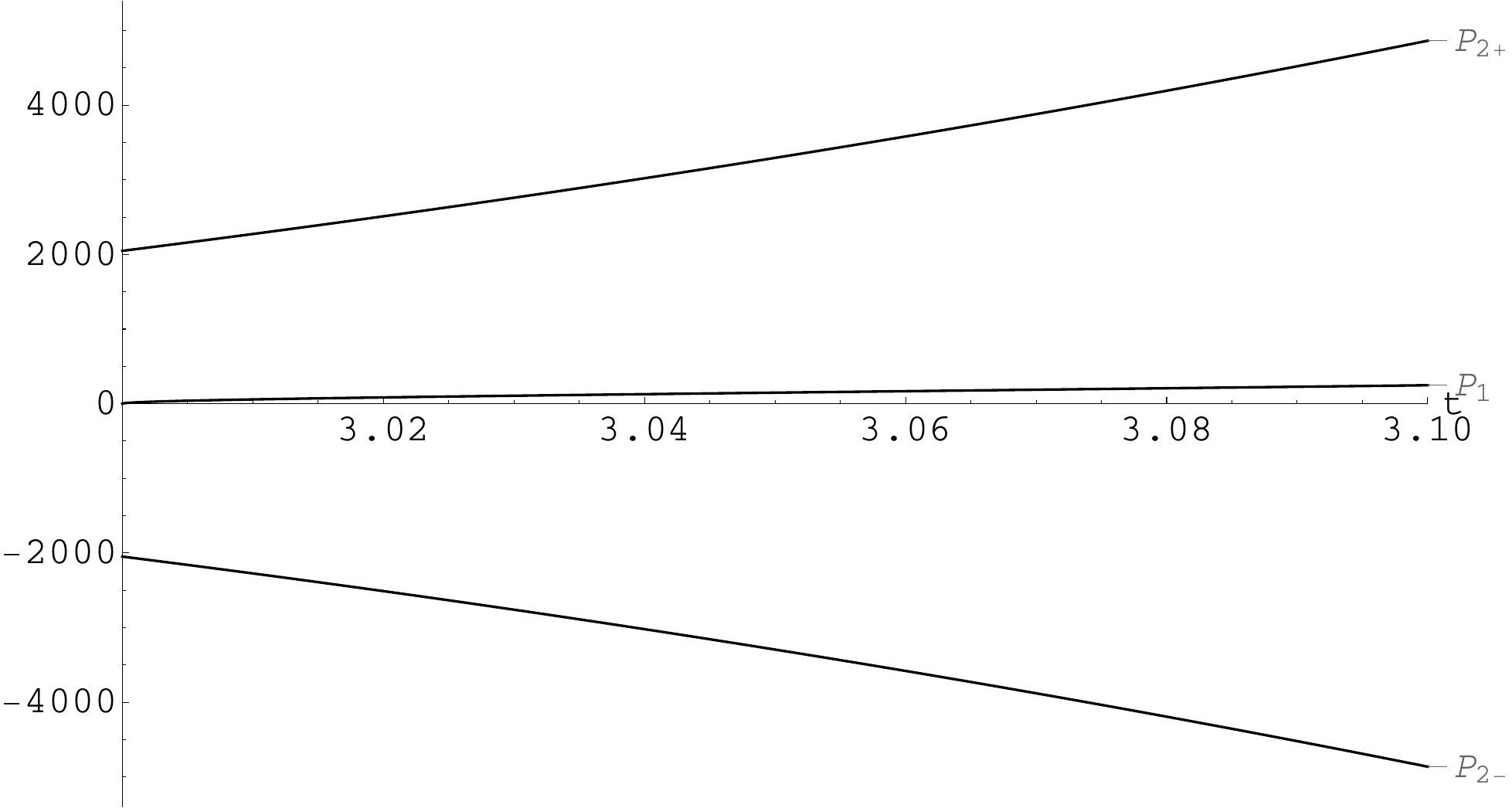}

              \caption{Graphics of $P_1$ and $P_{2,\mp}$.}\label{SY_fig}
  \end{figure}

\section{General 2-soliton solution for NLS}\label{2solitonNLS}

Recall that $c_1,c_2>0$. Let $\al_1,\al_2\in \R$ and the position and shift variables
\be\label{Shifts_Fases}
\begin{aligned}
y_1 : =c_1(x+2\alpha_1 t), \quad &\quad   y_2:=c_2(x+2\alpha_2 t),\\
\Theta_1:= \al_1x+(\al_1^2-c_1^2)t,  \quad &\quad  \Theta_2:= \al_2 x+(\al_2^2-c_2^2)t.
\end{aligned}
\ee
The general SY 2-soliton solution is given by the expression (taken after \cite{Akhmediev1} and some lengthy simplifications)
\be\label{SYa1a2}
B_{SY,gen}(t,x):=Q_{SY}(t,x)   + \frac{2\sqrt{2} c_2  s\bar{r}}{|r|^2+|s|^2},
\ee
where (compare with \eqref{soliton})
\be\label{solitonQ1}
Q_{SY}(t,x):=-\sqrt{2} c_1e^{-i \Theta_1} \sech y_1,
\ee
and (see \eqref{Shifts_Fases})
\be\label{matrixelements1}
\begin{aligned}
 r:= &~{}e^{-\frac12 (y_2-i \Theta_2) }\Big(ic_1e^{y_2 + i(\Theta_1-\Theta_2)} + (\al_2 -\al_1+ic_2)\cosh(y_1)+  ic_1 \sinh(y_1)\Big) ,\\
 s:= &~{}e^{\frac12 (y_2-i \Theta_2)}\Big(ic_1e^{-(y_2 + i(\Theta_1-\Theta_2))} +  (\al_2 -\al_1+ic_2)\cosh(y_1) - ic_1 \sinh(y_1) \Big).
\end{aligned}
\ee
By using the time and space invariance under shifts $t_0,x_0\in\R$ of the equation, a more general 2-soliton $B_{SY,gen}(t+t_0,x+x_0)$ solution can be constructed. Note that \eqref{SYa1a2} reduces to \eqref{SY} when the \emph{speed} parameters $\alpha_1,\alpha_2$ tend to zero: for each $(t,x)\in \R^2$,
\be\label{Limit_SY}
\lim_{\al_1,\al_2\to 0}B_{SY,gen}(t,x) = B_{SY}(t,x).
\ee
It is not difficult to check that $B_{SY,gen}$ satisfies a nonlinear ODE which converges to the one \eqref{Ec_SY} satisfied by $B_{SY}$ as $\al_1,\al_2\to 0$. Namely, the ODE satisfied is  (here $B_{SY,gen}\equiv B$)
\be\label{Ec_SY_alfaneq0}
\begin{aligned}
 & ~B_{(4x)} + 3B_x^2 \bar B + 4 |B|^2B_{xx} + 2|B_x|^2 B + B^2 \bar B_{xx}   + \frac32 |B|^4B \\
&\qquad  \qquad  - m_{SY,\al_1,\al_2}(B_{xx} +  |B|^2 B) +  n_{SY,\al_1,\al_2} B +i p_{SY,\al_1,\al_2}B_x\\
&\qquad  \qquad +i l_{SY,\al_1,\al_2}(B_{xxx} + 3|B|^2B_{x})=0,
\end{aligned}
\ee
(compare with \eqref{Ec_SY}) with coefficients 
\be\label{coeffa1a2}
\begin{aligned}
&m_{SY,\al_1,\al_2} :=c_2^2+c_1^2 + (\al_1+\al_2)^2 + 2\al_1\al_2,\quad n_{SY,\al_1,\al_2}:=(c_1^2+\al_1^2)(c_2^2+\al_2^2),\\
&\quad p_{SY,\al_1,\al_2}:=-2(c_2^2\al_1 + c_1^2\al_2 + \al_1\al_2(\al_1+\al_2)), \quad l_{SY,\al_1,\al_2}:=2(\al_1 + \al_2).
\end{aligned}
\ee
 This ODE is obtained by performing a variational characterization of $B_{SY,\al_1,\al_2}$ using the modified functional
\[
\begin{aligned}
\mathcal H_{X,\al_1,\al_2}[u] := &~{} F_{SY}[u] + m_{SY,\al_1,\al_2} E_{SY}[u] + n_{SY,\al_1,\al_2} M_{SY}[u] \\
&~{} + ip_{SY,\al_1,\al_2} P_{SY}[u] + il_{SY,\al_1,\al_2} L_{SY}[u],
\end{aligned}
\]
where $M_{SY},E_{SY},F_{SY}$ and $P_{SY}$ are given in \eqref{Mass_NLS0}, \eqref{Energy_NLS0}, \eqref{F_NLS0} and \eqref{Momentum_0}, respectively. The conserved
quantity  $L_{SY}$ (a second momentum) is defined as follows
\[
L_{SY}[u]:=\ima\frac{1}{2}\int_\R\Big(u_{xxx}\bar{u} - \frac{1}{2}u|u|^2\bar{u}_x + \bar{u}|u|^2u_x\Big).
\]
The proof of \eqref{Ec_SY_alfaneq0} is very much in the spirit of Theorem \ref{TH1a}. Details will be given elsewhere, but the interested reader may check \eqref{Ec_SY_alfaneq0} either by lengthy computations, or using a standard symbolic computing software.

\end{document}